\documentclass[12pt,english]{article}
\usepackage[T1]{fontenc}
\usepackage[latin9]{inputenc}
\usepackage{geometry}
\usepackage{babel}
\usepackage{verbatim}
\usepackage{mathrsfs}
\usepackage{amsmath}
\usepackage{amsthm}
\usepackage{amssymb}
\usepackage{graphicx}
\usepackage{esint}
\usepackage{xargs}[2008/03/08]
\usepackage{nomencl}
\providecommand{\printnomenclature}{\printglossary}
\providecommand{\makenomenclature}{\makeglossary}
\makenomenclature
\usepackage[unicode=true,pdfusetitle,
 bookmarks=true,bookmarksnumbered=false,bookmarksopen=false,
 breaklinks=false,pdfborder={0 0 1},backref=false,colorlinks=false]
 {hyperref}

\makeatletter
\numberwithin{equation}{section}
\theoremstyle{remark}
\newtheorem*{acknowledgement*}{\protect\acknowledgementname}
\theoremstyle{definition}
\newtheorem*{defn*}{\protect\definitionname}
\theoremstyle{plain}
\newtheorem{thm}{\protect\theoremname}[section]
\theoremstyle{definition}
\newtheorem{example}[thm]{\protect\examplename}
\theoremstyle{definition}
\newtheorem{defn}[thm]{\protect\definitionname}
\theoremstyle{plain}
\newtheorem{cor}[thm]{\protect\corollaryname}
\theoremstyle{plain}
\newtheorem{assumption}[thm]{\protect\assumptionname}
\theoremstyle{remark}
\newtheorem{notation}[thm]{\protect\notationname}
\theoremstyle{plain}
\newtheorem{lem}[thm]{\protect\lemmaname}
\theoremstyle{remark}
\newtheorem{rem}[thm]{\protect\remarkname}
\theoremstyle{remark}
\newtheorem*{rem*}{\protect\remarkname}

\usepackage{fancyhdr}

\makeatother

\providecommand{\acknowledgementname}{Acknowledgement}
\providecommand{\assumptionname}{Assumption}
\providecommand{\corollaryname}{Corollary}
\providecommand{\definitionname}{Definition}
\providecommand{\examplename}{Example}
\providecommand{\lemmaname}{Lemma}
\providecommand{\notationname}{Notation}
\providecommand{\remarkname}{Remark}
\providecommand{\theoremname}{Theorem}

\begin{document}
\global\long\def\divery{\mathrm{div}_{y}}%
\global\long\def\diver{\mathrm{div}}%
\global\long\def\diverx{\mathrm{div}_{x}}%

\global\long\def\nablay{\nabla_{y}}%
\global\long\def\nablax{\nabla_{x}}%
\global\long\def\nablabot{\nabla^{\bot}}%
\global\long\def\onablabot{\bar{\nabla}^{\bot}}%
\global\long\def\nablas{\nabla^{s}}%

\global\long\def\support{\mathrm{supp}}%

\global\long\def\epi{\mathrm{epi}}%
\global\long\def\short{\mathrm{short}}%
\global\long\def\dom{\mathrm{dom}}%
\global\long\def\Pois{\mathrm{Pois}}%
\global\long\def\mat{\mathrm{mat}}%
\global\long\def\REM{\mathrm{rem}}%
\global\long\def\NEW{\mathrm{new}}%
\global\long\def\pois{\mathrm{pois}}%
\global\long\def\Length{\mathrm{Length}}%

\global\long\def\rmD{\mathrm{D}}%
\global\long\def\rmM{\mathrm{M}}%
\global\long\def\rmP{\mathrm{P}}%
\global\long\def\rmR{\mathrm{R}}%
\global\long\def\rmT{\mathrm{T}}%

\global\long\def\id{\mathrm{id}}%

\global\long\def\eps{\varepsilon}%

\global\long\def\borelB{\mathcal{B}}%
\global\long\def\lebesgueL{\mathcal{L}}%
\global\long\def\hausdorffH{\mathcal{H}}%

\global\long\def\ue{u^{\eps}}%
\global\long\def\ve{v^{\eps}}%

\global\long\def\vel{\boldsymbol{\upsilon}}%
\global\long\def\v{\upsilon}%

\global\long\def\bB{\boldsymbol{\mathrm{B}}}%
\global\long\def\bE{\boldsymbol{\mathrm{E}}}%
\global\long\def\bL{\boldsymbol{\mathrm{L}}}%
\global\long\def\bQ{\boldsymbol{\mathrm{Q}}}%
\global\long\def\bT{\boldsymbol{\mathrm{T}}}%
\global\long\def\bU{\boldsymbol{\mathrm{U}}}%
\global\long\def\bP{\boldsymbol{\mathrm{P}}}%
\global\long\def\bW{\boldsymbol{\mathrm{W}}}%
\global\long\def\bY{\boldsymbol{\mathrm{Y}}}%
\global\long\def\boldnu{\boldsymbol{\nu}}%

\global\long\def\be{\boldsymbol{\mathrm{e}}}%
\global\long\def\bj{\boldsymbol{\mathrm{j}}}%

\global\long\def\A{\mathbb{A}}%
\global\long\def\B{\mathbb{B}}%
\global\long\def\D{\mathbb{D}}%
\global\long\def\E{\mathbb{E}}%
\global\long\def\G{\mathbb{G}}%
\global\long\def\I{\mathbb{I}}%
\global\long\def\M{\mathbb{M}}%
\global\long\def\N{\mathbb{N}}%
\global\long\def\O{\mathbb{O}}%
\global\long\def\P{\mathbb{P}}%
\global\long\def\Q{\mathbb{Q}}%
\global\long\def\R{\mathbb{R}}%
\global\long\def\S{\mathbb{S}}%
\global\long\def\T{\mathbb{T}}%
\global\long\def\U{\mathbb{U}}%
\global\long\def\X{\mathbb{X}}%
\global\long\def\Y{\mathbb{Y}}%
\global\long\def\Z{\mathbb{Z}}%

\global\long\def\Apaths{\A\!\X}%

\global\long\def\closedsets{\mathfrak{F}}%

\global\long\def\sB{\mathscr{B}}%
\global\long\def\sF{\mathscr{F}}%
\global\long\def\sI{\mathscr{I}}%
\global\long\def\sT{\mathscr{T}}%
\global\long\def\ttopology#1{\sT_{#1}}%

\global\long\def\fA{\mathfrak{A}}%
\global\long\def\fB{\mathfrak{B}}%
\global\long\def\fC{\mathfrak{C}}%
\global\long\def\fI{\mathfrak{I}}%
\global\long\def\fM{\mathfrak{M}}%
\global\long\def\fm{\mathfrak{m}}%
\global\long\def\fP{\mathfrak{P}}%
\global\long\def\fp{\mathfrak{p}}%
\global\long\def\fr{\mathfrak{r}}%
\global\long\def\fs{\mathfrak{s}}%
\global\long\def\ft{\mathfrak{t}}%
\global\long\def\fY{\mathfrak{Y}}%
\global\long\def\fa{\mathfrak{a}}%
\global\long\def\fb{\mathfrak{b}}%
\global\long\def\fd{\mathfrak{d}}%

\global\long\def\sfR{\mathsf{R}}%
\global\long\def\sfS{\mathsf{S}}%

\global\long\def\Rd{\mathbb{R}^{d}}%
\global\long\def\Rdd{\mathbb{R}^{d\times d}}%
\global\long\def\Zd{\mathbb{Z}^{d}}%
\global\long\def\I{\mathbb{I}}%

\global\long\def\d{\mathrm{d}}%
\global\long\def\v{\upsilon}%

\global\long\def\weakto{\rightharpoonup}%

\global\long\def\tsto{\stackrel{2s}{\to}}%
\global\long\def\tsweakto{\stackrel{2s}{\weakto}}%

\global\long\def\simple{\mathrm{simple}}%
\global\long\def\fl{\mathrm{flat}}%
\global\long\def\loc{\mathrm{loc}}%
\global\long\def\pot{\mathrm{pot}}%
\global\long\def\sol{\mathrm{sol}}%
\global\long\def\argmin{\mathrm{argmin}}%
\global\long\def\hull{\mathrm{hull}}%

\global\long\def\dist{\mathrm{dist}}%
\global\long\def\conv{\mathrm{conv}}%
\global\long\def\co{\mathfrak{co}}%

\global\long\def\boldY{\boldsymbol{\mathrm{Y}}}%

\global\long\def\mugammapalm{\mu_{\Gamma,\mathcal{P}}}%
\global\long\def\mupalm{\mu_{\mathcal{P}}}%
\global\long\def\nupalm{\nu_{\mathcal{P}}}%

\global\long\def\cA{\mathcal{A}}%
\global\long\def\cB{\mathcal{B}}%
\global\long\def\cC{\mathcal{C}}%
\global\long\def\cE{\mathcal{E}}%
\global\long\def\cF{\mathcal{F}}%
\global\long\def\cG{\mathcal{G}}%
\global\long\def\cH{\mathcal{H}}%
\global\long\def\cI{\mathcal{I}}%
\global\long\def\cL{\mathcal{L}}%
\global\long\def\cM{\mathcal{M}}%
\global\long\def\cN{\mathcal{N}}%
\global\long\def\cO{\mathcal{O}}%
\global\long\def\cP{\mathcal{P}}%
\global\long\def\cQ{\mathcal{Q}}%
\global\long\def\cR{\mathcal{R}}%
\global\long\def\cS{\mathcal{S}}%
\global\long\def\cT{\mathcal{T}}%
\global\long\def\cU{\mathcal{U}}%
\global\long\def\cV{\mathcal{V}}%
\global\long\def\cX{\mathcal{X}}%
\global\long\def\cY{\mathcal{Y}}%

\global\long\def\sE{\mathscr{E}}%
\global\long\def\sI{\mathscr{I}}%

\global\long\def\mapA{\mathcal{A}}%

\global\long\def\muomega{\mu_{\omega}}%
\global\long\def\muomegaeps{\mu_{\omega}^{\eps}}%
\global\long\def\mugammaomega{\mu_{\Gamma(\omega)}}%
\global\long\def\mugammaomegaeps{\mu_{\Gamma(\omega)}^{\eps}}%
\global\long\def\Ball#1#2{\mathbb{B}_{#1}{\left(#2\right)}}%
\newcommandx\Balldim[3][usedefault, addprefix=\global, 1=]{\mathbb{B}_{#2}^{#1}{\left(#3\right)}}%

\newcommandx\Balldimclosed[3][usedefault, addprefix=\global, 1=]{\overline{\mathbb{B}}_{#2}^{#1}{\left(#3\right)}}%
\global\long\def\cone{\mathbb{C}}%

\global\long\def\norm#1{\left\Vert #1\right\Vert }%

\global\long\def\scp#1#2{\left\langle #1,#2\right\rangle }%

\global\long\def\LOM{L^{2}(\Omega;\Rdd)}%
\global\long\def\LOMns{L_{n,s}^{2}(\Omega)}%
\global\long\def\LOMn{L_{n}^{2}(\Omega)}%
\global\long\def\e{\mathrm{e}}%
\global\long\def\diam{\mathrm{diam}}%

\global\long\def\of#1{{\left(#1\right)}}%

\title{Stochastic homogenization on perforated domains I - Extension operators}
\author{Martin Heida}
\maketitle
\begin{abstract}
In this first part of a series of 3 papers, we set up a framework
to study the existence of uniformly bounded extension and trace operators
for $W^{1,p}$-functions on randomly perforated domains, where the
geometry is assumed to be stationary ergodic. We drop the classical
assumption of minimaly smoothness and study stationary geometries
which have no global John regularity. For such geometries, uniform
extension operators can be defined only from $W^{1,p}$ to $W^{1,r}$
with the strict inequality $r<p$. In particular, we estimate the
$L^{r}$-norm of the extended gradient in terms of the $L^{p}$-norm
of the original gradient. Similar relations hold for the symmetric
gradients (for $\Rd$-valued functions) and for traces on the boundary.
As a byproduct we obtain some Poincaré and Korn inequalities of the
same spirit. 

Such extension and trace operators are important for compactness in
stochastic homogenization. In contrast to former approaches and results,
we use very weak assumptions: local $(\delta,M)$-regularity to quantify
statistically the local Lipschitz regularity and isotropic cone mixing
to quantify the density of the geometry and the mesoscopic properties.
These two properties are sufficient to reduce the problem of extension
operators to the connectivity of the geometry. 

In contrast to former approaches we do not require a minimal distance
between the inclusions and we allow for globally unbounded Lipschitz
constants and percolating holes. We will illustrate our method by
applying it to the Boolean model based on a Poisson point process
and to a Delaunay pipe process, for which we can explicitly estimate
the connectivity terms.
\end{abstract}
\begin{acknowledgement*}
I thank B. Jahnel for his helpful suggestions on literature. Furthermore,
I thank D.R.M. Renger and G. Friesecke for critical questions that
pushed my research further. Finally I thank the DFG for funding my
reasearch via CRC 1114 Project C05.
\end{acknowledgement*}
\tableofcontents{}

\section{Introduction}

In 1979 Papanicolaou and Varadhan \cite{papanicolaou1979boundary}
and Kozlov \cite{kozlov1979averaging} for the first time independently
introduced concepts for the averaging of random elliptic operators.
At that time, the periodic homogenization theory had already advanced
to some extend (as can be seen in the book \cite{papanicolau1978asymptotic}
that had appeared one year before) dealing also with non-uniformly
elliptic operators \cite{marcellini1978homogenization} and domains
with periodic holes \cite{cioranescu1979homogenization}. The most
recent and most complete work for extension operators on periodically
perforated domains is \cite{hopker2014diss}.

In contrast, the homogenization on randomly perforated domains is
still open to a large extend. Recent results focus on minimally smooth
domains \cite{guillen2015quasistatic,piatnitski2020homogenization}
or on decreasing size of the perforations when the smallness parameter
tends to zero \cite{giunti2021derivation} (and references therein).
The main issue in homogenization on perforated domains compared to
classical homogenization problems is compactness. For elasticity,
this is completely open.

The results presented below are meant for application in quenched
convergence. The estimates for the extension and trace operators which
are derived strongly depends on the realization of the geometry -
thus on $\omega$. Nevertheless, if the geometry is stationary, a
corresponding estimate can be achieved for almost every $\omega$.

\subsubsection*{The Problem}

In order to illustrate the issues in stochastic homogenization on
perforated domains, we introduce the following example.

Let $\bP(\omega)\subset\Rd$ be a stationary random open set and let
$\eps>0$ be the smallness parameter and let $\tilde{\bP}\of{\omega}$
be an infinitely connected component (i.e. an unbounded connected
domain) of $\bP\of{\omega}$. For a bounded open domain $\bQ$, we
consider $\bQ_{\tilde{\bP}}^{\eps}(\omega):=\bQ\cap\eps\tilde{\bP}(\omega)$
and $\Gamma^{\eps}(\omega):=\bQ\cap\eps\partial\tilde{\bP}(\omega)$
with outer normal $\nu_{\Gamma^{\eps}(\omega)}$. For a sufficiently
regular and $\Rd$-valued function $u$ we denote $\nablas u:=\frac{1}{2}\left(\nabla u+(\nabla u)^{\top}\right)$
the symmetric part of $\nabla u$. A typical homogenization problem
then is the following::
\begin{align}
-\diver\left(\left|\nablas\ue\right|^{p-2}\nablas\ue\right) & =g(\ue) &  & \mbox{on }\bQ_{\tilde{\bP}}^{\eps}(\omega)\,,\nonumber \\
u & =0 &  & \mbox{on }\partial\bQ\cap\left(\eps\bP\right)\,,\label{eq:system-eps-p-Laplace}\\
\left|\nablas\ue\right|^{p-2}\nabla\ue\cdot\nu_{\Gamma^{\eps}(\omega)} & =f(\ue) &  & \mbox{on }\Gamma^{\eps}(\omega)\,.\nonumber 
\end{align}
Note that for simplicity of illustration, the only randomness that
we consider in this problem is due to $\bP(\omega)$.

One way to prove homogenization of (\ref{eq:system-eps-p-Laplace})
is to prove $\Gamma$-convergence of 
\[
\cE_{\eps,\omega}(u)=\int_{\bQ_{\tilde{\bP}}^{\eps}(\omega)}\left(\frac{1}{p}\left|\nablas u\right|^{p}-G(u)\right)+\int_{\Gamma^{\eps}(\omega)}F(u)\,,
\]
in a suitably chosen space where $G'=g$ and $F'=f$. Conceptually,
this implies convergence of the minimizers $\ue$ to a minimizer of
a limit functional but if $G$ or $F$ are non-monotone, we need compactness.
However, the minimizers are elements of $\bW^{1,p}(\bQ_{\tilde{\bP}}^{\eps}):=W^{1,p}(\bQ_{\tilde{\bP}}^{\eps};\Rd)$
and since this space changes with $\eps$, there is apriori no compactness
of $\ue$, even though we have uniform apriori estimates on the gradients. 

The \emph{canonical }path to circumvent this issue in \emph{periodic
}homogenization is via uniformly bounded extension operators $\cU_{\eps}:\,W^{1,p}(\bQ_{\tilde{\bP}}^{\eps})\to W^{1,p}(\bQ)$
that share the property that for some $C>0$ independent from $\eps$
it holds for all $u\in W^{1,p}(\bQ_{\tilde{\bP}}^{\eps})$ with $u|_{\Rd\backslash\bQ}\equiv0$
\begin{equation}
\norm{\nabla\cU_{\eps}u}_{L^{p}(\bQ)}\leq C\norm{\nabla u}_{L^{p}(\bQ_{\tilde{\bP}}^{\eps})}\,,\qquad\norm{\cU_{\eps}u}_{L^{p}(\bQ)}\leq C\norm u_{L^{p}(\bQ_{\tilde{\bP}}^{\eps})}\,,\label{eq:uniform-extension-estimate}
\end{equation}
see \cite{hopker2014diss,hopker2014note}, combined with uniformly
bounded trace operators, see \cite{gahn2016homogenization,guillen2015quasistatic}.
Such operators have also been provided for elasticity problems \cite{hopker2014diss,oleinik2009mathematical,yosifian2001homogenization,yosifian2002some},
i.e.
\[
\norm{\nablas\cU_{\eps}u}_{L^{p}(\bQ)}\leq C\norm{\nablas u}_{L^{p}(\bQ_{\tilde{\bP}}^{\eps})}\,.
\]
The last estimate then allows to use Korn's inequality combined with
Sobolev's embedding theorem to find $\cU_{\eps}\ue\weakto u_{0}$
weakly in $\bW^{1,p}(\bQ)$.

\paragraph*{What is the classical strategy?}

The existing results on extension and trace operators for random domains
are focused on a.s. minimally smooth domains. A connected domain $\bP\subset\Rd$
is minimally smooth \cite{stein2016singular} if there exist $(\delta,M)$
such that for every $x\in\partial\bP$ the set $\partial\bP\cap\Ball{\delta}x$
is the graph of a Lipschitz continuous function with Lipschitz constant
less than $M$. It is further assumed that the complement $\Rd\backslash\bP$
consists of uniformly bounded sets. This concept leads to almost sure
construction of uniformly bounded extension operators $W_{\loc}^{1,p}(\bP)\to W^{1,p}(\Rd)$
\cite{guillen2015quasistatic} in the sense that for every bounded
$\bQ$ and every $u\in W^{1,p}(\bQ\cap\bP)$ with $u|_{\Rd\backslash\bQ}\equiv0$
holds 
\begin{equation}
\norm{\nabla\cU u}_{L^{p}(\bQ)}\leq C\norm{\nabla u}_{L^{p}(\bQ\cap\bP)}\,,\qquad\norm{\cU u}_{L^{p}(\bQ)}\leq C\norm u_{L^{p}(\bQ\cap\bP)}\,,\label{eq:uniform-extension-global}
\end{equation}
with $C$ independent from $\bQ$. Similarly, one obtains for the
trace $\cT$ that \cite{piatnitski2020homogenization} 
\[
\norm{\cT u}_{L^{p}(\bQ\cap\partial\bP)}\leq C\left(\norm u_{L^{p}(\bQ\cap\bP)}+\norm{\nabla u}_{L^{p}(\bQ\cap\bP)}\right)\,.
\]
Using a scaling argument to obtain e.g. (\ref{eq:uniform-extension-estimate}),
such extension and trace operators are typically used in order to
treat nonlinearities in homogenization problems.

\paragraph*{Why does this work?}

The theory cited above is directly connected to the theory of Jones
\cite{jones1981quasiconformal} and Duran and Muschietti \cite{duran2004korn}
on so-called John domains. These are precisely the bounded domains
$\bP$ that admit extension operators $W^{1,p}(\bP)\to W^{1,p}(\Rd)$
satisfying 
\[
\norm{\cU u}_{W^{1,p}(\Rd)}\leq C\norm u_{W^{1,p}(\bQ\cap\bP)}\,.
\]
 
\begin{defn*}[John domains]
A bounded domain $\bP\subset\Rd$ is a John domain (a.k.a $(\eps,\delta)$-domain)
if there exists $\eps,\delta>0$ such that for every $x,y\in\bP$
with $\left|x-y\right|<\delta$ there exists a rectifiable path $\gamma:\,[0,1]\to\bP$
from $x$ to $y$ such that 
\begin{align*}
\mathrm{length}\gamma & \leq\frac{1}{\eps}\left|x-y\right|\qquad\text{and}\\
\forall t\in(0,1):\quad\inf_{z\in\Rd\backslash\bP}\left|\gamma(t)-z\right| & \geq\frac{\eps\left|x-\gamma(t)\right|\left|\gamma(t)-y\right|}{\left|x-y\right|}\,.
\end{align*}
\end{defn*}
Because of the locality implied by $\delta$, it is possible to glue
together local extension operators on John domains such as done in
\cite{hopker2014diss} for periodic or \cite{guillen2015quasistatic}
for minimally smooth domains. In the stochastic case one benefits
a lot from the uniform boundedness of the components of $\Rd\backslash\bP$,
which allows to split the extension problem into independent extension
problems on uniformly John-regular domains.

\paragraph*{Why this is not enough for general random domains!}

As one could guess from the emphasis that is put on the above explanations,
random geometries are merely minimally smooth. On an unbounded random
domain $\bP$, the constant $M$ can locally become very large in
points $x\in\partial\bP$, while simultaneously, $\delta$ can become
very small in the very same $x$. In fact, they are not even ``uniformly
John'' as the following, yet deterministic example illustrates.
\begin{example}
Considering 
\[
\bP:=\left\{ \left(x_{1},x_{2}\right)\in\R^{2}:\;\exists n\in\N:\,x_{1}-(2n+1)\in(-1,1],\,x_{2}<\max\left\{ 1,\,n\left|x_{1}-(2n+1)\right|\right\} \right\} 
\]
the Lipschitz constant on $(2n,2n+2)$ is $n$ and it is easy to figure
out that this non-uniformly Lipschitz domain violates the John condition
due to the cups. Hence, a uniform estimate of the form (\ref{eq:uniform-extension-global})
cannot exist.
\end{example}

Therefor, an alternative concept to measure the large scale regularity
of a random geometry is needed. Since the classical results do not
excluded the existence of an estimate 
\begin{equation}
\frac{1}{\left|\bQ\right|}\int_{\bQ}\left|\nabla\cU u\right|^{r}\leq C\left(\frac{1}{\left|\bQ\right|}\int_{\bQ\cap\Ball{\fr}{\bP}}\left|\nabla u\right|^{p}\right)^{\frac{r}{p}}\,,\qquad\frac{1}{\left|\bQ\right|}\int_{\bQ}\left|\cU u\right|^{r}\leq C\left(\frac{1}{\left|\bQ\right|}\int_{\bQ\cap\Ball{\fr}{\bP}}\left|u\right|^{p}\right)^{\frac{r}{p}}\,,\label{eq:non-sc-extension-1}
\end{equation}
or
\begin{equation}
\frac{1}{\left|\bQ\right|}\int_{\bQ}\left|\nablas\cU u\right|^{r}\leq C\left(\frac{1}{\left|\bQ\right|}\int_{\bQ\cap\Ball{\fr}{\bP}}\left|\nablas u\right|^{p}\right)^{\frac{r}{p}}\,,\qquad\frac{1}{\left|\bQ\right|}\int_{\bQ}\left|\cU u\right|^{r}\leq C\left(\frac{1}{\left|\bQ\right|}\int_{\bQ\cap\Ball{\fr}{\bP}}\left|u\right|^{p}\right)^{\frac{r}{p}}\,,\label{eq:non-sc-extension-2}
\end{equation}
where $1\leq r<p$ and $C$ is independent from $\bQ$, such inequalities
will be our goal.

\paragraph*{Our results in a nutshell}

We will provide inequalities of the form (\ref{eq:non-sc-extension-1})--(\ref{eq:non-sc-extension-2})
for a Voronoi-pipe model and for a Boolean model. On the way, we will
provide several concepts and intermediate results that can be reused
in further examples and general considerations such as planed in part
III of this series. Scaled versions (replacing $\eps=m^{-1}$ in Theorems
\ref{thm:Pipes-Model} and \ref{thm:main-Boolean}) of (\ref{eq:non-sc-extension-1})--(\ref{eq:non-sc-extension-2})
can be formulated for functions 
\[
u\in W_{0,\partial\bQ}^{1,p}(\eps\bP\cap\bQ):=\left\{ u\in W^{1,p}(\bQ\cap\eps\bP):\,u|_{(\eps\bP)\cap\partial\bQ}\equiv0\right\} \,,
\]
and will be of the form
\[
\frac{1}{\left|\bQ\right|}\int_{\Rd}\left|\nabla\cU_{\eps}u\right|^{r}\leq C\left(\frac{1}{\left|\bQ\right|}\int_{\bQ\cap\eps\bP}\left|\nabla u\right|^{p}\right)^{\frac{r}{p}}\,,\qquad\frac{1}{\left|\bQ\right|}\int_{\Rd}\left|\cU_{\eps}u\right|^{r}\leq C\left(\frac{1}{\left|\bQ\right|}\int_{\bQ\cap\eps\bP}\left|u\right|^{p}\right)^{\frac{r}{p}}\,,
\]
resp.
\[
\frac{1}{\left|\bQ\right|}\int_{\Rd}\left|\nablas\cU_{\eps}u\right|^{r}\leq C\left(\frac{1}{\left|\bQ\right|}\int_{\bQ\cap\eps\bP}\left|\nablas u\right|^{p}\right)^{\frac{r}{p}}\,,\qquad\frac{1}{\left|\bQ\right|}\int_{\Rd}\left|\cU_{\eps}u\right|^{r}\leq C\left(\frac{1}{\left|\bQ\right|}\int_{\bQ\cap\eps\bP}\left|u\right|^{p}\right)^{\frac{r}{p}}\,,
\]
where the support of $\cU_{\eps}u$ lies within $\Ball{\eps^{\beta}}{\bQ}$
for $\eps$ small enough and some arbitrarily chosen but fixed $\beta\in(0,1)$. 

\subsubsection*{Quantifying properties of random geometries}

As a replacement for periodicity, we introduce the concept of mesoscopic
regularity of a stationary random open set:
\begin{defn}[Mesoscopic regularity]
\label{def:meso-regularity}\nomenclature[Mesoscopic regularity]{Mesoscopic regularity}{Definition \ref{def:meso-regularity}}Let
$\bP$ be a stationary ergodic random open set, let $\tilde{f}$ be
a positive, monotonically decreasing function $\tilde{f}$ with $\tilde{f}(R)\to0$
as $R\to\infty$ and let $\fr>0$ s.t.
\begin{equation}
\P\of{\exists x\in\Ball R0:\,\Ball{4\sqrt{d}\fr}x\subset\Ball R0\cap\bP}\geq1-\tilde{f}\of R\,.\label{eq:cri:stat-erg-ball}
\end{equation}
Then $\bP$ is called $(\fr,\tilde{f})$-mesoscopic regular. $\bP$
is called polynomially (exponentially) regular if $1/\tilde{f}$ grows
polynomially (exponentially).
\end{defn}

As a consequence of Lemmas \ref{lem:Alway-mesoscopic-regular}, \ref{lem:stat-erg-ball}
and \ref{lem:Iso-cone-geo-estimate} we obtain the following.
\begin{cor}[All stationary ergodic random open sets are mesoscopic regular]
\label{cor:Existence-X-r}Let $\bP(\omega)$ be a stationary ergodic
random open set. Then there exists $\fr>0$ and a monotonically decreasing
function with $\tilde{f}(R)\to0$ as $R\to\infty$ such that $\bP$
is $(\fr,\tilde{f})$-mesoscopic regular. Furthermore, there exists
a jointly stationary random point process $\X_{\fr}(\omega)=\left(x_{a}\right)_{a\in\N}$
and for every $a\in\N$ it holds $\Ball{\frac{\fr}{2}}{x_{a}}\subset\bP$
and for all $a,b\in\N$, $a\neq b$, it holds $\left|x_{a}-x_{b}\right|>2\fr$.
Construct from $\X_{\fr}$ a Voronoi tessellation of cells $G_{a}$
with diameter $d_{a}=d(x_{a})$. Then for some constant $C>0$ and
some monotone decreasing $f:\,(0,\infty)\to\R$ and $C>0$ with $f(R)\leq C\tilde{f}(C^{-1}R)$
it holds 
\[
\P\of{d(x_{a})>D}<f\of D\,.
\]
\end{cor}

$\fr$, $\X_{\fr}$ and $f$ from Corollary \ref{cor:Existence-X-r}
will play a central role in the analysis. We summarize some of these
properties in the following.
\begin{assumption}
\label{assu:Points-X}Let $\bP$ be a Lipschitz domain and assume
there exists $\X_{\fr}=\left(x_{a}\right)_{a\in\N}$ be a set of points
having mutual distance $\left|x_{a}-x_{b}\right|>2\fr$ if $a\neq b$
and with $\Ball{\frac{\fr}{2}}{x_{a}}\subset\bP$ for every $a\in\N$
(e.g. $\X_{\fr}(\bP)$, see (\ref{eq:def-X_r})).
\end{assumption}

The second important concept to quantify in a stochastic manner is
that of local Lipschitz regularity. 
\begin{defn}[Local $\left(\delta,M\right)$-Regularity]
\label{def:loc-del-M-reg}\nomenclature[delta]{$(\delta,M)$-regularity}{(Definition \ref{def:loc-del-M-reg})}Let
$\bP\subset\Rd$ be an open set. $\bP$ is called \emph{$(\delta,M)$-regular
}in $p_{0}\in\partial\bP$ if there exists an open set $U\subset\R^{d-1}$
and a Lipschitz continuous function $\phi:\,U\to\R$ with Lipschitz
constant greater or equal to $M$ such that $\partial\bP\cap\Ball{\delta}{p_{0}}$
is subset of the graph of the function $\varphi:\,U\to\Rd\,,\;\tilde{x}\mapsto\left(\tilde{x},\phi\of{\tilde{x}}\right)$
in some suitable coordinate system.
\end{defn}

Every Lipschitz domain $\bP$ is \emph{locally $(\delta,M)$-regular
}in every $p_{0}\in\partial\bP$. In what follows, we bound $\delta$
from above by $\fr$ only for practical reasons in the proofs. The
following quantities can be derived from local $(\delta,M)$-regularity.
\begin{defn}
For a Lipschitz domain $\bP\subset\Rd$ and for every $p\in\partial\bP$
and $n\in\N\cup\{0\}$
\begin{align}
\Delta{\left(p\right)} & :=\sup_{\delta<\fr}\left\{ \exists M>0:\,\bP\text{ is }\left(\delta,M\right)\text{-regular in }p\right\} \,,\quad\delta_{\Delta}{\left(p\right)}:=\frac{\Delta{\left(p\right)}}{2}\,,\label{eq:cover-delta-1}\\
M_{r}(p) & :=\inf_{\eta>r}\inf\left\{ M:\,\bP\text{ is }\left(\eta,M\right)\text{-regular in }p\right\} \,,\label{eq:cover-delta-M}\\
\rho_{n}{\left(p\right)} & :=\sup_{r<\delta\of p}r\left(4M_{r}\of p^{2}+2\right)^{-\frac{n}{2}}\,,\label{eq:cover-delta-2}
\end{align}
If no confusion occurs, we write $\delta=\delta_{\Delta}$. Furthermore,
for $c\in(0,1]$ let $\eta(p)=c\delta_{\Delta}(p)$ or $\eta(p)=c\rho_{n}(p)$,
$n\in\N$ and $r\in C^{0,1}(\partial\bP)$ and define\nomenclature[M eta]{$\tilde{M}_{\eta}(x)$}{Equation \eqref{eq:lem:local-delta-M-construction-estimate-1b}, a quantity on $\Rd$}
\begin{align}
\eta_{[r],\Rd}\of x & :=\inf\left\{ \eta\of{\tilde{x}}\,:\;\tilde{x}\in\partial\bP\,\text{s.t. }x\in\Ball{r\of{\tilde{x}}}{\tilde{x}}\right\} \,,\label{eq:lem:local-delta-M-construction-estimate-1}\\
M_{[r,\eta],\Rd}(x) & :=\sup\left\{ M_{r(\tilde{x})}(\tilde{x})\,:\;\tilde{x}\in\partial\bP\,\text{s.t. }x\in\overline{\Ball{\eta\of{\tilde{x}}}{\tilde{x}}}\right\} \,,\label{eq:lem:local-delta-M-construction-estimate-1b}
\end{align}
where $\inf\emptyset=\sup\emptyset:=0$ for notational convenience.
We also write $M_{[\eta],\Rd}(x):=M_{[\eta,\eta],\Rd}(x)$ and $\eta_{\Rd}\of x:=\eta_{[\eta],\Rd}\of x$.
Of course, we can also consider $M_{[r],\partial\bP}:\,p\mapsto M_{r(p)}(p)$
as a function on $\partial\bP$, and we will do this once in Lemma
\ref{lem:M-eta}.

When it comes to application of the abstract results found below,
it is important to have in mind that $\eta$ and $M_{r}$ are quantities
on $\partial\bP$, while $\eta_{[r],\Rd}$ and $M_{[r,\eta],\Rd}$
are quantities on $\Rd$. Hence, while trivially 
\[
\P(\eta_{[r],\Rd}\in(\eta_{1},\eta_{2}))=\lim_{n\to\infty}n^{-d}\left|\bQ\right|^{-1}\left|\left\{ x\in n\bQ:\,\eta_{[r],\Rd}\in(\eta_{1},\eta_{2})\right\} \right|
\]
(and similarly for $M_{[r,\eta],\Rd}$) for every convex bounded open
$\bQ$, we have in mind
\[
\P(\eta\in(\eta_{1},\eta_{2}))=\left(\lim_{n\to\infty}\hausdorffH^{d-1}\of{\partial\bP\cap n\bQ}\right)^{-1}\hausdorffH^{d-1}\of{\left\{ x\in(n\bQ)\cap\partial\bP:\,\eta\in(\eta_{1},\eta_{2})\right\} }\,.
\]

We will prove measurability of $\eta_{[r],\Rd}$ and $M_{[r,\eta],\Rd}$
in Lemma \ref{lem:delta-rho-M-measurable} and see how the weighted
expectations of $\eta_{[r],\Rd}$ and $M_{[r,\eta],\Rd}$ can be estimated
by weighted expectations of $M$ and $\eta$ in Lemma \ref{lem:delta-tilde-construction-estimate}. 
\end{defn}

\subsubsection*{Traces}

The first important result is the boundedness of the trace operator.
\begin{thm}
\label{thm:uniform-trace-estimate-1}Let $\bP\subset\Rd$ be a Lipschitz
domain, $\frac{1}{8}>\fr>0$ and let $\bQ\subset\Rd$ be a bounded
open set and let $1\leq r<p_{0}<p$. Then the trace operator $\cT$
satisfies for every $u\in W_{\loc}^{1,p}\of{\bP}$
\[
\frac{1}{\left|\bQ\right|}\int_{\bQ\cap\partial\bP}\left|\cT u\right|^{r}\leq C\left(\frac{1}{\left|\bQ\right|}\int_{\Ball{\frac{1}{4}}{\bQ}\cap\bP}\left|u\right|^{p}+\left|\nabla u\right|^{p}\right)^{\frac{r}{p}}
\]
where for some constant $C_{0}$ depending only on $p_{0}$, $p$
and $r$ and $d$ and for $\tilde{\rho}=2^{-5}\rho_{1}$ one has 
\begin{align}
C & =C_{0}\left(\frac{1}{\left|\bQ\right|}\int_{\Ball{\frac{1}{4}\fr}{\bQ}\cap\partial\bP}\tilde{\rho}_{\Rd}^{-\frac{1}{p_{0}-r}}\right)^{\frac{p_{0}-r}{p_{0}}}\left(\frac{1}{\left|\bQ\right|}\int_{\Ball{\frac{1}{4}\fr}{\bQ}\cap\bP}\left(1+\tilde{M}_{[\frac{1}{32}\delta],\Rd}\right)^{\left(\frac{1}{p_{0}}+1+\hat{d}\right)\frac{p}{p-p_{0}}}\right)^{\frac{p-p_{0}}{p_{0}p}}\,,\label{eq:lem:uniform-trace-estimate-1-1}\\
C & =C_{0}\left(\frac{1}{\left|\bQ\right|}\int_{\Ball{\frac{1}{4}\fr}{\bQ}\cap\partial\bP}\left(\tilde{\rho}_{\Rd}\left(1+\tilde{M}_{[\frac{1}{32}\delta],\Rd}\right)\right)^{-\frac{1}{p-r}}\right)^{\frac{p-r}{p}}\,.
\end{align}
\end{thm}

\begin{proof}
This is proved in Section \ref{subsec:Proof-traces}.
\end{proof}

\subsubsection*{Local Covering of $\partial\protect\bP$}

In view of Corollary \ref{cor:cover-boundary}, for every $n=1$ or
$n=2$ there exist a complete covering of $\partial\bP$ by balls
$\Ball{\tilde{\rho}_{n}\left(p_{i}^{n}\right)}{p_{i}^{n}}$, $\left(p_{i}^{n}\right)_{i\in\N}$,
where $\tilde{\rho}_{n}\of p:=2^{-5}\rho_{n}\of p$. We write $\tilde{\rho}_{n,i}:=\tilde{\rho}_{n}(p_{i}^{n})$. 
\begin{defn}[Microscopic regularity and extension order]
\label{assu:M-alpha-bound} The inner microscopic regularity $\alpha$
is 
\[
\alpha:=\inf\left\{ \tilde{\alpha}\geq0:\;\forall p\in\partial\bP\exists y\in\bP:\,\Ball{\tilde{\rho}(p)/32(1+M_{\tilde{\rho}(p)}(p)^{\tilde{\alpha}})}p\subset\Ball{\tilde{\rho}(p)/8}p\right\} \,.
\]
\end{defn}

In Lemma \ref{lem:small-ball-in-cone} we will see that indeed $\alpha\leq1$.
\begin{defn}[Extension order]
\label{def:extension-order} The geometry is \emph{of extension order}
$n\in\N\cup\{0\}$ if there exists $C>0$ such that for almost every
$p\in\partial\bP$ there exists a local extension operator 
\begin{align}
\cU:\,W^{1,p}(\Ball{\frac{1}{8}\delta(p)}p\cap\bP) & \to W^{1,p}(\Ball{\frac{1}{8}\rho_{n}(p)}p)\,,\nonumber \\
\norm{\nabla\cU u}_{L^{p}(\Ball{\frac{1}{8}\rho_{n}(p)}p)} & \leq C\left(1+M_{\frac{1}{8}\delta(p)}(p)\right)\norm{\nabla u}_{L^{p}(\Ball{\frac{1}{8}\delta(p)}p)}\,.\label{eq:ex-order-estimate-1}
\end{align}
The geometry is \emph{of symmetric extension order} $n\in\N\cup\{0\}$
if there exists $C>0$ such that for almost every $p\in\partial\bP$
there exists a local extension operator 
\begin{align}
\cU:\,\bW^{1,p}(\Ball{\frac{1}{8}\delta(p)}p\cap\bP) & \to\bW^{1,p}(\Ball{\frac{1}{8}\rho_{n}(p)}p)\,,\nonumber \\
\norm{\nablas\cU u}_{L^{p}(\Ball{\frac{1}{8}\rho_{n}(p)}p)} & \leq C\left(1+M_{\frac{1}{8}\delta(p)}(p)\right)^{2}\norm{\nablas u}_{L^{p}(\Ball{\frac{1}{8}\delta(p)}p)}\,.\label{eq:ex-order-estimate-2}
\end{align}
\end{defn}

Corollary \ref{cor:maximal-extension-order} shows that every locally
Lipschitz geometry is of extension order $n=1$ and every locally
Lipschitz geometry is of symmetric extension order $n=2$. However,
better results for $n$ are possible, as we will see below. 

\subsubsection*{Global Tessellation of $\protect\bP$}

Let $\X=\left(x_{a}\right)_{a\in\N}$ be a jointly stationary point
process with $\bP$ such that $\Ball{\fr}{\X}\subset\bP$. In this
work, we will often assume that $\left|x_{a}-x_{b}\right|>2\fr$ for
all $a\neq b$ for simplicity in Sections \ref{subsec:The-Issue-of-Connectedness}
and \ref{sec:Sample-Geometries}. The existence of such a process
is always guarantied by Lemmas \ref{lem:Alway-mesoscopic-regular}
and \ref{lem:stat-erg-ball}. Its choice in a concrete example is,
however, delicate. Worth mentioning, for most of the theory developed
until the end of Section \ref{sec:Extension-and-Trace-d-M} (Except
for Lemmas \ref{lem:Iso-cone-geo-estimate} and \ref{lem:estim-E-fa-fb}
which are not used before Section \ref{subsec:The-Issue-of-Connectedness}),
is completely independent from this mutual minimal distance assumption.

From $\X$ we construct a Voronoi tessellation with cells $\left(G_{a}\right)_{a\in\N}$
and we chose for each $x_{a}$ a radius $\fr_{a}\leq\fr$ with $\Ball{\fr_{a}}{x_{a}}\subset G_{a}\cap\bP$.
Again, using Corollary \ref{cor:Existence-X-r}, we assume that $\fr_{a}=\fr$
is constant for simplicity. 

\subsubsection*{Extensions I: Gradients}
\begin{notation}
\label{nota:extension-1}Given $n\in\{0,1\}$ and $\alpha\in[0,1]$
we chose 
\begin{equation}
\fr_{n,\alpha,i}:=\tilde{\rho}_{n,i}/32(1+M_{\tilde{\rho}_{n,i}}(p_{n,i})^{\alpha})\label{eq:def-fr-n-alph}
\end{equation}
 and some $y_{n,\alpha,i}$ such that 
\begin{equation}
B_{n,\alpha,i}:=\Ball{\fr_{n,\alpha,i}}{y_{n,\alpha,i}}\subset\bP\cap\Ball{\frac{1}{8}\tilde{\rho}_{n,i}}{p_{n,i}}\,.\label{eq:condi-fr-n-i}
\end{equation}
 and for every $i$ and $a$, we define 
\[
\tau_{n,\alpha,i}u:=\fint_{B_{n,\alpha,i}}u\,,\qquad\cM_{a}u:=\fint_{\Ball{\frac{\fr_{a}}{16}}{x_{a}}}u\,,
\]
local averages close to $\partial\bP$ and in $x_{a}$. We say that
$x_{a}\sim\sim x_{b}$ if $G_{a}\cap\Ball{\fr}{G_{b}}\neq\emptyset$
and we say $x_{a}\in\X_{\fr}(\bQ)$ if $\Ball{\fr}{G_{a}}\cap\bQ\neq\emptyset$.
Based on (\ref{eq:def-extension-op}) we obtain the following extension
result.
\end{notation}

\begin{thm}
\label{thm:Main-2}Let $\fr>0$ and let $\bP\subset\Rd$ be a stationary
ergodic random Lipschitz domain such that Assumption \ref{assu:Points-X}
holds for $\X=\left(x_{a}\right)_{a\in\N}$ and $\bP$ has microscopic
regularity $\alpha$ with extension order $n$. Let $\bQ\subset\Rd$
be a bounded open set with $\Ball{\frac{1}{4}}0\subset\bQ$ and let
$1\leq r<p$. Furthermore, let 
\[
\E\of{\left(\left(1+M_{[\frac{3\delta}{8},\frac{\delta}{8}],\Rd}\right)^{nd}\left(1+M_{[\frac{1}{8}\delta],\Rd}\right)^{r}\left(1+M_{[\tilde{\rho}_{n}],\Rd}\right)^{\alpha(d-1)}\right)^{\frac{p}{p-r}}}<\infty
\]
then there exist $C>0$ depending only on $d$, $r$ and $p$ such
that for a.e. $\omega$ there exists an extension operator $\cU_{\omega}:\,W_{\loc}^{1,p}(\bP(\omega))\to W_{\loc}^{1,p}(\Rd)$
and $C_{\omega}>0$ such that for every $m\geq1$ and every $u\in W^{1,p}(\bP(\omega))$
with $u|_{\bP(\omega)\backslash m\bQ}\equiv0$ it holds
\begin{align*}
\frac{1}{\left|m\bQ\right|}\int_{m\bQ}\left|\nabla\left(\cU_{\omega}u\right)\right|^{r} & \leq C_{\omega}\left(\frac{1}{m^{d}}\int_{\bP\cap\Ball{\fr}{m\bQ}}\left|\nabla u\right|^{p}\right)^{\frac{r}{p}}\\
 & +C\frac{1}{m^{d}}\int_{\bP\cap\Ball{\fr}{m\bQ}}\sum_{\substack{i\neq0}
}\sum_{a}\tilde{\rho}_{\bP}^{-r}\chi_{\Ball{\frac{\fr}{2}}{G_{a}}}\chi_{\Ball{\tilde{\rho}_{n,i}}{p_{n,i}}}\left|\tau_{n,\alpha,i}u-\cM_{a}u\right|^{r}\\
 & \quad+C\left|\frac{1}{m^{d}}\int_{\bP\cap m\bQ}\sum_{a}\sum_{a\sim\sim b}\chi_{\Ball{\fr}{G_{a}}}\left|\cM_{a}u-\cM_{b}u\right|\right|^{r}\,,\\
\frac{1}{\left|m\bQ\right|}\int_{m\bQ}\left|\cU_{\omega}u\right|^{r}\leq & C_{\omega}\left(\frac{1}{m^{d}}\int_{\bP\cap\Ball{\fr}{m\bQ}}\left|u\right|^{p}\right)^{\frac{r}{p}}\,.
\end{align*}
\end{thm}

\begin{proof}
This is a consequence of Lemma \ref{lem:local-delta-M-extension-estimate}. 
\end{proof}
In case one is interested in a weaker estimate on the extension operator,
we propose the following:
\begin{thm}
Under the assumptions of Theorem \ref{thm:Main-2} let additionally
\[
\E\of{\tilde{\rho}_{\bP}^{-\frac{rp}{p-r}}}<\infty
\]
then there exists an extension operator $\cU_{\omega}:\,W_{\loc}^{1,p}(\bP(\omega))\to W_{\loc}^{1,p}(\Rd)$
such that for every $m\geq1$ and every $u\in W^{1,p}(\bP(\omega))$
with $u|_{\bP(\omega)\backslash m\bQ}\equiv0$ it holds
\[
\frac{1}{\left|m\bQ\right|}\int_{m\bQ}\left(\left|\nabla\left(\cU_{\omega}u\right)\right|^{r}+\left|\cU_{\omega}u\right|^{r}\right)\leq C_{\omega}\left(\frac{1}{m^{d}}\int_{\bP\cap\Ball{\fr}{m\bQ}}\left(\left|\nabla u\right|^{p}+\left|u\right|^{p}\right)\right)^{\frac{r}{p}}\,.
\]
\end{thm}

\begin{proof}
This is a consequence of the proof of Lemma \ref{lem:local-delta-M-extension-estimate},
replacing $M_{a}u$ in the definition of $\cU u$ by $0$. 
\end{proof}

\subsubsection*{Percolation and Connectivity}

The terms depending on $\left|\tau_{n,\alpha,i}u-\cM_{a}u\right|$
or $\left|\cM_{a}u-\cM_{a}u\right|$ appearing on the right hand side
in Theorem \ref{thm:Main-2} need to be replaced by an integral over
$\left|\nabla u\right|^{p}$. Here, the pathwise topology of the geometry
comes into play. By this we mean that we have to integrate the gradient
of $u$ over a path connecting e.g. $p_{i}$ and $x_{a}$. Here, the
mesoscopic properties of the geometry will play a role. In particular,
we need pathwise connectedness of the random domain, a phenomenon
which is known as percolation in the theory of random sets. We will
discuss two different examples to see that these terms can indeed
be handled in application, but shift a general discussion of arbitrary
geometries to a later publication.

\subsubsection*{Extensions II: Symmetric gradients}

We now turn to the situation that $u$ is a $\Rd$-valued function
and that the given PDE system yields only estimates for $\nablas u=\frac{1}{2}\left(\nabla u+(\nabla u)^{T}\right)$.
We introduce the following quantities:
\begin{defn}
\label{def:cU-Q-sym}Given $n\in\{0,1,2\}$ and $\alpha\in[0,1]$
such that such that (\ref{eq:condi-fr-n-i}) holds for $\fr_{i}=\fr_{n,\alpha,i}$
for every $i$ let for $i,a$ 
\begin{align*}
\onablabot_{n,\alpha,i}u & :=\fint_{\Ball{\fr_{n,\alpha,i}}{y_{n,\alpha,i}}}\left(\nabla u-\nablas u\right)\,, & \left[\tau_{n,\alpha,i}^{\fs}u\right](x):= & \onablabot_{n,\alpha,i}u\,\left(x-y_{2,i}\right)+\fint_{\Ball{\fr_{n,\alpha,i}}{y_{n,\alpha,i}}}u\,,\\
\onablabot_{a}u & :=\fint_{\Ball{\frac{\fr_{a}}{16}}{x_{a}}}\left(\nabla u-\nablas u\right)\,, & \left[\cM_{a}^{\fs}u\right](x):= & \onablabot_{a}u\,\left(x-x_{a}\right)+\fint_{\Ball{\frac{\fr_{a}}{16}}{x_{a}}}u\,.
\end{align*}
\end{defn}

Using above introduced notation and $\bW$ do denote $\Rd$-valued
Sobolev spaces, we find the following.
\begin{thm}
\label{thm:Main-4}Let $\fr>0$ and let $\bP\subset\Rd$ be a stationary
ergodic random Lipschitz domain such that Assumption \ref{assu:Points-X}
holds for $\X=\left(x_{a}\right)_{a\in\N}$ and $\bP$ has microscopic
regularity $\alpha$ with symmetric extension order $n\leq2$. Let
$\bQ\subset\Rd$ be a bounded open set with $\Ball{\frac{1}{4}}0\subset\bQ$
and let $1\leq r<p_{0}<p$. Furthermore, let 
\[
\E\of{\left(\left(1+M_{[\frac{3\delta}{8},\frac{\delta}{8}],\Rd}\right)^{nd}\left(1+M_{[\frac{1}{8}\delta],\Rd}\right)^{2r}\left(1+M_{[\tilde{\rho}_{n}],\Rd}\right)^{\alpha(d-1)}\right)^{\frac{p}{p-r}}}<\infty
\]
then hen there exist $C>0$ depending only on $d$, $r$, $s$ and
$p$ such that for a.e. $\omega$ there exists an extension operator
$\cU_{\omega}:\,\bW_{\loc}^{1,p}(\bP(\omega))\to\bW_{\loc}^{1,p}(\Rd)$
and $C_{\omega}>0$ such that for every $m\geq1$ and every $u\in\bW^{1,p}(\bP(\omega))$
with $u|_{\bP(\omega)\backslash\bQ}\equiv0$ it holds
\begin{align*}
\frac{1}{\left|\bQ\right|}\int_{\bQ}\left|\nablas\left(\cU_{\omega}u\right)\right|^{r} & \leq C_{\omega}\left(\frac{1}{\left|\bQ\right|}\int_{\Ball{\fr}{\bQ}\cap\bP}\left|\nablas u\right|^{p}\right)^{\frac{r}{p}}\\
 & \quad+C\frac{1}{\left|\bQ\right|}\int_{\bQ\backslash\bP}\sum_{a}\sum_{i\neq0}\rho_{1,i}^{-r}\chi_{A_{1,i}}\chi_{\fA_{1,a}}\left|\tau_{n,\alpha,i}^{\fs}u-\cM_{a}^{\fs}u\right|^{r}\\
 & \quad+\frac{1}{\left|\bQ\right|}\int_{\bQ}\left|\sum_{l=1}^{d}\sum_{a:\,\partial_{l}\Phi_{a}>0}\sum_{b:\,\partial_{l}\Phi_{b}<0}\frac{\partial_{l}\Phi_{a}\left|\partial_{l}\Phi_{b}\right|}{D_{l+}^{\Phi}}\left(\cM_{a}^{\fs}u-\cM_{b}^{\fs}u\right)\right|^{r}\\
\frac{1}{\left|\bQ\right|}\int_{\bQ}\left|\cU_{\omega}u\right|^{r} & \leq C_{\omega}\left(\frac{1}{\left|\bQ\right|}\int_{\Ball{\fr}{\bQ}\cap\bP}\left|u\right|^{p}\right)^{\frac{r}{p}}\,,
\end{align*}
\end{thm}

\begin{proof}
This is a consequence of Lemma \ref{lem:local-delta-M-extension-estimate-sym}. 
\end{proof}
\begin{thm}
Under the assumptions of Theorem \ref{thm:Main-4} let additionally
\[
\E\of{\tilde{\rho}_{\bP}^{-\frac{rp}{p-r}}}<\infty
\]
then there exists an extension operator $\cU_{\omega}:\,W_{\loc}^{1,p}(\bP(\omega))\to W_{\loc}^{1,p}(\Rd)$
such that for every $m\geq1$ and every $u\in W^{1,p}(\bP(\omega))$
with $u|_{\bP(\omega)\backslash m\bQ}\equiv0$ it holds
\[
\frac{1}{\left|m\bQ\right|}\int_{m\bQ}\left(\left|\nabla\left(\cU_{\omega}u\right)\right|^{r}+\left|\cU_{\omega}u\right|^{r}\right)\leq C_{\omega}\left(\frac{1}{m^{d}}\int_{\bP\cap\Ball{\fr}{m\bQ}}\left(\left|\nabla u\right|^{p}+\left|u\right|^{p}\right)\right)^{\frac{r}{p}}\,.
\]
\end{thm}

\begin{proof}
This is a consequence of the proof of Lemma \ref{lem:local-delta-M-extension-estimate},
replacing $M_{a}^{\fs}u$ in the definition of $\cU u$ by $0$. 
\end{proof}

\subsection*{Discussion: Random Geometries and Applicability of the Method}

In Section \ref{sec:Sample-Geometries} we discuss two standard models
from the theory of stochastic geometries. The first one is a system
of random pipes: Starting from a Poisson point process and deleting
all points with nearest neighbor closer than $2\fr$ and introducing
the Delaunay neighboring condition on the points, every two neighbors
are connect through a pipe of random thickness $2\delta,$ where $\delta$
is distributed i.i.d among the pipes and we complete the geometry
by adding a ball of radius $\frac{\fr}{2}$ around each point. Defining
for bounded open domains $\bQ\subset\Rd$ and $n\in\N$
\[
u\in W_{0,\partial(n\bQ)}^{1,p}(\bP\cap n\bQ):=\left\{ u\in W^{1,p}(\bP\cap n\bQ):\,u|_{\partial(n\bQ)}\equiv0\right\} \,,
\]
and using $\bW$ instead of $W$ for $\Rd$-valued functions, we find
our first result:
\begin{thm}
\label{thm:Pipes-Model}In the pipe model of Section \ref{subsec:Matern-Process}
let $\P(\delta(x,y)<\delta_{0})\leq C_{\delta}\delta_{0}^{\beta}$
and let $1\leq r<s<p$ be such that $\max\left\{ \frac{p\left(s+d\right)}{p-s},\frac{p(2d-s-1)}{p-s}\right\} \leq\beta$
and $\frac{sr}{s-r}\leq\beta+d-1$. Then $\alpha=n=0$ both for extension
and symmetric extension order and there almost surely exists an extension
operator $\cU:\,W_{\loc}^{1,p}(\bP)\to W_{\loc}^{1,p}(\Rd)$ and constants
$C,R>1$ such that for all $m\in\N$ and every $u\in W_{0,\partial(m\bQ)}^{1,p}(\bP\cap m\bQ)$
it holds 
\[
\frac{1}{\left|m\bQ\right|}\int_{\Rd}\left|\nabla\left(\cU u\right)\right|^{r}\leq C\left(\frac{1}{m^{d}}\int_{\bP\cap m\bQ}\left|\nabla u\right|^{p}\right)^{\frac{r}{p}}\,.
\]
Furthermore there almost surely exists an extension operator $\cU:\,\bW_{\loc}^{1,p}(\bP)\to\bW_{\loc}^{1,p}(\Rd)$
and a constant $C>0$ such that for all $m\in\N$ and every $u\in\bW_{0,\partial(m\bQ)}^{1,p}(\bP\cap m\bQ)$
\[
\frac{1}{\left|m\bQ\right|}\int_{\Rd}\left|\nablas\left(\cU u\right)\right|^{r}\leq C\left(\frac{1}{m^{d}}\int_{\bP\cap m\bQ}\left|\nablas u\right|^{p}\right)^{\frac{r}{p}}\,.
\]
In both cases for every $\beta\in(0,1)$ the following holds: for
some $m_{0}>1$ depending on $\omega$ and every $m>m_{0}$ the support
of $\cU u$ lies within $\Ball{m^{1-\beta}}{m\bQ}$.
\end{thm}

\begin{proof}
The proof is given at the very end of Section \ref{subsec:Matern-Process}.
\end{proof}
\begin{cor}
If $\P(\delta(x,y)<\delta_{0})\leq C_{\delta}e^{-\gamma\delta_{0}^{-1}}$
then the last theorem holds for every $1\leq r<p$. 
\end{cor}

In Section \ref{subsec:Boolean-Model-for} we study the Boolean model
based on a Poisson point process in the percolation case. Introduced
in Example \ref{exa:poisson-point-proc} we will consider a Poisson
point process $\X_{\pois}(\omega)=\left(x_{i}(\omega)\right)_{i\in\N}$
with intensity $\lambda$ (recall Example \ref{exa:poisson-point-proc}).
To each point $x_{i}$ a random ball $B_{i}=\Ball 1{x_{i}}$ is assigned
and the family $\B:=\left(B_{i}\right)_{i\in\N}$ is called the Poisson
ball process. We say that $x_{i}\sim x_{j}$ if $\left|x_{i}-x_{j}\right|<2$.
In case $\lambda>\lambda_{c}$ the union of these balls has a unique
infinite connected component (that means we have percolation) and
we denote $\X_{\pois,\infty}$ the sellection of all points that contribute
to the infinite component and $\bP_{\infty}\left(\omega\right):=\bigcup_{i\in\X_{\pois,\infty}}B_{i}$
this infinite open set and seek for a corresponding uniform extension
operator. The connectedness of $\bP_{\infty}$ is hereby essential.
We use results from percolation theory that otherwise would not hold.

Here we can show that the micro- and mesoscopic assumptions are fulfilled,
at least in case $\bP$ is given as the union of balls. If we choose
$\bP$ as the complement of the balls, the situation becomes more
involved. On one hand, Theorem \ref{thm:boolean-delta-M-distrib}
shows that $\alpha$ and $n$ change in an unfortunate way. Furthermore,
the connectivity estimate remains open. However, some of these problems
might be overcome using a Matern modification of the Poisson process.
For the moment, we state the following.
\begin{thm}
\label{thm:main-Boolean}In the boolean model of Section \ref{subsec:Boolean-Model-for}
it holds $\alpha=0$ in case $\bP=\bP_{\infty}$ and both the extension
order and the symmetric extension order are $n=0$. If $d<p$ and
\[
\frac{pr}{p-r}<2,\quad r<d+2
\]
Then there almost surely exists an extension operator $\cU:\,W_{\loc}^{1,p}(\bP)\to W_{\loc}^{1,p}(\Rd)$
and a constant $C>0$ such that for all $m\in\N$ and every $u\in W_{0,\partial\bQ}^{1,p}(\bP\cap m\bQ)$
\[
\frac{1}{\left|m\bQ\right|}\int_{m\bQ}\left|\nabla\left(\cU u\right)\right|^{r}\leq C\left(\frac{1}{m^{d}}\int_{\bP\cap m\bQ}\left|\nabla u\right|^{p}\right)^{\frac{r}{p}}\,.
\]
If furthermore

\[
r<\frac{d+2}{2}
\]
then there almost surele exists an extension operator $\cU:\,\bW_{\loc}^{1,p}(\bP)\to\bW_{\loc}^{1,p}(\Rd)$
and a constant $C>0$ such that for all $m\in\N$ and every $u\in\bW_{0,\partial\bQ}^{1,p}(\bP\cap m\bQ)$
\[
\frac{1}{\left|m\bQ\right|}\int_{m\bQ}\left|\nablas\left(\cU u\right)\right|^{r}\leq C\left(\frac{1}{m^{d}}\int_{\bP\cap m\bQ}\left|\nablas u\right|^{p}\right)^{\frac{r}{p}}\,.
\]
In both cases for every $\beta\in(0,1)$ the following holds: for
some $m_{0}>1$ depending on $\omega$ and every $m>m_{0}$ the support
of $\cU u$ lies within $\Ball{m^{1-\beta}}{m\bQ}$.
\end{thm}

\begin{proof}
The proof is given at the very end of Section \ref{subsec:Boolean-Model-for}.
\end{proof}

\subsection*{Notes}

\subsubsection*{Structure of the article}

We close the introduction by providing an overview over the article
and its main contributions. In Section \ref{sec:Preliminaries} we
collect some basic concepts and inequalities from the theory of Sobolev
spaces, random geometries and discrete and continuous ergodic theory.
We furthermore establish local regularity properties for what we call
$\eta$-regular sets, as well as a related covering theorem in Section
\ref{subsec:Local--Regularity}. In Section \ref{subsec:Dynamical-Systems-on-Zd}
we will demonstrate that stationary ergodic random open sets induce
stationary processes on $\Zd$, a fact which is used later in the
construction of the mesoscopic Voronoi tessellation in Section \ref{subsec:Mesoscopic-Regularity}.

In Section \ref{sec:Nonlocal-regularity} we introduce the regularity
concepts of this work. More precisely, in Section \ref{subsec:Microscopic-Regularity}
we introduce the concept of local $\left(\delta,M\right)$-regularity
and use the theory of Section \ref{subsec:Local--Regularity} in order
to establish a local covering result for $\partial\bP$, which will
allow us to infer most of our extension and trace results. In Section
\ref{subsec:Mesoscopic-Regularity} we show how isotropic cone mixing
geometries allow us to construct a stationary Voronoi tessellation
of $\Rd$ such that all related quantities like ``diameter'' of
the cells are stationary variables whose expectation can be expressed
in terms of the isotropic cone mixing function $f$. Moreover we prove
the important integration Lemma \ref{lem:estim-E-fa-fb}.

In Sections \ref{sec:Extension-and-Trace-d-M}--\ref{subsec:The-Issue-of-Connectedness}
we finally provide the aforementioned extension operators and prove
estimates for these extension operators and for the trace operator.
In Section \ref{sec:Sample-Geometries} we study the sample geometries.

\subsubsection*{A Remark on Notation}

This article uses concepts from partial differential equations, measure
theory, probability theory and random geometry. Additionally, we introduce
concepts which we believe have not been introduced before. This makes
it difficult to introduce readable self contained notation (the most
important aspect being symbols used with different meaning) and enforces
the use of various different mathematical fonts. Therefore, we provide
an index of notation at the end of this work. As a rough orientation,
the reader may keep the following in mind:

We use the standard notation $\N$, $\Q$, $\R$, $\Z$ for natural
($>0$), rational, real and integer numbers. $\P$ denotes a probability
measure, $\E$ the expectation. Furthermore, we use special notation
for some geometrical objects, i.e. $\T^{d}=[0,1)^{d}$ for the torus
($\T$ equipped with the topology of the torus), $\I^{d}=(0,1)^{d}$
the open interval as a subset of $\Rd$ (we often omit the index $d$),
$\B$ a ball, $\cone$ a cone and $\X$ a set of points. In the context
of finite sets $A$, we write $\#A$ for the number of elements.

Bold large symbols ($\bU$, $\bQ$, $\bP$,$\dots$) refer to open
subsets of $\Rd$ or to closed subsets with $\partial\bP=\partial\mathring{\bP}$.
The Greek letter $\Gamma$ refers to a $d-1$ dimensional manifold
(aside from the notion of $\Gamma$-convergence).

Calligraphic symbols ($\mathcal{A}$, $\cU$, $\dots$) usually refer
to operators and large Gothic symbols ($\fB,\fC,\dots$) indicate
topological spaces, except for $\fA$.

\subsection*{Outlook}

This work is the first part of a triology. In part II, we will see
how to apply the extension and trace operators introduced above. 

In part III we will discuss general quantifyable properties of the
geometry that are eventually accessible also to computer algorithms
that will allow to predict homogenization behavior of random geometries.

\section{\label{sec:Preliminaries}Preliminaries}

We first collect some notation and mathematical concepts which will
be frequently used throughout this paper. We first start with the
standard geometric objects, which will be labeled by bold letters.

\subsection{Fundamental Notation and Geometric Objects}

Throughout this work, we use $\left(\be_{i}\right)_{i=1,\dots d}$
for the Euclidean basis of $\Rd$. By $C>0$ we denote any constant
that depends on $p$ and $d$ but no further dependencies unless explicitly
mentioned. Such mentioning may expressed in some cases through the
notation $C(a,b,\dots)$. Furthermore, we use the following notation.

\textbf{Unit cube~~~} The torus $\T=[0,1)^{d}$ is quipped with
the topology of the metric $d(x,y)=\min_{z\in\Zd}\left|x-y+z\right|$.
In contrast, the open interval $\I^{d}:=(0,1)^{d}$ is considered
as a subset of $\Rd$. We often omit the index $d$ if this does not
provoke confusion.\nomenclature[T]{$\T$}{$\T=[0,1)^d$ the torus (Section \ref{sec:Preliminaries})}\nomenclature[I]{$\I$}{$\I=[0,1)^d$ the torus (Section \ref{sec:Preliminaries})}

\textbf{Balls~~~} Given a metric space $\left(M,d\right)$ we denote
$\Ball rx$ \nomenclature[Ball]{$\Ball{r}{x}$}{Ball around $x$ with radius $r$ (Section \ref{sec:Preliminaries})}the
open ball around $x\in M$ with radius $r>0$. The surface of the
unit ball in $\Rd$ is $\S^{d-1}$. Furthermore, we denote for every
$A\subset\Rd$ by $\Ball rA:=\bigcup_{x\in A}\Ball rx$.

\textbf{Points~~~} A sequence of points will be labeled by $\X:=\left(x_{i}\right)_{i\in\N}$.\nomenclature[X]{$\X$, $Y$}{Families of points (Section \ref{sec:Preliminaries})}

\textbf{A cone~~~} in $\Rd$ is usually labeled by $\cone$. In
particular, we define for a vector $\nu$ of unit length, $0<\alpha<\frac{\pi}{2}$
and $R>0$ the \nomenclature[Cone]{$\cone_{\nu,\alpha,R}(x)$}{Cone with apix $x$, direction $\nu$, opening angle $\alpha$ and hight $R$  (Section \ref{sec:Preliminaries})}cone
\[
\cone_{\nu,\alpha,R}(x):=\left\{ z\in\Ball Rx\,:\;z\cdot\nu>\left|z\right|\cos\alpha\right\} \quad\text{and}\quad\cone_{\nu,\alpha}(x):=\cone_{\nu,\alpha,\infty}(x)\,.
\]
\textbf{Inner and outer hull~~~} We use balls of radius $r>0$
to define for a closed set $\bP\subset\Rd$ the sets \nomenclature[P]{$\bP_{r},\bP_{-r}$}{Inner and outer hull of $\bP$ with hight $r$ (Section \ref{sec:Preliminaries})}
\begin{equation}
\begin{aligned}\bP_{r} & :=\overline{\Ball r{\bP}}:=\left\{ x\in\Rd\,:\;\dist\left(x,\bP\right)\leq r\right\} \,,\\
\bP_{-r} & :=\Rd\backslash\left[\Ball r{\Rd\setminus\bP}\right]:=\left\{ x\in\Rd\,:\;\dist\left(x,\Rd\setminus\bP\right)\geq r\right\} \,.
\end{aligned}
\label{eq:Pr}
\end{equation}
One can consider these sets as inner and outer hulls of $\bP$. The
last definition resembles a concept of ``negative distance'' of
$x\in\bP$ to $\partial\bP$ and ``positive distance'' of $x\not\in\bP$
to $\partial\bP$. For $A\subset\Rd$ we denote $\conv(A)$ \nomenclature[convex]{$\conv A$}{Convex hull of $A$ (Section \ref{sec:Preliminaries})}the
closed convex hull of $A$. 

The natural geometric measures we use in this work are the Lebesgue
measure on $\Rd$, written $\left|A\right|$ for $A\subset\Rd$, and
the $k$-dimensional Hausdorff measure, denoted by $\cH^{k}$ on $k$-dimensional
submanifolds of $\Rd$ (for $k\leq d$).

\subsection{Simple Local Extensions and Traces}

In the following, we formulate some extension and trace results. Although
it is well known how such results are proved and the proofs are standard,
we include them for completeness since we are particularly interested
in the dependence of the operator norm on the local Lipschitz regularity
of the boundary.

The following is well known:
\begin{lem}
\label{lem:simple-extension}For every $1\leq p\leq\infty$ there
exists $C_{p}>0$ such that for every $R>0$ there exists an extension
operator $\cU:\,W^{1,p}(\Ball R0)\to W^{1,p}(\Ball{2R}0)$ such that
\[
\norm{\nabla\cU u}_{L^{p}(\Ball{2R}0)}\leq C_{p}\norm{\nabla u}_{L^{p}(\Ball R0)}\,.
\]
 
\end{lem}

Let $\bP\subset\Rd$ be an open set and let $p\in\partial\bP$ and
$\delta>0$ be a constant such that $\Ball{\delta}p\cap\partial\bP$
is graph of a Lipschitz function. We denote \nomenclature[Meta]{$M(p,\delta)$}{Lemma \eqref{eq:M(delta-p)}}
\begin{align}
M(p,\delta) & :=\inf\left\{ M:\,\,\exists\phi:U\subset\R^{d-1}\to\R\text{ }\right.\nonumber \\
 & \phantom{:=\inf}\left.\phi\,\text{Lipschitz, with constant }M\text{ s.t. }\Ball{\delta}p\cap\partial\bP\text{ is graph of }\phi\right\} \,.\label{eq:M(delta-p)}
\end{align}

\begin{rem}
For every $p$, the function $M(p,\cdot)$ is monotone increasing
in $\delta$.
\end{rem}

\begin{lem}[Uniform Extension for Balls]
\label{lem:uniform-extension-lemma} Let $\bP\subset\Rd$ be an open
set, $0\in\partial\bP$ and assume there exists $\delta>0$, $M>0$
and an open domain $U\subset\Ball{\delta}0\subset\R^{d-1}$ such that
$\partial\bP\cap\Ball{\delta}0$ is graph of a Lipschitz function
$\varphi:\,U\subset\R^{d-1}\to\Rd$ of the form $\varphi(\tilde{x})=(\tilde{x},\phi(\tilde{x}))$
in $\Ball{\delta}0$ with Lipschitz constant $M$ and $\varphi(0)=0$.
Writing $x=(\tilde{x},x_{d})$ and defining $\rho=\delta\sqrt{4M^{2}+2}^{-1}$
there exist an extension operator\nomenclature[Uc]{$\cU$}{local and global extension operators (Lemma \ref{lem:uniform-extension-lemma})}
\begin{equation}
\left(\cU u\right)(x)=\begin{cases}
u(x) & \mbox{ if }x_{d}<\phi(\tilde{x})\\
u\left(\tilde{x},{-}x_{d}+2\phi(\tilde{x})\right) & \mbox{ if }x_{d}>\phi(\tilde{x})
\end{cases}\,,\label{eq:lem:def-cU}
\end{equation}
such that for\nomenclature[Ac]{$\cA\left(0,\bP,\rho\right)$}{$\left\{ \left(\tilde{x},{-}x_{d}+2\phi(\tilde{x})\right)\,:\;\left(\tilde{x},x_{d}\right)\in\Ball{\rho}0\backslash\bP\right\} $ (Lemma \ref{lem:uniform-extension-lemma})}
\begin{equation}
\cA\left(0,\bP,\rho\right):=\left\{ \left(\tilde{x},{-}x_{d}+2\phi(\tilde{x})\right)\,:\;\left(\tilde{x},x_{d}\right)\in\Ball{\rho}0\backslash\bP\right\} \subset\Ball{\delta}0\,,\label{eq:lem:uniform-extension-lemma-A-rho}
\end{equation}
and for every $p\in[1,\infty]$ the operator 
\[
\cU:\,W^{1,p}\of{\cA\left(0,\bP,\rho\right)}\to W^{1,p}(\Ball{\rho}0)\,,
\]
is continuous with 
\begin{equation}
\left\Vert \cU u\right\Vert _{L^{p}(\Ball{\rho}0\backslash\bP)}\leq\left\Vert u\right\Vert _{L^{p}\left(\cA\left(0,\bP,\rho\right)\right)}\,,\qquad\left\Vert \nabla\cU u\right\Vert _{L^{p}(\Ball{\rho}0\backslash\bP)}\leq2M\left\Vert \nabla u\right\Vert _{L^{p}\left(\cA\left(0,\bP,\rho\right)\right)}\,.\label{eq:lem:uniform-extension-lemma-estim}
\end{equation}
\end{lem}

\begin{rem}
\label{rem:ext-order}In case $\phi(\tilde{x})\geq0$ we find $\cA\left(0,\bP,\rho\right)\subset\Ball{\rho}0$. 
\end{rem}

\begin{proof}[Proof of Lemma \ref{lem:uniform-extension-lemma}]
In case $\phi(\tilde{x})\equiv0$ we consider the extension operator
$\cU_{+}:\,W^{1,p}(\R^{d-1}\times({-}\infty,0))\to W^{1,p}(\Rd)$
having the form (compare also \cite[chapter 5]{Evans2010}, \cite{adams2003sobolev})
\[
\left(\cU_{+}u\right)(x)=\begin{cases}
u(x) & \mbox{ if }x_{d}<0\\
u\left(\tilde{x},{-}x_{d}\right) & \mbox{ if }x_{d}>0
\end{cases}\,.
\]
The general case follows from transformation. 
\end{proof}
\begin{lem}
\label{lem:basic-trace}Let $\bP\subset\Rd$ be an open set, $0\in\partial\bP$
and assume there exists $\delta>0$, $M>0$ and an open domain $U\subset\Ball{\delta}0\subset\R^{d-1}$
such that $\partial\bP\cap\Ball{\delta}0$ is graph of a Lipschitz
function $\varphi:\,U\subset\R^{d-1}\to\Rd$ of the form $\varphi(\tilde{x})=(\tilde{x},\phi(\tilde{x}))$
in $\Ball{\delta}0$ with Lipschitz constant $M$ and $\varphi(0)=0$
and define $\rho=\delta\sqrt{4M^{2}+2}^{-1}$. Writing $x=(\tilde{x},x_{d})$
we consider the trace operator $\cT:\,C^{1}\left(\bP\cap\Ball{\delta}0\right)\to C\left(\partial\bP\cap\Ball{\rho}0\right)$.
For every $p\in[1,\infty]$ and every $r<\frac{p\left(1-d\right)}{\left(p-d\right)}$
the operator $\cT$ can be continuously extended to 
\[
\cT:\,W^{1,p}\left(\bP\cap\Ball{\delta}0\right)\to L^{r}(\partial\bP\cap\Ball{\rho}0)\,,
\]
such that 
\begin{equation}
\norm{\cT u}_{L^{r}(\partial\bP\cap\Ball{\rho}0)}\leq C_{r,p}\rho^{\frac{d\left(p-r\right)}{rp}-\frac{1}{r}}\sqrt{4M^{2}+2}^{\frac{1}{r}+1}\norm u_{W^{1,p}\left(\bP\cap\Ball{\delta}0\right)}\,.\label{eq:lem:basic-trace}
\end{equation}
\end{lem}

\begin{proof}
We proceed similar to the proof of Lemma \ref{lem:uniform-extension-lemma}.

Step 1: Writing $B_{r}=\Ball r0$ together with $B_{r}^{-}=\left\{ x\in B_{r}:\,x_{d}<0\right\} $
and $\Sigma_{r}:=\left\{ x\in B_{r}:\,x_{d}=0\right\} $ we recall
the standard estimate
\[
\left(\int_{\Sigma_{1}}\left|u\right|^{r}\right)^{\frac{1}{r}}\leq C_{r,p}\left(\left(\int_{B_{1}^{-}}\left|\nabla u\right|^{p}\right)^{\frac{1}{p}}+\left(\int_{B_{1}^{-}}\left|u\right|^{p}\right)^{\frac{1}{p}}\right)\,,
\]
which leads to
\[
\left(\int_{\Sigma_{\rho}}\left|u\right|^{r}\right)^{\frac{1}{r}}\leq C_{r,p}\rho^{\frac{d\left(p-r\right)}{rp}-\frac{1}{r}}\left(\rho\left(\int_{B_{\rho}^{-}}\left|\nabla u\right|^{p}\right)^{\frac{1}{p}}+\left(\int_{B_{\rho}^{-}}\left|u\right|^{p}\right)^{\frac{1}{p}}\right)\,.
\]
Step 2: Using the transformation rule and the fact that $1\leq\left|\det D\varphi\right|\leq\sqrt{4M^{2}+2}$
we infer (\ref{eq:lem:basic-trace}) similar to Step 2 in the proof
of Lemma \ref{lem:uniform-extension-lemma}.
\begin{align*}
\left(\int_{\partial\bP\cap\Ball{\rho}0}\left|u\right|^{r}\right)^{\frac{1}{r}} & \leq\sqrt{4M^{2}+2}^{\frac{1}{r}}\left(\int_{\Sigma_{\rho}}\left|u\circ\varphi\right|^{r}\right)^{\frac{1}{r}}\\
 & \leq C_{r,p}\rho^{\frac{d\left(p-r\right)}{rp}-\frac{1}{r}}\sqrt{4M^{2}+2}^{\frac{1}{r}}\left(\rho\left(\int_{B_{\rho}^{-}}\left|\nabla\left(u\circ\varphi\right)\right|^{p}\right)^{\frac{1}{p}}+\left(\int_{B_{\rho}^{-}}\left|u\circ\varphi\right|^{p}\right)^{\frac{1}{p}}\right)\\
 & \leq C_{r,p}\rho^{\frac{d\left(p-r\right)}{rp}-\frac{1}{r}}\sqrt{4M^{2}+2}^{\frac{1}{r}+1}\,\cdot\\
 & \qquad\cdot\left(\rho\left(\int_{B_{\rho}^{-}}\left|\left(\nabla u\right)\circ\varphi\right|^{p}\det D\varphi\right)^{\frac{1}{p}}+\left(\int_{B_{\rho}^{-}}\left|u\circ\varphi\right|^{p}\det D\varphi\right)^{\frac{1}{p}}\right)
\end{align*}
and from this we conclude the Lemma with $\varphi^{-1}(B_{\rho}^{-})\subset\Ball{\delta}0$.
\end{proof}

\subsection{Local Nitsche-Extensions}

In this work, we will use bold letters for $\Rd$-valued function
spaces. In particular, we introduce for $1\leq p\leq\infty$ 
\begin{align*}
\bL^{p}(\bQ) & :=L^{p}(u;\Rd)\,,\\
\bW^{1,p}(\bQ) & :=\left\{ u\in\bL^{p}(\bQ)\,:\;\nabla u\in L^{p}\of{\bQ;\R^{d\times d}}\right\} \,.
\end{align*}

From \cite{duran2004korn} we know that on general Lipschitz domains
an estimate like the following holds:
\begin{lem}
\label{lem:uniform-korn-extension}For every $1\leq p\leq\infty$
there exists a constant $C>0$ depending only on the dimension $d\geq2$
such that the following holds: For every radius $R>0$ there exists
an extension operator $\cU_{R}:\,W^{1,p}\of{\Ball R0}\to W^{1,p}\of{\Ball{2R}0}$
such that 
\[
\norm{\nablas\left(\cU_{R}u\right)}_{W^{1,p}\of{\Ball{2R}0}}\leq C\norm{\nablas u}_{W^{1,p}\of{\Ball R0}}\,.
\]
\end{lem}

Again, we will need a refined estimate on extensions on Lipschitz
domains which explicitly accounts for the local Lipschitz constant. 
\begin{lem}[Uniform Nitsche-Extension for Balls]
\label{lem:uniform-Nitsche-extension-lemma} For every $d\geq2$
there exists a constant $C_{\cN}$ depending only on the dimension
$d$ such that the following holds: Let $\bP\subset\Rd$ be an open
set, $0\in\partial\bP$ and assume there exists $\delta>0$, $M>0$
and an open domain $U\subset\Ball{\delta}0\subset\R^{d-1}$ such that
$\partial\bP\cap\Ball{\delta}0$ is graph of a Lipschitz function
$\varphi:\,U\subset\R^{d-1}\to\Rd$ of the form $\varphi(\tilde{x})=(\tilde{x},\phi(\tilde{x}))$
in $\Ball{\delta}0$ with Lipschitz constant $M$ and $\varphi(0)=0$.
Writing $x=(\tilde{x},x_{d})$ and defining $\rho=\delta\sqrt{4M^{2}+2}^{-1}$
and 
\begin{equation}
\cA\left(0,\bP,\rho\right):=\left\{ \left(\tilde{x},x_{d}\right)\in\bP\,:\;\left|\tilde{x}\right|<\rho,\,x_{d}\leq C_{\cN}(1+M^{2})\right\} \,,\label{eq:lem:uniform-extension-lemma-A-rho-Nitsche}
\end{equation}
and for every $p\in[1,\infty]$ there exists a continuous operator
\[
\cU:\,W^{1,p}\of{\cA\left(0,\bP,\rho\right)}\to W^{1,p}(\Ball{\rho}0)\,,
\]
such that for some constant $C$ independent from $\left(\delta,M\right)$
and $\bP$ it holds 
\begin{equation}
\left\Vert \nablas\cU u\right\Vert _{L^{p}(\Ball{\rho}0\backslash\bP)}\leq C\left(1+M\right)^{2}\left\Vert \nablas u\right\Vert _{L^{p}\left(\cA\left(0,\bP,\rho\right)\right)}\,.\label{eq:Nitsche-Estimate}
\end{equation}
\end{lem}

\begin{rem}
\label{rem:ext-order-sym}In case $\phi(\tilde{x})\geq0$ the proof
reveals $\cA\left(0,\bP,\rho\right)\subset\Ball{c\rho}0$ for some
$c$ depending only on the dimension $d$.
\end{rem}

In order to prove such a result we need the following lemma.
\begin{lem}[\cite{stein2016singular} Chapter 6 Section 1 Theorem 2]
\label{lem:Stein-distance}There exist constants $c_{1},c_{2},c_{3}>0$
such that for every open set $\bP\subset\Rd$ with local Lipschitz
boundary there exists a function $d_{\bP}:\,\Rd\backslash\overline{\bP^{\complement}}\to\R$
with 
\[
\begin{aligned} & c_{1}d_{\bP}(x) & \leq\dist\of{x,\bP} & \leq c_{2}d_{\bP}(x)\,,\\
\forall i\in\left\{ 1,\dots,d\right\} : &  & \left|\partial_{i}d_{\bP}(x)\right| & \leq c_{3}\,,\\
\forall i,k\in\left\{ 1,\dots,d\right\} : &  & \left|\partial_{i}\partial_{k}d_{\bP}(x)\right| & \leq c_{3}\left|d_{\bP}(x)\right|^{-1}\,.
\end{aligned}
\]
\end{lem}

From the theory presented by Stein \cite{stein2016singular} we will
not get an explicit form of $C_{\cN}$ but only an upper bound that
grows exponentially with dimension $d$.
\begin{proof}[Proof of Lemma \ref{lem:uniform-Nitsche-extension-lemma}]
We use an idea by Nitsche \cite{nitsche1981korn}, which we transfer
from $p=2$ to the general case, thereby explicitly quantifying the
influence of $M$. For simplicity we write $\bP_{\delta}:=\bP\cap\Ball{\delta}0$
and $\bP_{\delta}^{\complement}:=\Ball{\delta}0\backslash\bP$ and
assume that $x\in\bP_{\delta}$ iff $x\in\Ball{\delta}0$ and $x_{d}<\phi(\tilde{x})$. 

As observed by Nitsche, it holds
\[
\forall x\in\bP_{\delta}^{\complement}:\quad0<\left(1+M^{2}\right)^{-\frac{1}{2}}\left(x_{d}-\phi(\tilde{x})\right)\leq\dist(x,\partial\bP)\leq x_{d}-\phi(\tilde{x})\,,
\]
and together with Lemma \ref{lem:Stein-distance}, we can define $d_{\bP,M}(x):=2c_{2}\left(1+M^{2}\right)^{\frac{1}{2}}d_{\bP}(x)$
and find for $c>\max\left\{ \frac{2c_{2}}{c_{1}},4c_{2}c_{3}\right\} $
that 
\[
\begin{aligned} &  & 2\left(x_{d}-\phi(\tilde{x})\right)\leq\d_{\bP,M}(x) & \leq c\left(1+M^{2}\right)^{\frac{1}{2}}\left(x_{d}-\phi(\tilde{x})\right)\,,\\
\forall i\in\left\{ 1,\dots,d\right\} : &  & \left|\partial_{i}d_{\bP,M}(x)\right| & \leq c\left(1+M^{2}\right)^{\frac{1}{2}}\,,\\
\forall i,k\in\left\{ 1,\dots,d\right\} : &  & \left|\partial_{i}\partial_{k}d_{\bP,M}(x)\right| & \leq c\left(1+M^{2}\right)\left|d_{\bP,M}(x)\right|^{-1}\,.
\end{aligned}
\]
If $\psi\in C([1,2])$ satisfies 
\begin{equation}
\int_{1}^{2}\psi(t)\,\d t=1\,,\qquad\int_{1}^{2}t\,\psi(t)\,\d t=0.\label{eq:lem:uniform-extension-lemma-help-1}
\end{equation}
Nitsche introduced $x_{\lambda}:=(\tilde{x},x_{d}-\lambda d_{\bP,M}(x))$
and proposed the following extension on $x\in\bP_{\delta}^{\complement}$:
\[
u_{i}(x):=\int_{1}^{2}\psi(\lambda)\left(u_{i}(x_{\lambda})+\lambda u_{d}(x_{\lambda})\,\partial_{i}d_{\bP,M}(x)\right)\d\lambda\,.
\]
One can quickly verify that this maps $C(\overline{\bP_{\delta}})$
onto $C\of{\overline{\Ball{\rho}0}}$. In what follows, we write $\eps[u](x):=\nabla^{s}u(x)$
and particularly $\eps_{ij}[u](x):=\frac{1}{2}\left(\partial_{i}u_{j}+\partial_{j}u_{i}\right)$
as well as $\eps_{ij}^{\lambda}[u](x)=\eps_{ij}[u](x_{\lambda})$
for $x\in\bP_{\delta}^{\complement}$. Then for $x\in\bP_{\delta}^{\complement}\cap\Ball{\rho}0$
\begin{align}
\eps_{ij}[u](x)=\int_{1}^{2}\psi(\lambda) & \left(\eps_{ij}^{\lambda}(x)+\lambda\partial_{i}d_{\bP,M}(x)\,\eps_{jd}^{\lambda}(x)+\lambda\partial_{j}d_{\bP,M}(x)\,\eps_{id}^{\lambda}(x)\right.\label{eq:lem:uniform-extension-lemma-help-2}\\
 & \;\left.+\lambda^{2}\partial_{i}d_{\bP,M}(x)\,\partial_{j}d_{\bP,M}(x)\,\eps_{dd}^{\lambda}(x)+\lambda\partial_{i}\partial_{j}d_{\bP,M}(x)\,u_{d}(x_{\lambda})\right)
\end{align}
From the fundamental theorem of calculus we find 
\[
u_{d}(x_{\lambda})=u_{d}(x_{1})+\delta(\tilde{x})\int_{1}^{\lambda}\partial_{d}u_{d}(x_{t})\,\d t\,,
\]
which leads by (\ref{eq:lem:uniform-extension-lemma-help-1}) to 
\[
\int_{1}^{2}\psi(\lambda)\lambda\partial_{i}\partial_{j}d_{\bP,M}(x)\,u_{d}(x_{\lambda})\,\d\lambda=\partial_{i}\partial_{j}d_{\bP,M}(x)\,d_{\bP,M}(x)\int_{1}^{2}\eps_{dd}[u](x_{t})\,\d t\int_{\mu}^{2}\psi(\lambda)\lambda\,\d\lambda\,.
\]
We may now apply $\left|\,\cdot\,\right|^{p}$ on both sides of (\ref{eq:lem:uniform-extension-lemma-help-2}),
integrate over $\bP_{\delta}^{\complement}\cap\Ball{\rho}0$ and use
the integral transformation theorem for each $\lambda$ to find 
\[
\norm{\eps[u]}_{L^{p}\of{\bP_{\delta}^{\complement}\cap\Ball{\rho}0}}\leq C\left(1+M^{2}\right)\norm{\eps[u]}_{L^{p}\of{\bP_{\delta}}}\,.
\]
\end{proof}

\subsection{\label{subsec:Poincar=0000E9-Inequalities}Poincaré Inequalities}

We denote for bounded open $A\subset\Rd$
\[
W_{(0),r}^{1,p}(A):=\left\{ u\in W^{1,p}(A)\,:\;\exists x:\,B_{r}(x)\subset A\;\vee\;\fint_{B_{r}(x)}u=0\right\} \,.
\]
Note that this is not a linear vector space.
\begin{lem}
\label{lem:Poincare-ball}For every $p\in[1,\infty)$ there exists
$C_{p}>0$ such that the following holds: Let $0<r<R$ and $x\in\Ball R0$
such that $\Ball rx\subset\Ball R0$ then for every $u\in W^{1,p}(\Ball R0)$
\begin{equation}
\left\Vert u\right\Vert _{L^{p}(\Ball R0)}^{p}\leq C_{p}\left(R^{p}\frac{R^{d-1}}{r^{d-1}}\left\Vert \nabla u\right\Vert _{L^{p}(\Ball R0)}^{p}+\frac{R^{d}}{r^{d}}\left\Vert u\right\Vert _{L^{p}(\Ball rx)}^{p}\right)\,,\label{eq:lem:Poincare-ball-1}
\end{equation}
and for every $u\in W_{(0),r}^{1,p}((\Ball R0)$ it holds 
\begin{equation}
\left\Vert u\right\Vert _{L^{p}(\Ball R0)}^{p}\leq C_{p}R^{p}\left(\frac{r}{R}\right)^{1-d}\left(1+\left(\frac{r}{R}\right)^{p-1}\right)\left\Vert \nabla u\right\Vert _{L^{p}(\Ball R0)}^{p}\,.\label{eq:lem:Poincare-ball-2}
\end{equation}
\end{lem}

\begin{rem*}
In case $p\geq d$ we find that (\ref{eq:lem:Poincare-ball-2}) holds
iff $u(x)=0$ for some $x\in\Ball 10$.
\end{rem*}
\begin{proof}
 In a first step, we assume $x=0$ and $R=1$. The underlying idea
of the proof is to compare every $u(y)$, $y\in\Ball 10\backslash\Ball r0$
with $u(rx)$. In particular, we obtain for $y\in\Ball 10\backslash\Ball r0$
that 
\[
u(y)=u(ry)+\int_{0}^{1}\nabla u\of{ry+t(1-r)y}\cdot(1-r)y\,\d t
\]
and hence by Jensen's inequality 
\[
\left|u(y)\right|^{p}\leq C\left(\int_{0}^{1}\left|\nabla u\of{ry+t(1-r)y}\right|^{p}(1-r)^{p}\left|y\right|^{p}\,\d t+\left|u(ry)\right|^{p}\right)\,.
\]
We integrate the last expression over $\Ball 10\backslash\Ball r0$
and find 
\begin{align*}
\int_{\Ball 10\backslash\Ball r0}\left|u(y)\right|^{p}\d y & \leq\int_{S^{d-1}}\int_{r}^{1}C\left(\int_{0}^{1}\left|\nabla u\of{rs\nu+t(1-r)s\nu}\right|^{p}(1-r)^{p}s^{p}\,\d t\right)s^{d-1}\d s\d\nu\\
 & \quad+\int_{\Ball 10\backslash\Ball r0}\left|u(ry)\right|^{p}\d y\\
 & \leq\int_{S^{d-1}}\int_{r}^{1}C\left(\int_{rs}^{s}\left|\nabla u\of{t\nu}\right|^{p}(1-r)^{p-1}s^{p-1}\,\d t\right)s^{d-1}\d s\\
 & \quad+\int_{\Ball 10\backslash\Ball r0}\left|u(ry)\right|^{p}\d y\\
 & \leq C\int_{r}^{1}\d s\;s^{d-1}\frac{1}{\left(rs\right)^{d-1}}\int_{rs}^{s}\d t\;t^{d-1}\int_{S^{d-1}}\left|\nabla u\of{t\nu}\right|^{p}(1-r)^{p-1}s^{p-1}\\
 & \quad+\int_{\Ball 10\backslash\Ball r0}\left|u(ry)\right|^{p}\d y\\
 & \leq C\frac{1}{r^{d-1}}\left\Vert \nabla u\right\Vert _{L^{p}(\Ball 10)}^{p}+\frac{1}{r^{d}}\left\Vert u\right\Vert _{L^{p}(\Ball r0)}^{p}\,.
\end{align*}
For general $x\in\Ball 10$, use the extension operator $\cU:\,W^{1,p}(\Ball 10)\to W^{1,p}(B_{4}(0))$
such that $\left\Vert \cU u\right\Vert _{W^{1,p}(B_{4}(0))}\leq C\norm u_{W^{1,p}(\Ball 10)}$
and $\left\Vert \nabla\cU u\right\Vert _{W^{1,p}(B_{4}(0))}\leq C\norm{\nabla u}_{W^{1,p}(\Ball 10)}$.
Since $\Ball 10\subset B_{2}(x)\subset B_{4}(0)$ we infer 
\[
\left\Vert u\right\Vert _{L^{p}(\Ball 10)}^{p}\leq\left\Vert \cU u\right\Vert _{L^{p}(B_{2}(x))}^{p}\leq C\left(\frac{1}{r^{d-1}}\left\Vert \nabla\cU u\right\Vert _{L^{p}(B_{2}(x))}^{p}+\frac{1}{r^{d}}\left\Vert \cU u\right\Vert _{L^{p}(B_{r}(x))}^{p}\right)\,.
\]
and hence (\ref{eq:lem:Poincare-ball-1}). Furthermore, since there
holds $\left\Vert u\right\Vert _{L^{p}(\Ball 10)}^{p}\leq C\left\Vert \nabla u\right\Vert _{L^{p}(\Ball 10)}^{p}$
for every $u\in W_{(0)}^{1,p}(\Ball 10)$, a scaling argument shows
$\left\Vert u\right\Vert _{L^{p}(\Ball r0)}^{p}\leq Cr^{p}\left\Vert \nabla u\right\Vert _{L^{p}(\Ball r0)}^{p}$
for every $u\in W_{(0),r}^{1,p}(\Ball 10)$ and hence (\ref{eq:lem:Poincare-ball-2}).
For general $R>0$ use a scaling argument.
\end{proof}
A similar argument leads to the following, where we remark that the
difference in the appearing of $\frac{1}{r}$ is due to the fact,
that integrating the cylinder needs no surface element $r^{d-1}$.
\begin{cor}
\label{cor:Poincare-zylinder}For every $p\in[1,\infty)$ and $r>0$
there exists $C_{p}>0$ such that the following holds: Let $r<L$,
$P_{L,r}:=\Balldim[d-1]r0\times(0,L)$ and $x\in P_{L,r}$ such that
$\Ball rx\subset P_{L,r}$ then for every $u\in W^{1,p}(P_{L,r})$
\begin{equation}
\left\Vert u\right\Vert _{L^{p}(P_{L,r})}^{p}\leq C_{p}\left(L^{p}\left\Vert \nabla u\right\Vert _{L^{p}(P_{L,r})}^{p}+\frac{L}{r}\left\Vert u\right\Vert _{L^{p}(\Ball rx)}^{p}\right)\,,\label{eq:cor:Poincare-zylinder}
\end{equation}
and if additionally $\fint_{\Ball rx}u=0$ then 
\begin{equation}
\left\Vert u\right\Vert _{L^{p}(P_{L,r})}^{p}\leq C_{p}\left(L^{p}\left\Vert \nabla u\right\Vert _{L^{p}(P_{L,r})}^{p}+Lr^{p-1}\left\Vert \nabla u\right\Vert _{L^{p}(\Ball rx)}^{p}\right)\,,\label{eq:cor:Poincare-zylinder-1}
\end{equation}
Let $y\in P_{L,r}$ such that $\Ball ry\subset P_{L,r}$ then for
every $u\in W^{1,p}(P_{L,r})$ 
\begin{equation}
\left|\fint_{\Ball ry}u-\fint_{\Ball rx}u\right|^{p}\leq C_{p}\left(L^{p-1}r^{1-d}\left\Vert \nabla u\right\Vert _{L^{p}(P_{L,r})}^{p}\right)\,.\label{eq:cor:Poincare-zylinder-2}
\end{equation}
\end{cor}

\subsection{Korn Inequalities}

We introduce on open sets $A\subset\Rd$ the Sobolev space 
\[
\bW_{\nablabot(0)}^{1,p}\of A:=\left\{ u\in\bW^{1,p}\of A:\;\forall i,j:\,\int_{A}\partial_{i}u_{j}-\partial_{j}u_{i}=0\right\} \,.
\]
To the authors best knowledge, the following is the most general Korn
inequality in literature.
\begin{thm}[\cite{duran2004korn} Theorem 2.7 and Corollary 2.8]
\label{thm:Duran-Korn}Let $1\leq p\leq\infty$ and $\eps\in(0,1)$
and $\tilde{\delta}>0$. Then there exists a constant $C_{p}>0$ depending
only on $d$, $p$, $\eps$ and $\tilde{\delta}$ such that for every
bounded open set $A\in\Rd$ with $\delta>0$ such that $\delta/\diam A\geq\tilde{\delta}$
and with the property 
\begin{equation}
\left.\begin{aligned} & \forall x,y\in A,\,\left|x-y\right|<\delta:\quad\exists\gamma\in C^{1}([0,1];A),\,\gamma(0)=x,\,\gamma(1)=y\text{ such that:}\\
 & \qquad l(\gamma)\leq\frac{1}{\eps}\left|x-y\right|\quad\text{and}\quad\forall t\in(0,1):\,\dist\of{\gamma(t),\partial A}\geq\frac{\eps\left|x-\gamma(t)\right|\left|y-\gamma(t)\right|}{x-y}
\end{aligned}
\right\} \label{eq:Jones-domain}
\end{equation}
it holds 
\begin{equation}
\forall u\in\bW_{\nablabot(0)}^{1,p}\of A:\,\norm{\nabla u}_{L^{p}\of A}\leq C_{p}\norm{\nabla^{s}u}_{L^{p}\of A}\,.\label{eq:simple-Korn}
\end{equation}
\end{thm}

\begin{rem}
In the original work the claimed dependence of $C_{p}$ was on $d$,
$p$, $\eps$, $\delta$ and $A$ with the observation that (\ref{eq:simple-Korn})
is invariant under scaling of $A$. However, this scale invariance
results in the dependence on $d$, $p$, $\eps$ and $\delta/\diam A$
since $\eps$, $p$ and $d$ are not sensitive to scaling of $A$. 
\end{rem}

\begin{defn}
Domains $A\subset\Rd$ satisfying (\ref{eq:Jones-domain}) for some
$\eps\in(0,1)$ and $\delta>0$ are called $(\eps,\delta)$-John domains
or simply John domains.
\end{defn}

\begin{cor}
For every $1\leq p\leq\infty$ there exists $C_{p}$ depending only
on $d$ and $p$ such that for every bounded open convex set $A\subset\Rd$
the estimate (\ref{eq:simple-Korn}) holds.
\end{cor}

We furthermore introduce the set 
\[
\bW_{\nablabot(0),r}^{1,p}(A):=\left\{ u\in\bW^{1,p}(A)\,:\;\exists x:\,\Ball rx\subset A\;\vee\;\forall i,j:\,\int_{\Ball rx}\partial_{i}u_{j}-\partial_{j}u_{i}=0\right\} 
\]
which is \textbf{not} a vector space. 
\begin{lem}[Mixed Korn inequality]
\label{lem:General-Korn}Let $1\leq p\leq\infty$ and $\eps,\delta\in(0,1)$.
Then there exists a constant $\tilde{C}_{p}>0$ depending only on
$d$, $p$, $\eps$ and $\delta$ such that for every $(\eps,\delta)$-John
domain $A\subset\Ball 10$ and for every $r\in(0,1)$ and every $x\in A$
with $\Ball rx\subset A$ it holds 
\begin{equation}
\forall u\in\bW^{1,p}\of A:\,\norm{\nabla u}_{L^{p}\of A}\leq\tilde{C}_{p}\left(\frac{\left|A\right|}{r^{d}}\right)^{\frac{1}{p}}\left(\norm{\nabla^{s}u}_{L^{p}\of A}+\norm{\nabla u}_{L^{p}\of{\Ball rx}}\right)\,.\label{eq:Scaled-Korn-inequality}
\end{equation}
Furthermore, 
\begin{equation}
\forall u\in\bW_{\nablabot(0),r}^{1,p}(A):\,\norm{\nabla u}_{L^{p}\of A}\leq\tilde{C}_{p}\left(\frac{\left|A\right|}{\left|\S^{d-1}\right|r^{d}}\right)^{\frac{1}{p}}\left(\norm{\nabla^{s}u}_{L^{p}\of A}\right)\,.\label{eq:Scaled-Korn-inequality-1}
\end{equation}
\end{lem}

Unfortunately, we do not have a reference for a comparable Lemma in
the literature except for \cite{schweizer2013partielle} in case $p=2$.
The author strongly supposes a proof must exist somewhere, however,
we provide it for completeness.
\begin{proof}
Let $C_{p}$ be the constant from Theorem \ref{thm:Duran-Korn} for
domains with a diameter less than $2$ and suppose (\ref{eq:Scaled-Korn-inequality})
was wrong. Then there exists a sequence of $(\eps,\delta)$-John domains
$A_{n}\subset\Ball 10$ with $x_{n}\in A_{n}$, $r_{n}\in(0,1)$ with
$\Ball{r_{n}}{x_{n}}\subset A_{n}$ and functions $u_{n}\in\bW^{1,p}\of{A_{n}}$
such that 
\[
1=\norm{\nabla u_{n}}_{L^{p}\of{A_{n}}}\geq C_{p}\left(\frac{\left|A_{n}\right|}{\left|\S^{d-1}\right|r_{n}^{d}}\right)^{\frac{1}{p}}n\left(\norm{\nabla^{s}u_{n}}_{L^{p}\of{A_{n}}}+\norm{\nabla u_{n}}_{L^{p}\of{\Ball{r_{n}}{x_{n}}}}\right)\,.
\]
We define $\overline{\nabla_{n}^{\bot}}(u_{n}):=\fint_{A_{n}}\left(\nabla u_{n}-\nablas u_{n}\right)$
and $u_{n,\bot}(x):=u_{n}(x)-\overline{\nabla_{n}^{\bot}}(u_{n})\,x$
with $\nabla^{s}u_{n,\bot}=\nabla^{s}u_{n}$. Hence by (\ref{eq:simple-Korn})
\begin{align*}
\norm{\nabla u_{n}-\overline{\nabla_{n}^{\bot}}(u_{n})}_{L^{p}\of{A_{n}}} & \leq C_{p}\norm{\nablas u_{n,\bot}}_{L^{p}\of{A_{n}}}=C_{p}\norm{\nablas u_{n}}_{L^{p}\of{A_{n}}}\,.
\end{align*}
We directly infer with $C_{n}:=\frac{\left|A_{n}\right|}{\left|\S^{d-1}\right|r_{n}^{d}}$
\begin{equation}
C_{p}C_{n}^{\frac{1}{p}}\left(\norm{\nabla^{s}u_{n}}_{L^{p}\of{A_{n}}}+\norm{\nabla u_{n}}_{L^{p}\of{\Ball{r_{n}}{x_{n}}}}\right)\to0\,,\qquad C_{n}^{\frac{1}{p}}\norm{\nabla u_{n}-\overline{\nabla_{n}^{\bot}}(u_{n})}_{L^{p}\of{A_{n}}}\to0\,.\label{eq:lem:General-Korn-help-1}
\end{equation}
Furthermore, we find 
\begin{align*}
1=\norm{\nabla u_{n}}_{L^{p}\of{A_{n}}} & \geq\norm{\overline{\nabla_{n}^{\bot}}(u_{n})}_{L^{p}\of{A_{n}}}-\norm{\nabla u_{n}-\overline{\nabla_{n}^{\bot}}(u_{n})}_{L^{p}\of{A_{n}}}\,,\\
\norm{\overline{\nabla_{n}^{\bot}}(u_{n})}_{L^{p}\of{A_{n}}} & \geq\norm{\nabla u_{n}}_{L^{p}\of{A_{n}}}-\norm{\nabla u_{n}-\overline{\nabla_{n}^{\bot}}(u_{n})}_{L^{p}\of{A_{n}}}\,,
\end{align*}
and hence $\norm{\overline{\nabla_{n}^{\bot}}(u_{n})}_{L^{p}\of{A_{n}}}\to1$
due to (\ref{eq:lem:General-Korn-help-1}). Since $\overline{\nabla_{n}^{\bot}}(u_{n})$
are constant, it holds 
\[
C_{n}\,\norm{\overline{\nabla_{n}^{\bot}}(u_{n})}_{L^{p}\of{\Ball{r_{n}}{x_{n}}}}^{p}=\norm{\overline{\nabla_{n}^{\bot}}(u_{n})}_{L^{p}\of{A_{n}}}^{p}
\]
and we infer from a similar calculation 
\begin{align*}
C_{n}^{\frac{1}{p}}\left(\norm{\nabla u_{n}}_{L^{p}\of{\Ball{r_{n}}{x_{n}}}}+\norm{\nabla u_{n}-\overline{\nabla_{n}^{\bot}}(u_{n})}_{L^{p}\of{\Ball{r_{n}}{x_{n}}}}\right) & \geq C_{n}^{\frac{1}{p}}\norm{\overline{\nabla_{n}^{\bot}}(u_{n})}_{L^{p}\of{\Ball{r_{n}}{x_{n}}}}\\
 & \geq\norm{\overline{\nabla_{n}^{\bot}}(u_{n})}_{L^{p}\of{A_{n}}}\,.
\end{align*}
This implies $\norm{\overline{\nabla_{n}^{\bot}}(u_{n})}_{L^{p}\of{A_{n}}}\to0$
by (\ref{eq:lem:General-Korn-help-1}), a contradiction. Hence, (\ref{eq:Scaled-Korn-inequality})
holds with $\tilde{C}_{p}=nC_{p}$ for some $n\in\N$. 

Estimate (\ref{eq:Scaled-Korn-inequality-1}) now follows from (\ref{eq:Scaled-Korn-inequality})
and (\ref{eq:simple-Korn}) and the definition of $\bW_{\nablabot(0),r}^{1,p}(\Ball R0)$.
\end{proof}

\subsection{Korn-Poincaré Inequalities}

Generalizing the above Korn inequality to a Korn-Poincaré inequality,
we define 
\begin{align*}
\bW_{(0),\nablabot(0),r}^{1,p}\of{\Ball R0} & :=\left\{ u\in\bW^{1,p}(\Ball r0)\,:\;\exists x:\,B_{r}(x)\subset\Ball R0\;\vee\right.\phantom{:=W^{1,p}\int_{\Ball rx}\quad}\\
 & \phantom{:=\left\{ u\in W^{1,p}(\Ball r0)\,:\right\} }\left.\int_{\Ball rx}u_{i}=0\;\vee\;\forall i,j:\,\int_{\Ball rx}\partial_{i}u_{j}-\partial_{j}u_{i}=0\right\} \,.
\end{align*}

\begin{lem}[Mixed Korn-Poincaré inequality on balls]
\label{lem:general-Korn-Poincar=0000E9}For every $p\in[1,\infty)$
there exists $C_{p}>0$ such that for every $R>0$, $r\in(0,R)$ and
every $x\in\Ball R0$ with $\Ball rx\subset\Ball R0$ it holds 
\begin{align}
\forall u\in\bW_{(0),\nablabot(0),r}^{1,p}\of{\Ball R0}: &  & \norm{\nabla u}_{L^{p}\of{\Ball R0}}^{p} & \leq C_{p}\left(\frac{R}{r}\right)^{d}\norm{\nabla^{s}u}_{L^{p}\of{\Ball R0}}^{p}\,,\label{eq:Korn-all-average-zero}\\
 &  & \norm u_{L^{p}\of{\Ball R0}}^{p} & \leq C_{p}\left(\frac{R}{r}\right)^{2d-1}\left(1+\left(\frac{R}{r}\right)^{1-p}\right)R^{p}\norm{\nabla^{s}u}_{L^{p}\of{\Ball R0}}^{p}\,.\label{eq:lem:Poincare-ball-2-Korn}
\end{align}
\end{lem}

\begin{proof}
Apply Lemma \ref{lem:General-Korn} for $R=1$ and use a simple scaling
argument to obtain 
\[
\norm{\nabla u}_{L^{p}\of{\Ball R0}}\leq C_{p}\left(\frac{R}{r}\right)^{d}\left(\norm{\nabla^{s}u}_{L^{p}\of{\Ball r0}}\right)\,.
\]
 Afterwards apply Lemma \ref{lem:Poincare-ball}.
\end{proof}
\begin{lem}[Mixed Korn-Poincaré inequality on cylinders]
\label{lem:Mixed-Korn-Zylinder}For every $p\in[1,\infty)$ and $r>0$
there exists $C_{p}>0$ such that the following holds: Let $r<L$,
$P_{L,r}:=(0,L)\times\Balldim[d-1]r0$ and $x\in P_{L,r}$ such that
$\Ball rx\subset P_{L,r}$ then for every $u\in\bW^{1,p}(P_{L,r})$
\begin{equation}
\left\Vert \nabla u\right\Vert _{L^{p}(P_{L,r})}^{p}\leq C_{p}\left(\left(\frac{L}{r}\right)^{p}\left\Vert \nablas u\right\Vert _{L^{p}(P_{L,r})}^{p}+\frac{L}{r}\left\Vert \nabla u\right\Vert _{L^{p}(\Ball rx)}^{p}\right)\,.\label{eq:lem:Mixed-Korn-zylinder}
\end{equation}
Furthermore,
\begin{equation}
\left\Vert u\right\Vert _{L^{p}(P_{L,r})}^{p}\leq C_{p}\left(\frac{L^{2p}}{r^{p}}\left\Vert \nablas u\right\Vert _{L^{p}(P_{L,r})}^{p}+\frac{L^{p+1}}{r}\left\Vert \nabla u\right\Vert _{L^{p}(\Ball rx)}^{p}+\frac{L}{r}\left\Vert u\right\Vert _{L^{p}(\Ball rx)}^{p}\right)\,,\label{eq:lem:Mixed-Korn-zylinder-2}
\end{equation}
and if additionally $u\in\bW_{(0),\nabla^{\bot}(0),r}^{1,p}(P_{L,r})$
then 
\begin{equation}
\left\Vert \nabla u\right\Vert _{L^{p}(P_{L,r})}^{p}\leq C_{p}\frac{L^{p}}{r^{p}}\left\Vert \nablas u\right\Vert _{L^{p}(P_{L,r})}^{p}\,,\quad\left\Vert u\right\Vert _{L^{p}(P_{L,r})}^{p}\leq C_{p}\frac{L^{2p}}{r^{p}}\left\Vert \nablas u\right\Vert _{L^{p}(P_{L,r})}^{p}\,,\label{eq:lem:Mixed-Korn-zylinder-3}
\end{equation}
Defining $\overline{\nabla_{a,\delta}^{\bot}}u:=\fint_{\Ball{\delta}a}\left(\nabla u-\nablas u\right)$
and 
\begin{equation}
\left[\cM_{a}^{\fs,\delta}u\right](x):=\overline{\nabla_{a,\delta}^{\bot}}u\of{x-a}+\fint_{\Ball{\delta}a}u\label{eq:lem:Mixed-Korn-zylinder-cM}
\end{equation}
 we find for $a,b$ with $\Ball{\delta}a,\Ball{\delta}b\subset P_{L,r}$
for every $u\in\bW^{1,p}(P_{L,r})$ that
\begin{equation}
\left|\left[\cM_{a}^{\fs,\delta}u\right](x)-\left[\cM_{b}^{\fs,\delta}u\right](x)\right|^{p}\leq C\left|x-a\right|^{p}\frac{\left|a-b\right|^{2p}}{\delta^{p+d}}\left(\int_{\conv\of{\Ball{\delta}a\cup\Ball{\delta}b}}\left|\nablas u\right|^{p}\right)\,.\label{eq:lem:Mixed-Korn-zylinder-4}
\end{equation}
Furthermore, for every $\delta<r$ we find
\begin{equation}
\left|\left[\cM_{a}^{\fs,r}u\right](x)-\left[\cM_{a}^{\fs,\delta}u\right](x)\right|^{p}\leq C\left(\left(\frac{\delta}{r}\right)^{-d}\left|x-a\right|^{p}+\left(\frac{\delta}{r}\right)^{1-d}(1+\left(\frac{\delta}{r}\right)^{p-d})\right)r^{p-d}\norm{\nablas u}_{L^{p}\of{\Ball ra}}^{p}\,.\label{eq:lem:Mixed-Korn-zylinder-5}
\end{equation}
\end{lem}

\begin{proof}
\emph{Step1:} W.l.o.g we assume $L\in\N$, $a=\frac{1}{2}\be_{1}$,
$b=(L-\frac{1}{2})\be_{1}$, $r=\frac{1}{2}$ and define 
\begin{align*}
\bP_{k} & :=\left(k\be_{1}+[0,1)\times\Balldim[d-1]{\frac{1}{2}}0\right)\,,\quad\bB_{k}:=k\be_{1}+\Ball{\frac{1}{2}}{\frac{1}{2}\be_{1}}\\
\tau_{k}^{\fs}u(x) & :=\left[\cM_{\left(k+\frac{1}{2}\right)\be_{1}}^{\fs,\frac{1}{2}}u\right](x)=\left[\fint_{\bB_{k}}\left(\nabla u-\nablas u\right)\right]x+\fint_{\bB_{k}}u\,.
\end{align*}
Then we find by Lemma \ref{lem:General-Korn}
\begin{align*}
\norm{\nabla u}_{L^{p}(\bP_{K})}^{p} & \leq C\left(\norm{\nabla\left(u-\tau_{K}^{\fs}u\right)}_{L^{p}(\bP_{K})}^{p}+\norm{\nabla\tau_{K}^{\fs}u}_{L^{p}(\bP_{K})}^{p}\right)\\
 & \leq C\left(\norm{\nablas u}_{L^{p}(\bP_{K})}^{p}+\norm{\nabla\tau_{K}^{\fs}u}_{L^{p}(\bP_{K})}^{p}\right)\,.
\end{align*}
Since $\nabla\tau_{k}^{\fs}u$ is constant, we find
\begin{align*}
\norm{\nabla\tau_{K}^{\fs}u}_{L^{p}(\bP_{K})}^{p} & \leq C\norm{\nabla\tau_{0}^{\fs}u}_{L^{p}(\bP_{0})}^{p}+C\left(\sum_{k=0}^{K-1}\norm{\nabla\left(\tau_{k+1}^{\fs}u-\tau_{k}^{\fs}u\right)}_{L^{1}\of{\bP_{k+1}}}\right)^{p}\,.
\end{align*}
Furthermore, we find 
\begin{align*}
\tau_{k}^{\fs}\of{u-\tau_{k+1}^{\fs}u} & =\fint_{\bB_{k}}\left(\nabla u-\fint_{\bB_{k+1}}\left(\nabla u-\nablas u\right)-\nablas u\right)x+\fint_{\bB_{k}}\left(u-\fint_{\bB_{k+1}}u\right)\\
 & =\tau_{k}^{\fs}u-\tau_{k+1}^{\fs}u=\tau_{k+1}^{\fs}\of{u-\tau_{k}^{\fs}u}\,.
\end{align*}
This implies by $\nabla\tau_{k+1}^{\fs}\left(u-\tau_{k}^{\fs}u\right)=\fint_{\bB_{k+1}}\left(\nabla-\nablas\right)\left(u-\tau_{k}^{\fs}u\right)$
and Lemma \ref{lem:General-Korn} and Theorem \ref{thm:Duran-Korn}
\begin{align*}
\norm{\nabla\left(\tau_{k+1}^{\fs}u-\tau_{k}^{\fs}u\right)}_{L^{p}\of{\bP_{k+1}}}^{p} & \leq C\norm{\nabla\tau_{k+1}^{\fs}\left(u-\tau_{k}^{\fs}u\right)}_{L^{p}\of{\bB_{k+1}}}^{p}\\
 & \leq C\norm{\nabla\left(u-\tau_{k}^{\fs}u\right)}_{L^{p}\of{\bB_{k+1}}}^{p}\\
 & \stackrel{\ref{lem:General-Korn}}{\leq}\leq C\left(\norm{\nabla^{s}\left(u-\tau_{k}^{\fs}u\right)}_{L^{p}\of{\bP_{k+1}\cup\bP_{k}}}^{p}+\norm{\nabla\left(u-\tau_{k}^{\fs}u\right)}_{L^{p}\of{\bB_{k}}}^{p}\right)\\
 & \stackrel{\ref{thm:Duran-Korn}}{\leq}C\norm{\nabla^{s}u}_{L^{p}\of{\bP_{k+1}\cup\bP_{k}}}^{p}\,.
\end{align*}
Since the last inequality implies 
\[
\left(\sum_{k=0}^{K-1}\norm{\nabla\left(\tau_{k+1}^{\fs}u-\tau_{k}^{\fs}u\right)}_{L^{1}\of{\bP_{k+1}}}\right)^{p}\leq K^{p-1}C\norm{\nablas u}_{L^{p}\of{(0,K)\times\Balldim[d-1]10}}^{p}
\]
and $\norm{\nabla\tau_{0}^{\fs}u}_{L^{p}(\bP_{0})}^{p}\leq C\left(\norm{\nablas u}_{L^{p}(\bP_{0})}^{p}+\norm{\nabla u}_{L^{p}(\bB_{0})}^{p}\right)$
by Lemma \ref{lem:General-Korn} we find in total 
\[
\norm{\nabla\tau_{K}^{\fs}u}_{L^{p}(\bP_{K})}^{p}\leq C\norm{\nabla u}_{L^{p}(\bB_{0})}^{p}+CK^{p-1}\norm{\nablas u}_{L^{p}\of{(0,K)\times\Balldim[d-1]10}}^{p}\,.
\]
Adding the last inequality from $K=0$ to $K=L$ implies (\ref{eq:lem:Mixed-Korn-zylinder})
through scaling. Applying Corollary \ref{cor:Poincare-zylinder} we
infer that (\ref{eq:lem:Mixed-Korn-zylinder-2}) and (\ref{eq:lem:Mixed-Korn-zylinder-3}).

Step 2: We observe that Step 1 also holds for $P_{L,r}$ being replaced
by $\conv(\Ball{\delta}a\cup\Ball{\delta}b)$. Writing $u_{b}:=u-\cM_{b}^{\fs,\delta}u$
we find from the above calculations
\begin{align*}
\left|\cM_{a}^{\fs,\delta}u-\cM_{b}^{\fs,\delta}u\right|^{p}(x) & =\left|\cM_{a}^{\fs,\delta}\left(u-\cM_{b}^{\fs,\delta}u\right)\right|^{p}(x)\\
 & \leq C\frac{1}{\delta^{d}}\left(\left|x-a\right|^{p}\int_{\Ball{\delta}a}\left|\nabla u_{b}-\nablas u_{b}\right|^{p}+\int_{\Ball{\delta}a}\left|u_{b}\right|^{p}\right)\,.
\end{align*}
Using that $u_{b}\in\bW_{(0),\nabla^{\bot}(0),r}^{1,p}(\conv(\Ball{\delta}a\cup\Ball{\delta}b))$,
we find (\ref{eq:lem:Mixed-Korn-zylinder-4}) with help of (\ref{eq:lem:Mixed-Korn-zylinder-3})
and Lemma \ref{lem:general-Korn-Poincar=0000E9}.

Step 3: W.l.o.g. $a=0$. Writing $\bar{u}(y):=u(y)-\of{\overline{\nabla_{a,\delta}^{\bot}}u}\,y$
with $\fint_{\Ball r0}u=\fint_{\Ball r0}\overline{u}$ we infer (\ref{eq:lem:Mixed-Korn-zylinder-5})
from Lemmas \ref{lem:General-Korn} and \ref{lem:Poincare-ball} via
\begin{align*}
\left|\left[\cM_{0}^{\fs,1}u\right](x)-\left[\cM_{0}^{\fs,\delta}u\right](x)\right|^{p} & \leq C\left|\int_{\Ball 10}\nabla u-\nablas u-\overline{\nabla_{a,\delta}^{\bot}}u\right|^{p}\left|x\right|^{p}+\left|\fint_{\Ball 10}\overline{u}-\fint_{\Ball{\delta}0}\overline{u}\right|^{p}\\
 & \leq C\int_{\Ball 10}\left(\left|\nabla u-\overline{\nabla_{a,\delta}^{\bot}}u\right|^{p}+\left|\nablas u\right|^{p}\right)\left|x\right|^{p}+\fint_{\Ball 10}\left|\overline{u}-\fint_{\Ball{\delta}0}\overline{u}\right|^{p}\\
 & \leq C\left(\delta^{-d}\left|x\right|^{p}\norm{\nablas u}_{L^{p}\of{\Ball 10}}^{p}+\delta^{1-d}\left(1+\delta^{p-d}\right)\norm{\nabla\overline{u}}_{L^{p}\of{\Ball 10}}^{p}\right)\,.
\end{align*}
\end{proof}
\newpage{}

\subsection{Voronoi Tessellations and Delaunay Triangulation}
\begin{defn}[Voronoi Tessellation]
\label{def:Voronoi}Let $\X=\left(x_{i}\right)_{i\in\N}$ be a sequence
of points in $\Rd$ with $x_{i}\neq x_{k}$ if $i\neq k$. For each
$x\in\X$ let \nomenclature[G]{$G(x)$}{Voronoi cell with center $x$ (Definition \ref{def:Voronoi})}
\[
G\of x:=\left\{ y\in\Rd\,:\;\forall\tilde{x}\in\X\backslash\left\{ x\right\} :\,\left|x-y\right|<\left|\tilde{x}-y\right|\right\} \,.
\]
Then $\left(G\of{x_{i}}\right)_{i\in\N}$ is called the \emph{Voronoi
tessellation} of $\Rd$ w.r.t. $\X$. For each $x\in\X$ we define
$d\of x:=\diam G\of x$.
\end{defn}

We will need the following result on Voronoi tessellation of a minimal
diameter.
\begin{lem}
\label{lem:estim-diam-Voronoi-cells}Let $\fr>0$ and let $\X=\left(x_{i}\right)_{i\in\N}$
be a sequence of points in $\Rd$ with $\left|x_{i}-x_{k}\right|>2\fr$
if $i\neq k$. For $x\in\X$ let $\cI\of x:=\left\{ y\in\X\,:\;G\of y\cap\Ball{\fr}{G\of x}\not=\emptyset\right\} $.
Then $y\in\cI(x)$ implies $\left|x-y\right|\leq4d(x)$ and 
\begin{equation}
\#\cI\of x\leq\left(\frac{4d\of x}{\fr}\right)^{d}\,.\label{eq:lem:estim-diam-Voronoi-cells}
\end{equation}
\end{lem}

\begin{proof}
Let $\X_{k}=\left\{ x_{j}\in\X\,:\;\cH^{d-1}\of{\partial G_{k}\cap\partial G_{j}}\geq0\right\} $
the neighbors of $x_{k}$ and $d_{k}:=d\of{x_{k}}$. Then all $x_{j}\in\X$
satisfy $\left|x_{k}-x_{j}\right|\leq2d_{k}$. Moreover, every $\tilde{x}\in\X$
with $\left|\text{\ensuremath{\tilde{x}}}-x_{k}\right|>4d_{k}$ has
the property that $\dist\of{\,\partial G\left(\tilde{x}\right),\,x_{k}\,}>2d_{k}>d_{k}+\fr$
and $\tilde{x}\not\in\cI_{k}$. Since every Voronoi cell contains
a ball of radius $\fr$, this implies that $\#\cI_{k}\leq\left|\Ball{4d_{k}}{x_{k}}\right|/\left|\Ball{\fr}0\right|=\left(\frac{4d_{k}}{\fr}\right)^{d}$.
\end{proof}
\begin{defn}[Delaunay Triangulation]
\label{def:delaunay}Let $\X=\left(x_{i}\right)_{i\in\N}$ be a sequence
of points in $\Rd$ with $x_{i}\neq x_{k}$ if $i\neq k$. The Delaunay
triangulation is the dual unoriented graph (see Def. \ref{def:unoriented-graph}
below) of the Voronoi tessellation, i.e. we say $\D(\X):=\left\{ (x,y):\;\hausdorffH^{d-1}\of{\partial G(x)\cap\partial G(y)}\neq0\right\} $. 
\end{defn}

\subsection{\label{subsec:Local--Regularity}Local $\eta$-Regularity}

\begin{figure}
 \begin{minipage}[c]{0.5\textwidth} \includegraphics{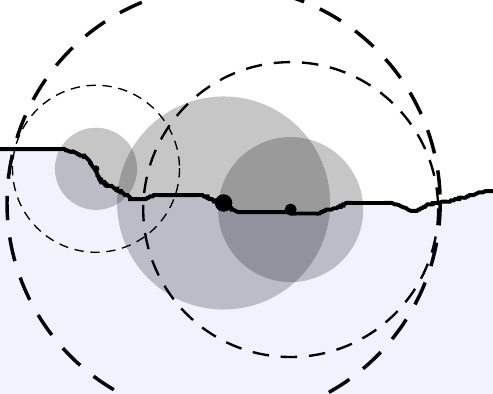}\end{minipage}\hfill   \begin{minipage}[c]{0.45\textwidth}\caption{An illustration of $\eta$-regularity. In Theorem \ref{thm:delta-M-rho-covering}
we will rely on a \textquotedblleft gray\textquotedblright{} region
like in this picture.}
\end{minipage}
\end{figure}

\begin{defn}[$\eta$- regularity]
\label{def:eta-regular}\nomenclature[G]{$\eta$-regular (local)}{(Definition \ref{def:eta-regular})}

For a function $\eta:\,\partial\bP\to(0,r]$ we call $\bP$ $\eta$-regular
if 
\begin{equation}
\forall p\in\partial\bP,\,\eps\in\of{0,\frac{1}{2}}\,,\,\tilde{p}\in\Ball{\eps\eta(p)}p\cap\partial\bP:\,\eta(\tilde{p})>(1-\eps)\eta(p)\,.\label{eq:def:eta-regular}
\end{equation}
\end{defn}

\begin{rem}
This concept and its consequences from Lemma \ref{lem:eta-lipschitz}
and Theorem \ref{thm:delta-M-rho-covering} will be extensively used
later to cover $\partial\bP$ by a suitable family of open balls.
\end{rem}

\begin{lem}
\label{lem:eta-lipschitz}Let $\bP$ be a locally $\eta$-regular
set for $\eta:\partial\bP\to(0,\fr)$. Then $\eta:\,\bP\to\R$ is
locally Lipschitz continuous with Lipschitz constant $1$ and for
every $\eps\in\left(0,\frac{1}{2}\right)$ and $\tilde{p}\in\Ball{\eps\eta}p\cap\bP$
it holds
\begin{equation}
\frac{1-\eps}{1-2\eps}\eta\of p>\eta\of{\tilde{p}}>\eta\of p-\left|p-\tilde{p}\right|>\left(1-\eps\right)\eta\of p\,.\label{eq:eta-lipschitz-ineq-chain}
\end{equation}
Furthermore,
\begin{equation}
\left|p-\tilde{p}\right|\leq\eps\max\left\{ \eta\of p,\eta\of{\tilde{p}}\right\} \quad\Rightarrow\quad\left|p-\tilde{p}\right|\leq\frac{\eps}{1-\eps}\min\left\{ \eta\of p,\eta\of{\tilde{p}}\right\} \label{eq:lem:eta-lipschitz-dist-estim}
\end{equation}
\end{lem}

\begin{proof}
Let $p,\tilde{p}$ such that $\left|\tilde{p}-p\right|<\frac{1}{2}\eta(p)$
with $\eps_{p,\tilde{p}}:=\inf\left\{ \eps:\,\left|\tilde{p}-p\right|<\eps\eta(p)\right\} $.
This means $\eps\in[\eps_{p,\tilde{p}},\frac{1}{2})$ iff $\eta\of{\tilde{p}}\geq\left(1-\eps\right)\eta(p)$
and we find 
\begin{align*}
\eta\of{\tilde{p}} & \geq\eta\of p-\left|p-\tilde{p}\right|=\eta\of p-\eps_{p,\tilde{p}}\eta\of p>\left(1-\eps\right)\eta\of p
\end{align*}
 which implies $\left|\tilde{p}-p\right|<\frac{\eps}{1-\eps}\eta\of{\tilde{p}}$
and the local Lipschitz continuity by a symmetry argument in $p$,
$\tilde{p}$. This in turn leads to $\eta\of p>\left(1-\frac{\eps}{1-\eps}\right)\eta\of{\tilde{p}}$
or 
\[
\eta\of p=\frac{1-\eps}{1-\eps}\eta\of p<\frac{1}{1-\eps}\left(\eta\of p-\left|p-\tilde{p}\right|\right)<\frac{1}{1-\eps}\eta\of{\tilde{p}}\leq\frac{1}{1-2\eps}\eta\of p\,,
\]
implying (\ref{eq:eta-lipschitz-ineq-chain}) and continuity of $\eta$.

In order to prove (\ref{eq:lem:eta-lipschitz-dist-estim}), w.l.o.g.
let $\eta\of{\tilde{p}}\leq\eta\of p$. Then
\[
\left|p-\tilde{p}\right|\leq\eps\eta\of p\leq\frac{\eps}{1-\eps}\eta\of{\tilde{p}}\,.
\]
\end{proof}
\begin{thm}
\label{thm:delta-M-rho-covering}Let $\Gamma\subset\Rd$ be a closed
set and let $\eta\of{\cdot}\in C\of{\Gamma}$ be bounded and satisfy
for every $\eps\in\left(0,\frac{1}{2}\right)$ and for $\left|p-\tilde{p}\right|<\eps\eta\of p$
\begin{equation}
\frac{1-\eps}{1-2\eps}\,\eta\of p>\eta\of{\tilde{p}}>\eta\of p-\left|p-\tilde{p}\right|>\left(1-\eps\right)\eta\of p\,.\label{eq:thm:delta-M-rho-covering-a}
\end{equation}
and define $\tilde{\eta}\of p=2^{-K}\eta\of p$, $K\geq2$. Then for
every $C\in(0,1)$ there exists a locally finite covering of $\Gamma$
with balls $\Ball{\tilde{\eta}\of{p_{k}}}{p_{k}}$ for a countable
number of points $\of{p_{k}}_{k\in\N}\subset\Gamma$ such that for
every $i\neq k$ with $\Ball{\tilde{\eta}\of{p_{i}}}{p_{i}}\cap\Ball{\tilde{\eta}\of{p_{k}}}{p_{k}}\neq\emptyset$
it holds 
\begin{equation}
\begin{aligned} & \frac{2^{K-1}-1}{2^{K-1}}\tilde{\eta}\of{p_{i}}\leq\tilde{\eta}\of{p_{k}}\leq\frac{2^{K-1}}{2^{K-1}-1}\tilde{\eta}\of{p_{i}}\\
\text{and}\quad & \frac{2^{K}-1}{2^{K-1}-1}\min{\left\{ \tilde{\eta}\of{p_{i}},\tilde{\eta}\of{p_{k}}\right\} }\geq\left|p_{i}-p_{k}\right|\geq C\max{\left\{ \tilde{\eta}\of{p_{i}},\tilde{\eta}\of{p_{k}}\right\} }
\end{aligned}
\label{eq:thm:delta-M-rho-covering}
\end{equation}
\end{thm}

\begin{proof}
We chose $\delta>0$, $n\in\N$ such that $\left(1-\frac{1}{n}\right)\left(1-\delta\right)>C$.
W.o.l.g. assume $\tilde{\eta}<(1-\delta)$. Consider $\tilde{Q}:=\left[0,\frac{1}{n}\right]^{d}$,
let $q_{1,\dots,n^{d}}$ denote the $n^{d}$ elements of $[0,1)^{d}\cap\frac{\Q^{d}}{n}$
and let $\tilde{Q}_{z,i}=\tilde{Q}+z+q_{i}$, $z\in\Zd$. We set $B_{(0)}:=\emptyset$,
$\Gamma_{1}=\Gamma$, $\eta_{k}:=\left(1-\delta\right)^{k}$ and for
$k\geq1$ we construct the covering using inductively defined open
sets $B_{(k)}$ and closed set $\Gamma_{k}$ as follows:
\begin{enumerate}
\item Define $\Gamma_{k,1}=\Gamma_{k}$. For $i=1,\dots,n^{d}$ do the following:
\begin{enumerate}
\item For every $z\in\Zd$ do 
\[
\begin{aligned} & \text{if }\exists p\in\left(\eta_{k}\tilde{Q}_{z,i}\right)\cap\Gamma_{k,i},\,\tilde{\eta}\of p\in(\eta_{k},\eta_{k-1}] & \text{then set } & b_{z,i}=\Ball{\tilde{\eta}\of p}p\,,\;\X_{z,i}=\left\{ p\right\} \\
 & \text{otherwise } & \text{set } & b_{z,i}=\emptyset\,,\;\X_{z,i}=\emptyset\,.
\end{aligned}
\]
\item Define $B_{(k),i}:=\bigcup_{z\in\Zd}b_{z,i}$ and $\Gamma_{k,i+1}=\Gamma_{k,i}\backslash B_{(k),i}$
and $\X_{(k),i}:=\bigcup_{z\in\Zd}\X_{z,i}$.\\
Observe: $p_{1},p_{2}\in\X_{(k),i}$ implies $\left|p_{1}-p_{2}\right|>\left(1-\frac{1}{n}\right)\eta_{k}$
and $p_{3}\in\X_{(k),j}$, $j<i$ implies $p_{1}\not\in\Ball{\eta_{k}}{p_{3}}$
and hence $\left|p_{1}-p_{3}\right|>\eta_{k}$. Similar, $p_{3}\in\X_{l}$,
$l<k$, implies $\left|p_{1}-p_{3}\right|>\eta_{l}>\eta_{k}$.
\end{enumerate}
\item Define $\Gamma_{k+1}:=\Gamma_{k,n^{d}+1}$, $\X_{k}:=\bigcup_{i}\X_{(k),i}$.
\end{enumerate}
The above covering of $\Gamma$ is complete in the sense that every
$x\in\Gamma$ lies in one of the balls (by contradiction). We denote
$\X:=\bigcup_{k}\X_{k}=\left(p_{i}\right)_{i\in\N}$ the family of
centers of the above constructed covering of $\Gamma$ and find the
following properties: Let $p_{1},p_{2}\in\X$ be such that $\Ball{\tilde{\eta}\of{p_{1}}}{p_{1}}\cap\Ball{\tilde{\eta}\of{p_{2}}}{p_{2}}\neq\emptyset$.
W.l.o.g. let $\tilde{\eta}\of{p_{1}}\geq\tilde{\eta}\of{p_{2}}$.
Then the following two properties are satisfied due to (\ref{eq:thm:delta-M-rho-covering-a})
\begin{enumerate}
\item It holds $\left|p_{1}-p_{2}\right|\leq2\tilde{\eta}\of{p_{1}}\leq\frac{1}{2^{K-1}}\eta\of{p_{1}}$
and hence $\Ball{\tilde{\eta}\of{p_{2}}}{p_{2}}\subset\Ball{2^{2-K}\eta\of{p_{1}}}{p_{1}}$
and $\eta\of{p_{2}}\geq\frac{2^{K-1}-1}{2^{K-1}}\eta\of{p_{1}}$.
Furthermore $\tilde{\eta}\of{p_{1}}\geq\tilde{\eta}\of{p_{2}}\geq\frac{2^{K-1}-1}{2^{K-1}}\tilde{\eta}\of{p_{1}}$.
\item Let $k$ such that $\tilde{\eta}\of{p_{1}}\in\left(\eta_{k},\eta_{k+1}\right]$.
If also $\tilde{\eta}\of{p_{2}}\in\left(\eta_{k},\eta_{k+1}\right]$
then the observation in Step 1.(b) implies $\left|p_{1}-p_{2}\right|\geq\left(1-\frac{1}{n}\right)\eta_{k}\geq\left(1-\frac{1}{n}\right)\left(1-\delta\right)\tilde{\eta}\of{p_{1}}$.
If $\tilde{\eta}\of{p_{2}}\not\in\left[\eta_{k},\eta_{k+1}\right)$
then $\tilde{\eta}\of{p_{2}}<\eta_{k}$ and hence $p_{2}\not\in\Ball{\tilde{\eta}\of{p_{1}}}{p_{1}}$,
implying $\left|p_{1}-p_{2}\right|>\tilde{\eta}\of{p_{1}}$.
\end{enumerate}
Due to our choice of $n$ and $\delta$, this concludes the proof.
\end{proof}

\subsection{Dynamical Systems}
\begin{assumption}
\label{assu:separable}Throughout this work we assume that $(\Omega,\sF,\P)$
is a probability space with countably generated $\sigma$-algebra
$\sF$.
\end{assumption}

Due to the insight in \cite{heida2011extension}, shortly sketched
in the next two subsections, after a measurable transformation the
probability space $\Omega$ can be assumed to be metric and separable,
which always ensures Assumption \ref{assu:separable}.
\begin{defn}[Dynamical system]
\label{def:Omega-mu-tau}A dynamical \nomenclature[Tau]{$\tau_x$}{Dynamical system (Definitions \ref{def:Omega-mu-tau}, \ref{def:A-dynamical-system-Zd}) with respect to $x\in\Rd$ or $x\in\Zd$}system
on $\Omega$ is a family $(\tau_{x})_{x\in\Rd}$ of measurable bijective
mappings $\tau_{x}:\Omega\mapsto\Omega$ satisfying (i)-(iii):

\begin{enumerate}
\item [(i)]$\tau_{x}\circ\tau_{y}=\tau_{x+y}$ , $\tau_{0}=id$ (Group
property)
\item [(ii)]$\P(\tau_{-x}B)=\P(B)\quad\forall x\in\Rd,\,\,B\in\sF$ (Measure
preserving)
\item [(iii)]$A:\,\,\Rd\times\Omega\rightarrow\Omega\qquad(x,\omega)\mapsto\tau_{x}\omega$
is measurable (Measurability of evaluation)
\end{enumerate}
\end{defn}

A set $A\subset\Omega$ is almost invariant if \nomenclature[I]{$\sI$}{Invariant sets, \eqref{eq:invariant-sets}}$\P\left(\left(A\cup\tau_{x}A\right)\backslash\left(A\cap\tau_{x}A\right)\right)=0$.
The family 
\begin{equation}
\sI=\left\{ A\in\sF\,:\;\forall x\in\Rd\,\P\left(\left(A\cup\tau_{x}A\right)\backslash\left(A\cap\tau_{x}A\right)\right)=0\right\} \label{eq:invariant-sets}
\end{equation}
 of almost invariant sets is $\sigma$-algebra and \nomenclature[E]{$\E(f|\sI)$}{Expectation of $f$ wrt. the invariant sets, \eqref{eq:invariant-sets-expectation}}
\begin{equation}
\E\left(f|\sI\right)\text{denotes the expectation of }f:\,\Omega\to\R\text{ w.r.t. }\sI\,.\label{eq:invariant-sets-expectation}
\end{equation}
A concept linked to dynamical systems is the concept of stationarity.
\begin{defn}[Stationary]
\label{def:stationary}\nomenclature{stationary}{Definition \ref{def:stationary}}Let
$X$ be a measurable space and let $f:\Omega\times\Rd\to X$. Then
$f$ is called (weakly) stationary if $f(\omega,x)=f(\tau_{x}\omega,0)$
for (almost) every $x$.
\end{defn}

\begin{defn}
\label{def:convex-averaging-sequence}\nomenclature[Convex averaging sequence]{Convex averaging sequence}{(Definition \ref{def:convex-averaging-sequence}) }A
family $\left(A_{n}\right)_{n\in\N}\subset\Rd$ is called convex averaging
sequence if
\begin{enumerate}
\item [(i)]each $A_{n}$ is convex
\item [(ii)]for every $n\in\N$ holds $A_{n}\subset A_{n+1}$
\item [(iii)]there exists a sequence $r_{n}$ with $r_{n}\to\infty$ as
$n\to\infty$ such that $B_{r_{n}}(0)\subseteq A_{n}$.
\end{enumerate}
\end{defn}

We sometimes may take the following stronger assumption.
\begin{defn}
A convex averaging sequence $A_{n}$ is called regular if 
\[
\left|A_{n}\right|^{-1}\#\left\{ z\in\Zd\,:\;\left(z+\T\right)\cap\partial A_{n}\neq\emptyset\right\} \to0\,.
\]
\end{defn}

The latter condition is evidently fulfilled for sequences of cones
or balls. Convex averaging sequences are important in the context
of ergodic theorems.
\begin{thm}[Ergodic Theorem \cite{Daley1988} Theorems 10.2.II and also \cite{tempel1972ergodic}]
\label{thm:Ergodic-Theorem} \nomenclature[Ergodic Theorem]{Ergodic Theorem}{(Theorems \ref{thm:Ergodic-Theorem}, \ref{thm:Ergodic-Theorem-ran-meas}) }Let
$\left(A_{n}\right)_{n\in\N}\subset\Rd$ be a convex averaging sequence,
let $(\tau_{x})_{x\in\Rd}$ be a dynamical system on $\Omega$ with
invariant $\sigma$-algebra $\sI$ and let $f:\,\Omega\to\R$ be measurable
with $\left|\E(f)\right|<\infty$. Then for almost all $\omega\in\Omega$
\begin{equation}
\left|A_{n}\right|^{-1}\int_{A_{n}}f\of{\tau_{x}\omega}\,\d x\to\E\of{f|\sI}\,.\label{eq:ergodic convergence}
\end{equation}
\end{thm}

We observe that $\E\left(f|\sI\right)$ is of particular importance.
For the calculations in this work, we will particularly focus on the
case of trivial $\sI$. This is called ergodicity, as we will explain
in the following.
\begin{defn}[Ergodicity and mixing]
\nomenclature[Ergodicity]{Ergodicity}{(Definition \ref{def:erg-mixing}) }\label{def:erg-mixing}\nomenclature[Mixing]{Mixing}{(Definition \ref{def:erg-mixing}) }A
dynamical system $(\tau_{x})_{x\in\Rd}$ on a probability space $(\Omega,\sF,\P)$
is called \emph{mixing }if for every measurable $A,B\subset\Omega$
it holds 
\begin{equation}
\lim_{\norm x\to\infty}\P\of{A\cap\tau_{x}B}=\P(A)\,\P(B)\,.\label{eq:def-mixing}
\end{equation}
A dynamical system is called \emph{ergodic }if 
\begin{equation}
\lim_{n\to\infty}\frac{1}{\left(2n\right)^{d}}\int_{[-n,n]^{d}}\P\of{A\cap\tau_{x}B}\d x=\P(A)\,\P(B)\,.\label{eq:def-ergodic}
\end{equation}
\end{defn}

\begin{rem}
\label{rem:trivial-mixing}a) Let $\Omega=\left\{ \omega_{0}=0\right\} $
with the trivial $\sigma$-algebra and $\tau_{x}\omega_{0}=\omega_{0}$.
Then $\tau$ is evidently mixing. However, the realizations are constant
functions $f_{\omega}(x)=c$ on $\Rd$ for some constant $c$.

b) A typical ergodic system is given by $\Omega=\T$ with the Lebesgue
$\sigma$-algebra and $\P=\lebesgueL$ the Lebesgue measure. The dynamical
system is given by $\tau_{x}y:=\left(x+y\right)\!\mod\T$.

c) It is known that $(\tau_{x})_{x\in\Rd}$ is ergodic if and only
if every almost invariant set $A\in\sI$ has probability $\P(A)\in\left\{ 0,1\right\} $
(see \cite{Daley1988} Proposition 10.3.III) i.e. 
\begin{equation}
\left[\,\,\forall x\,\P\of{\left(\tau_{x}A\cup A\right)\backslash\left(\tau_{x}A\cap A\right)}=0\,\,\right]\;\Rightarrow\;\P\of A\in\left\{ 0,1\right\} \,.\label{eq:def-ergodic-2}
\end{equation}

d) It is sufficient to show (\ref{eq:def-mixing}) or (\ref{eq:def-ergodic})
for $A$ and $B$ in a ring that generates the $\sigma$-algebra $\sF$.
We refer to \cite{Daley1988}, Section 10.2, for the later results.
\end{rem}

A further useful property of ergodic dynamical systems, which we will
use below, is the following:
\begin{lem}[Ergodic times mixing is ergodic]
\label{lem:erg-and-mix-is-erg}Let $(\tilde{\Omega},\tilde{\sF},\tilde{\P})$
and $(\hat{\Omega},\hat{\sF},\hat{\P})$ be probability spaces with
dynamical systems $(\tilde{\tau}_{x})_{x\in\Rd}$ and $(\hat{\tau}_{x})_{x\in\Rd}$
respectively. Let $\Omega:=\tilde{\Omega}\times\hat{\Omega}$ be the
usual product measure space with the notation $\omega=(\tilde{\omega},\hat{\omega})\in\Omega$
for $\tilde{\omega}\in\tilde{\Omega}$ and $\hat{\omega}\in\hat{\Omega}$.
If $\tilde{\tau}$ is ergodic and $\hat{\tau}$ is mixing, then $\tau_{x}(\tilde{\omega},\hat{\omega}):=(\tilde{\tau}_{x}\tilde{\omega},\hat{\tau}_{x}\hat{\omega})$
is ergodic.
\end{lem}

\begin{proof}
Relying on Remark \ref{rem:trivial-mixing}.c) we verify (\ref{eq:def-ergodic})
by proving it for sets $A=\tilde{A}\times\hat{A}$ and $B=\tilde{B}\times\hat{B}$
which generate $\sF:=\tilde{\sF}\otimes\hat{\sF}$. We make use of
$A\cap B=\left(\tilde{A}\cap\tilde{B}\right)\times\left(\hat{A}\cap\hat{B}\right)$
and observe that 
\begin{align*}
\P\of{A\cap\tau_{x}B} & =\P\of{\left(\tilde{A}\cap\tilde{\tau}_{x}\tilde{B}\right)\times\left(\hat{A}\cap\hat{\tau}_{x}\hat{B}\right)}=\hat{\P}\of{\hat{A}\cap\hat{\tau}_{x}\hat{B}}\,\tilde{\P}\of{\tilde{A}\cap\tilde{\tau}_{x}\tilde{B}}\\
 & =\hat{\P}\of{\hat{A}\cap\hat{B}}\,\tilde{\P}\of{\tilde{A}\cap\tilde{\tau}_{x}\tilde{B}}+\left[\hat{\P}\of{\hat{A}\cap\hat{\tau}_{x}\hat{B}}-\hat{\P}\of{\hat{A}\cap\hat{B}}\right]\,\tilde{\P}\of{\tilde{A}\cap\tilde{\tau}_{x}\tilde{B}}\,.
\end{align*}
Using ergodicity, we find that
\begin{align}
\lim_{n\to\infty}\frac{1}{\left(2n\right)^{d}}\int_{[-n,n]^{d}}\hat{\P}\of{\hat{A}\cap\hat{B}}\,\tilde{\P}\of{\tilde{A}\cap\tilde{\tau}_{x}\tilde{B}}\d x & =\hat{\P}\of{\left(\hat{A}\cap\hat{B}\right)}\,\tilde{\P}\of{\tilde{A}\cap\tilde{B}}\nonumber \\
 & =\P\of{A\cap B}\,.\label{eq:lem:erg-and-mix-is-erg-help-1}
\end{align}
Since $\hat{\tau}$ is mixing, we find for every $\eps>0$ some $R>0$
such that $\norm x>R$ implies $\left|\hat{\P}\of{\hat{A}\cap\hat{\tau}_{x}\hat{B}}-\hat{\P}\of{\hat{A}\cap\hat{B}}\right|<\eps$.
For $n>R$ we find 
\begin{multline}
\frac{1}{\left(2n\right)^{d}}\int_{[-n,n]^{d}}\left|\hat{\P}\of{\hat{A}\cap\hat{\tau}_{x}\hat{B}}-\hat{\P}\of{\hat{A}\cap\hat{B}}\right|\,\tilde{\P}\of{\tilde{A}\cap\tilde{\tau}_{x}\tilde{B}}\\
\leq\frac{1}{\left(2n\right)^{d}}\int_{[-n,n]^{d}}\eps+\frac{1}{\left(2n\right)^{d}}\int_{[-R,R]^{d}}2\to\eps\quad\text{as }n\to\infty\,.\label{eq:lem:erg-and-mix-is-erg-help-2}
\end{multline}
The last two limits (\ref{eq:lem:erg-and-mix-is-erg-help-1}) and
(\ref{eq:lem:erg-and-mix-is-erg-help-2}) imply (\ref{eq:def-ergodic}).
\end{proof}
\begin{rem}
The above proof heavily relies on the mixing property of $\hat{\tau}$.
Note that for $\hat{\tau}$ being only ergodic, the statement is wrong,
as can be seen from the product of two periodic processes in $\T\times\T$
(see Remark \ref{rem:trivial-mixing}). Here, the invariant sets are
given by $I_{A}:=\left\{ \left(\left(y+x\right)\!\mod\T\,,\,x\right)\,:\;y\in A\right\} $
for arbitrary measurable $A\subset\T$.
\end{rem}

\subsection{\label{subsec:Random-measures-and}Random Measures and Palm Theory}

We recall some facts from random measure theory (see \cite{Daley1988})
which will be needed for homogenization. Let $\fM(\Rd)$ \nomenclature[M]{$\fM(\Rd)$}{Measures on $\Rd$ (Section \ref{subsec:Random-measures-and}) }denote
the space of locally bounded Borel measures on $\Rd$ (i.e. bounded
on every bounded Borel-measurable set) equipped with the Vague topology,
which is generated by the sets 
\[
\left\{ \mu\,:\;\int f\,\d\mu\in A\right\} \text{ for every open }A\subset\Rd\text{ and }f\in C_{c}\of{\Rd}\,.
\]
This topology is metrizable, complete and countably generated. A random
measure is a measurable mapping 
\[
\mu_{\bullet}:\;\Omega\to\fM(\Rd)\,,\qquad\omega\mapsto\mu_{\omega}
\]
which is equivalent to both of the following conditions
\begin{enumerate}
\item For every bounded Borel set $A\subset\Rd$ the map $\omega\mapsto\mu_{\omega}(A)$
is measurable
\item For every $f\in C_{c}(\Rd)$ the map $\omega\mapsto\int f\,\d\mu_{\omega}$
is measurable.
\end{enumerate}
A random measure is stationary if the distribution of $\mu_{\omega}(A)$
is invariant under translations of $A$ that is $\mu_{\omega}(A)$
and $\mu_{\omega}(A+x)$ share the same distribution. From stationarity
of $\mu_{\omega}$ one concludes the existence (\cite{heida2011extension,papanicolaou1979boundary}
and references therein) of a dynamical system $\left(\tau_{x}\right)_{x\in\Rd}$
on $\Omega$ such that $\mu_{\omega}\left(A+x\right)=\mu_{\tau_{x}\omega}\left(A\right)$.
By a deep theorem due to Mecke (see \cite{Mecke1967,Daley1988}) the
measure 
\[
\mupalm(A)=\int_{\Omega}\int_{\Rd}g(s)\,\chi_{A}(\tau_{s}\omega)\,\d\muomega(s)\,\d\P(\omega)
\]
can be defined on $\Omega$ for every positive $g\in L^{1}(\Rd)$
with compact support. $\mupalm$ is independent from $g$ and in case
$\mu_{\omega}=\lebesgueL$ we find $\mupalm=\P$. Furthermore, for
every $\borelB(\Rd)\times\borelB(\Omega)$-measurable non negative
or $\mupalm\times\lebesgueL$- integrable functions $f$ the Campbell
formula
\[
\int_{\Omega}\int_{\Rd}f(x,\tau_{x}\omega)\,\d\muomega(x)\,\d\P(\omega)=\int_{\Rd}\int_{\Omega}f(x,\omega)\,\d\mupalm(\omega)\,\d x
\]
 holds. The measure $\mu_{\omega}$ has finite intensity if $\mupalm(\Omega)<+\infty$.

We denote by \nomenclature[Emup]{$\E_{\mupalm}(f|\sI)$}{Expectation of $f$ wrt. $\mupalm$ and the invariant sets, \eqref{eq:invariant-sets-expectation-mup}}
\begin{equation}
\E_{\mupalm}\of{f|\sI}:=\int_{\Omega}f\text{ the expectation of }f\text{ w.r.t. the }\sigma\text{-algebra }\sI\text{ and }\mupalm\,.\label{eq:invariant-sets-expectation-mup}
\end{equation}
For random measures we find a more general version of Theorem \ref{thm:Ergodic-Theorem}.
\begin{thm}[Ergodic Theorem \cite{Daley1988} 12.2.VIII]
\label{thm:Ergodic-Theorem-ran-meas} Let $\left(\Omega,\sF,\P\right)$
be a probability space, $\left(A_{n}\right)_{n\in\N}\subset\Rd$ be
a convex averaging sequence, let $(\tau_{x})_{x\in\Rd}$ be a dynamical
system on $\Omega$ with invariant $\sigma$-algebra $\sI$ and let
$f:\,\Omega\to\R$ be measurable with $\int_{\Omega}\left|f\right|\d\mupalm<\infty$.
Then for $\P$-almost all $\omega\in\Omega$ 
\begin{equation}
\left|A_{n}\right|^{-1}\int_{A_{n}}f\of{\tau_{x}\omega}\,\d\mu_{\omega}(x)\to\E_{\mupalm}\of{f|\sI}\,.\label{eq:ergodic convergence ran meas}
\end{equation}
\end{thm}

Given a bounded open (and convex) set $\bQ\subset\Omega$, it is not
hard to see that the following generalization holds:
\begin{thm}[General Ergodic Theorem]
\label{thm:Ergodic-Theorem-ran-meas-2} Let $\left(\Omega,\sF,\P\right)$
be a probability space, $\bQ\subset\Rd$ be a bounded open set with
$0\in\bQ$, let $(\tau_{x})_{x\in\Rd}$ be a dynamical system on $\Omega$
with invariant $\sigma$-algebra $\sI$ and let $f:\,\Omega\to\R$
be measurable with $\int_{\Omega}\left|f\right|\d\mupalm<\infty$.
Then for $\P$-almost all $\omega\in\Omega$ it holds
\begin{equation}
\forall\varphi\in C_{0}(\bQ):\quad n^{-d}\int_{n\bQ}\varphi(\frac{x}{n})f\of{\tau_{x}\omega}\,\d\mu_{\omega}(x)\to\E_{\mupalm}\of{f|\sI}\int_{\bQ}\varphi\,.\label{eq:ergodic convergence ran meas2}
\end{equation}
\end{thm}

\begin{proof}[Sketch of proof]
 Chose a countable dense family of functions $\varphi\in C_{0}(\bQ)$
that spans $L^{1}(\bQ)$ and that have support on a ball. Use a Cantor
argument and Theorem \ref{thm:Ergodic-Theorem-ran-meas} to prove
the statement for a countable dense family of $C_{0}(\bQ)$. From
here, we conclude by density.

The last result can be used to prove the most general ergodic theorem
which we will use in this work:
\end{proof}
\begin{thm}[General Ergodic Theorem for the Lebesgue measure]
\label{thm:general-lebesgue-ergodic-thm}Let $\left(\Omega,\sF,\P\right)$
be a probability space, $\bQ\subset\Rd$ be a bounded open set with
$0\in\bQ$, let $(\tau_{x})_{x\in\Rd}$ be a dynamical system on $\Omega$
with invariant $\sigma$-algebra $\sI$ and let $f\in L^{p}(\Omega;\mupalm)$
and $\varphi\in L^{q}(\bQ)$, where $1<p,q<\infty$, $\frac{1}{p}+\frac{1}{q}=1$.
Then for $\P$-almost all $\omega\in\Omega$ it holds
\[
n^{-d}\int_{n\bQ}\varphi(\frac{x}{n})f\of{\tau_{x}\omega}\,\d x\to\E\of f\int_{\bQ}\varphi\,.
\]
\end{thm}

\begin{proof}
Let $\varphi_{\delta}\in C(\overline{\bQ})$ with $\norm{\varphi-\varphi_{\delta}}_{L^{q}(\bQ)}<\delta$.
Then
\begin{align*}
 & \left|n^{-d}\int_{n\bQ}\varphi(\frac{x}{n})f\of{\tau_{x}\omega}\,\d x-\E\of f\int_{\bQ}\varphi\right|\\
 & \qquad\qquad\leq\norm{\varphi-\varphi_{\delta}}_{L^{q}(\bQ)}\left(n^{-d}\int_{n\bQ}\left|f\of{\tau_{x}\omega}\right|^{p}\,\d x\right)^{\frac{1}{p}}\\
 & \qquad\qquad\quad+\left|n^{-d}\int_{n\bQ}\varphi_{\delta}(x)f\of{\tau_{x}\omega}\,\d x-\E\of f\int_{\bQ}\varphi_{\delta}\right|+\E_{\mupalm}\of{f|\sI}\int_{\bQ}\left|\varphi-\varphi_{\delta}\right|\,,
\end{align*}
which implies the claim.
\end{proof}

\subsection{Random Sets}

The theory of random measures and the theory of random geometry are
closely related. In what follows, we recapitulate those results that
are important in the context of the theory developed below and shed
some light on the correlations between random sets and random measures.

Let $\closedsets(\Rd)$ denote the set of all closed sets in $\Rd$.
We write \nomenclature[F]{$\closedsets_{V}$, $\closedsets^{K}$, $(\closedsets(\Rd),\ttopology F)$}{(Equations \eqref{eq:closedsets-1}, \eqref{eq:closedsets-2}) }
\begin{eqnarray}
\closedsets_{V}:= & \left\{ F\in\closedsets(\Rd)\,:\;F\cap V\not=\emptyset\right\}  & \mbox{if }V\subset\Rd\quad\textnormal{is an open set}\,,\label{eq:closedsets-1}\\
\closedsets^{K}:= & \left\{ F\in\closedsets(\Rd)\,:\;F\cap K=\emptyset\right\}  & \mbox{if }K\subset\Rd\quad\textnormal{is a compact set}\,.\label{eq:closedsets-2}
\end{eqnarray}
The \emph{Fell-topology} $\ttopology F$ is created by all sets $\closedsets_{V}$
and $\closedsets^{K}$ and the topological space $(\closedsets(\Rd),\ttopology F)$
is compact, Hausdorff and separable\cite{Matheron1975}.
\begin{rem}
\label{rem:char-fell}We find for closed sets $F_{n},F$ in $\Rd$
that $F_{n}\to F$ if and only if \cite{Matheron1975}
\begin{enumerate}
\item for every $x\in F$ there exists $x_{n}\in F_{n}$ such that $x=\lim_{n\to\infty}x_{n}$
and
\item if $F_{n_{k}}$ is a subsequence, then every convergent sequence $x_{n_{k}}$
with $x_{n_{k}}\in F_{n_{k}}$ satisfies $\lim_{k\to\infty}x_{n_{k}}\in F$.
\end{enumerate}
\end{rem}

If we restrict the Fell-topology to the compact sets $\mathfrak{K}(\Rd)$
it is equivalent with the Hausdorff topology given by the Hausdorff
distance
\[
\d\of{A,B}=\max\left\{ \sup_{y\in B}\inf_{x\in A}\left|x-y\right|\,,\,\,\sup_{x\in A}\inf_{y\in B}\left|x-y\right|\right\} \,.
\]

\begin{rem}
For $A\subset\Rd$ closed, the set 
\[
\closedsets(A):=\left\{ F\in\closedsets(\Rd)\,:\;F\subset A\right\} 
\]
 is a closed subspace of $\closedsets\of{\Rd}$. This holds since
\[
\closedsets\of{\Rd}\backslash\closedsets\of A=\left\{ B\in\closedsets\of{\Rd}\,:\;B\cap\left(\Rd\backslash A\right)\not=\emptyset\right\} =\closedsets_{\Rd\backslash A}\quad\text{is open.}
\]
.
\end{rem}

\begin{lem}[Continuity of geometric operations]
\label{lem:contin-geom-Ops}The maps $\tau_{x}:\,A\mapsto A+x$ and
$b_{\delta}:\,A\mapsto\overline{\Ball{\delta}A}$ are continuous in
$\closedsets\left(\Rd\right)$.
\end{lem}

\begin{proof}
We show that preimages of open sets are open. For open sets $V$ we
find
\begin{align*}
\tau_{x}^{-1}\left(\closedsets_{V}\right) & =\left\{ F\in\closedsets(\Rd)\,:\;\tau_{x}F\cap V\not=\emptyset\right\} =\left\{ F\in\closedsets(\Rd)\,:\;F\cap\tau_{-x}V\not=\emptyset\right\} =\closedsets_{\tau_{-x}V}\,,\\
b_{\delta}^{-1}\left(\closedsets_{V}\right) & =\left\{ F\in\closedsets(\Rd)\,:\;\overline{\Ball{\delta}F}\cap V\not=\emptyset\right\} =\left\{ F\in\closedsets(\Rd)\,:\;F\cap\Ball{\delta}V\not=\emptyset\right\} =\closedsets_{\left(b_{\delta}V\right)^{\circ}}\,.
\end{align*}
The calculations for $\tau_{x}^{-1}\of{\closedsets^{K}}=\closedsets^{\tau_{-x}K}$
and $b_{\delta}^{-1}\of{\closedsets^{K}}=\closedsets^{b_{\delta}K}$
are analogue.
\end{proof}
\begin{rem}
\label{rem:The-Matheron--field}The \emph{Matheron-$\sigma$-field
$\sigma_{\closedsets}$} is the Borel-$\sigma$-algebra of the Fell-topology
and is fully characterized either by the class $\closedsets_{V}$
of $\closedsets^{K}$. 
\end{rem}

\begin{defn}[Random closed / open set according to Choquet (see \cite{Matheron1975}
for more details)]
\label{def:RACS}\nomenclature[Random closed sets]{Random closed sets}{(Definition \ref{def:RACS})}
\end{defn}

\begin{enumerate}
\item [a)] Let $(\Omega,\sigma,\P)$ be a probability space. Then a \emph{Random
Closed Set (RACS)} is a measurable mapping 
\[
A:(\Omega,\sigma,\P)\longrightarrow(\closedsets,\sigma_{\closedsets})
\]
\item [b)]Let $\tau_{x}$ be a dynamical system on $\Omega$. A random
closed set is called stationary if its characteristic functions $\chi_{A(\omega)}$
are stationary, i.e. they satisfy $\chi_{A(\omega)}(x)=\chi_{A(\tau_{x}\omega)}(0)$
for almost every $\omega\in\Omega$ for almost all $x\in\Rd$. Two
random sets are jointly stationary if they can be parameterized by
the same probability space such that they are both stationary.
\item [c)] A random closed set $\Gamma:(\Omega,\sigma,P)\longrightarrow(\closedsets,\sigma_{\closedsets})\quad\omega\mapsto\Gamma(\omega)$
is called a \emph{Random closed $C^{k}$-Manifold} if $\Gamma(\omega)$
is a piece-wise $C^{k}$-manifold for P almost every $\omega$.
\item [d)] A measurable mapping 
\[
A:(\Omega,\sigma,\P)\longrightarrow(\closedsets,\sigma_{\closedsets})
\]
is called \emph{Random Open Set (RAOS)} if $\omega\mapsto\Rd\backslash A(\omega)$
is a RACS.
\end{enumerate}
The importance of the concept of random geometries for stochastic
homogenization stems from the following Lemma by Zähle. It states
that every random closed set induces a random measure. Thus, every
stationary RACS induces a stationary random measure.
\begin{lem}[\cite{Zaehle1982} Theorem 2.1.3 resp. Corollary 2.1.5]
\label{lemmazaehlerandommeasure}Let $\closedsets_{m}\subset\closedsets$
be the space of closed m-dimensional sub manifolds of $\Rd$ such
that the corresponding Hausdorff measure is locally finite. Then,
the $\sigma$-algebra $\sigma_{\closedsets}\cap\closedsets_{m}$ is
the smallest such that 
\[
M_{B}:\closedsets_{m}\rightarrow\R\quad M\mapsto\mathcal{H}^{m}(M\cap B)
\]
 is measurable for every measurable and bounded $B\subset\Rd$.
\end{lem}

This means that 
\[
M_{\Rd}:\closedsets_{m}\rightarrow\fM(\Rd)\quad M\mapsto\mathcal{H}^{m}(M\cap\cdot)
\]
is measurable with respect to the $\sigma$-algebra created by the
Vague topology on $\fM(\Rd)$. Hence a random closed set always induces
a random measure. Based on Lemma \ref{lemmazaehlerandommeasure} and
on Palm-theory, the following useful result was obtained in \cite{heida2011extension}
(See Lemma 2.14 and Section 3.1 therein). We can thus assume w.l.o.g
that $\Omega$ is a separable metric space.
\begin{thm}
\label{thm:Main-THM-Sto-Geo}Let $(\Omega,\sigma,P)$ be a probability
space with an ergodic dynamical system $\tau$. Let $A:(\Omega,\sigma,P)\longrightarrow(\closedsets,\sigma_{\closedsets})$
be a stationary random closed $m$-dimensional $C^{k}$-Manifold.

There exists a separable metric space $\tilde{\Omega}\subset\fM\of{\Rd}$
with an ergodic dynamical system $\tilde{\tau}$ and a mapping $\tilde{A}:\,(\tilde{\Omega},\cB_{\tilde{\Omega}},\P)\to(\closedsets,\sigma_{\closedsets})$
such that $A$ and $\tilde{A}$ have the same law and such that $\tilde{A}$
still is stationary. Furthermore, $(x,\omega)\mapsto\tau_{x}\omega$
is continuous. We identify $\tilde{\Omega}=\Omega$, $\tilde{A}=A$
and $\tilde{\tau}=\tau$.
\end{thm}

Also the following result will be useful below.
\begin{lem}
\label{lem:lower-semi-cont-measures}Let $\mu$ be a Radon measure
on $\Rd$ and let $\bQ\subset\Rd$ be a bounded open set. Let $\closedsets_{0}\subset\closedsets\left(\overline{\bQ}\right)$
be such that $\closedsets_{0}\to\R$, $A\mapsto\mu(A)$ is continuous.
Then 
\[
m:\,\closedsets\times\closedsets_{0}\to\fM\of{\Rd}\,,\quad\left(P,B\right)\mapsto\begin{cases}
A\mapsto\mu\of{A\cap B} & B\subset P\\
0 & \text{else}
\end{cases}
\]
is measurable.
\end{lem}

\begin{proof}
For $f\in C_{c}(\Rd)$ we introduce $m_{f}$ through 
\[
m_{f}:\,\left(P,B\right)\mapsto\begin{cases}
\int_{B}f\,\d\mu & B\subset P\\
0 & \text{else}
\end{cases}
\]
and observe that $m$ is measurable if and only if for every $f\in C_{c}\of{\Rd}$
the map $m_{f}$ is measurable (see Section \ref{subsec:Random-measures-and}).
Hence, if we prove the latter property, the lemma is proved.

We assume $f\geq0$ and we show that the mapping $m_{f}$ is even
upper continuous. In particular, let $(P_{n},B_{n})\to(P,B)$ in $\closedsets\times\closedsets_{0}$
and assume that $B_{n}\subset P_{n}$ for all $n>N_{0}$. Since $\overline{\bQ}$
is compact, Remark \ref{rem:char-fell}.~2. implies that $B\subset P\cap\overline{\bQ}$.
Furthermore, since $f$ has compact support, we find $\left|\int_{B_{n}}f\,\d\mu-\int_{B}f\,\d\mu\right|\leq\norm f_{\infty}\left|\mu\of{B_{n}}-\mu\of B\right|\to0$.
On the other hand, if there exists a subsequence such that $B_{n}\not\subset P_{n}$
for all $n$, then either $B\not\subset P$ and $m_{f}(P_{n},B_{n})=0\to m_{f}(P,B)=0$
or $B\subset P$ and $0=\lim_{n\to\infty}m_{f}(P_{n},B_{n})\leq\int_{B}f\d\mu=m_{f}(P,B)$.
For $f\leq0$ we obtain lower semicontinuity and for general $f$
the map $m_{f}$ is the sum of an upper and a lower semicontinuous
map, hence measurable.
\end{proof}

\subsection{Point Processes}
\begin{defn}[(Simple) point processes]
A $\Z$-valued random measure $\mu_{\omega}$ is called point process.
In what follows, we consider the particular case that for almost every
$\omega$ there exist points $\left(x_{k}(\omega)\right)_{k\in\N}$
and values $\left(a_{k}\left(\omega\right)\right)_{k\in\N}$ in $\Z$
such that
\[
\mu_{\omega}=\sum_{k\in\N}a_{k}\delta_{x_{k}(\omega)}\,.
\]
The point process $\mu_{\omega}$ is called simple if almost surely
for all $k\in\N$ it holds $a_{k}\in\left\{ 0,1\right\} $.
\end{defn}

\begin{example}[Poisson process]
 \label{exa:poisson-point-proc}\nomenclature[Poisson Process]{Poinsson process}{(Example \ref{exa:poisson-point-proc})}A
particular example for a stationary point process is the Poisson point
process $\mu_{\omega}=\X_{\omega}$ with intensity $\lambda$. Here,
the probability $\P\of{\X(A)=n}$ to find $n$ points in a Borel-set
$A$ with finite measure is given by a Poisson distribution
\begin{equation}
\P\of{\X(A)=n}=e^{-\lambda\left|A\right|}\frac{\lambda^{n}\left|A\right|^{n}}{n!}\label{eq:PoisonPointPoc-Prob}
\end{equation}
with expectation $\E\of{\X(A)}=\lambda\left|A\right|$. Shift-invariance
of (\ref{eq:PoisonPointPoc-Prob}) implies that the Poisson point
process is stationary.
\end{example}

We can use a given random point process to construct further processes.
\begin{example}[Hard core Matern process]
\label{exa:Matern} \nomenclature[Matern Process]{Matern process}{(Example \ref{exa:Matern}--\ref{exa:Matern-Pois})}The
hard core Matern process is constructed from a given point process
$\X_{\omega}$ by mutually erasing all points with the distance to
the nearest neighbor smaller than a given constant $r$. If the original
process $\X_{\omega}$ is stationary (ergodic), the resulting hard
core process is stationary (ergodic) respectively.
\end{example}

\begin{example}[Hard core Poisson--Matern process]
 \label{exa:Matern-Pois} If a Matern process is constructed from
a Poisson point process, we call it a Poisson--Matern point process.
\end{example}

\begin{lem}
\label{lem:PP-is-RACS}Let $\mu_{\omega}$ be a simple point process
with $a_{k}=1$ almost surely for all $k\in\N$. Then $\X_{\omega}=\left(x_{k}(\omega)\right)_{k\in\N}$
is a random closed set of isolated points with no limit points. On
the other hand, if $\X_{\omega}=\left(x_{k}(\omega)\right)_{k\in\N}$
is a random closed set that almost surely has no limit points then
$\mu_{\omega}$ is a point process.
\end{lem}

\begin{proof}
Let $\mu_{\omega}$ be a point process. For open $V\subset\Rd$ and
compact $K\subset\Rd$ let 
\begin{gather*}
f_{V,R}(x)=\dist\of{\,x,\,\Rd\backslash\left(V\cap\Ball R0\right)\,}\,,\qquad f_{\delta}^{K}(x)=\max{\left\{ \,1-\frac{1}{\delta}\dist\of{x,K}\,,\,0\,\right\} }\,.
\end{gather*}
Then $f_{V,R}$ is Lipschitz with constant $1$ and $f_{\delta}^{K}$
is Lipschitz with constant $\frac{1}{\delta}$ and support in $\Ball{\delta}K$.
Moreover, since $\mu_{\omega}$ is locally bounded, the number of
points $x_{k}$ that lie within $\Ball 1K$ is bounded. In particular,
we obtain
\begin{align*}
\X^{-1}\of{\closedsets_{V}} & =\bigcup_{R>0}\left\{ \omega\,:\;\int_{\Rd}f_{V,R}\,\d\mu_{\omega}>0\right\} \,,\\
\X^{-1}\of{\closedsets^{K}} & =\bigcap_{\delta>0}\left\{ \omega\,:\;\int_{\Rd}f_{\delta}^{K}\,\d\mu_{\omega}>0\right\} \,,
\end{align*}
are measurable. Since $\closedsets_{V}$ and $\closedsets^{K}$ generate
the $\sigma$-algebra on $\closedsets\left(\Rd\right)$, it follows
that $\omega\to\X_{\omega}$ is measurable.

In order to prove the opposite direction, let $\X_{\omega}=\left(x_{k}(\omega)\right)_{k\in\N}$
be a random closed set of points. Since $\X_{\omega}$ has almost
surely no limit points the measure $\mu_{\omega}$ is locally bounded
almost surely. We prove that $\mu_{\omega}$ is a random measure by
showing that 
\[
\forall f\in C_{c}\of{\Rd}\,:\qquad F:\,\omega\mapsto\int_{\Rd}f\,\d\mu_{\omega}\text{ is measurable.}
\]
For $\delta>0$ let $\mu_{\omega}^{\delta}\of A:=\left(\left|\S^{d-1}\right|\delta^{d}\right)^{-1}\lebesgueL\of{A\cap\Ball{\delta}{\X_{\omega}}}$.
By Lemmas \ref{lem:contin-geom-Ops} and \ref{lem:lower-semi-cont-measures}
we obtain that $F_{\delta}:\,\omega\mapsto\int_{\Rd}f\,\d\mu_{\omega}^{\delta}$
are measurable. Moreover, for almost every $\omega$ we find $F_{\delta}\left(\omega\right)\to F\left(\omega\right)$
uniformly and hence $F$ is measurable.
\end{proof}
\begin{cor}
A random simple point process $\mu_{\omega}$ is stationary iff $\X_{\omega}$
is stationary.
\end{cor}

Hence we can provide the following definition based on Definition
\ref{def:RACS}.
\begin{defn}
\label{def:jontly-stat-point-regular}A point process $\mu_{\omega}$
and a random set $\bP$ are jointly stationary if $\bP$ and $\X$
are jointly stationary.
\end{defn}

\begin{lem}
\label{lem:Matern-is-RACS}Let $\X_{\omega}=\left(x_{i}\right)_{i\in\N}$
be a Matern point process from Example \ref{exa:Matern} with distance
$r$ and let for $\delta<\frac{r}{2}$ be $\bB(\omega):=\bigcup_{i}\overline{B_{\delta}(x_{i})}$.
Then $\bB(\omega)$ is a random closed set.
\end{lem}

\begin{proof}
This follows from Lemma \ref{lem:contin-geom-Ops}: $\X_{\omega}$
is measurable and $\X\mapsto\overline{B_{\delta}(\X)}$ is continuous.
Hence $\bB\left(\omega\right)$ is measurable.
\end{proof}

\subsection{\label{subsec:Dynamical-Systems-on-Zd}Dynamical Systems on $\protect\Zd$}
\begin{defn}
\label{def:A-dynamical-system-Zd}Let $\left(\hat{\Omega},\hat{\sF},\hat{\P}\right)$
be a probability space. A discrete dynamical system on $\hat{\Omega}$
is a family $(\hat{\tau}_{z})_{z\in r\Zd}$ of measurable bijective
mappings $\hat{\tau}_{z}:\hat{\Omega}\mapsto\hat{\Omega}$ satisfying
(i)-(iii) of Definition \ref{def:Omega-mu-tau} with $\Rd$ replaced
by $\Zd$. A set $A\subset\hat{\Omega}$ is almost invariant if for
every $z\in r\Zd$ it holds $\P\left(\left(A\cup\hat{\tau}_{z}A\right)\backslash\left(A\cap\hat{\tau}_{z}A\right)\right)=0$
and $\hat{\tau}$ is called ergodic w.r.t. $r\Zd$ if every almost
invariant set has measure $0$ or $1$. 
\end{defn}

Similar to the continuous dynamical systems, also in this discrete
setting an ergodic theorem can be proved.
\begin{thm}[See Krengel and Tempel'man \cite{krengel1985ergodic,tempel1972ergodic}]
\label{thm:Ergodic-Theorem-discrete} Let $\left(A_{n}\right)_{n\in\N}\subset\Rd$
be a convex averaging sequence, let $(\hat{\tau}_{z})_{z\in r\Zd}$
be a dynamical system on $\hat{\Omega}$ with invariant $\sigma$-algebra
$\sI$ and let $f:\,\hat{\Omega}\to\R$ be measurable with $\left|\E(f)\right|<\infty$.
Then for almost all $\hat{\omega}\in\hat{\Omega}$ 
\begin{equation}
\left|A_{n}\right|^{-1}\sum_{z\in A_{n}\cap r\Zd}f\of{\hat{\tau}_{z}\hat{\omega}}\to r^{-d}\E\of{f|\sI}\,.\label{eq:ergodic convergence-discrete}
\end{equation}
\end{thm}

In the following, we restrict to $r=1$ for simplicity of notation.

Let $\Omega_{0}\subset\Rd$. We consider an enumeration $\left(\xi_{i}\right)_{i\in\N}$
of $\Zd$ such that $\hat{\Omega}:=\Omega_{0}^{\Zd}=\Omega_{0}^{\N}$
and write $\hat{\omega}=\left(\hat{\omega}_{\xi_{1}},\hat{\omega}_{\xi_{2}},\dots\right)=\left(\hat{\omega}_{1},\hat{\omega}_{2},\dots\right)$
for all $\hat{\omega}\in\hat{\Omega}$. We define a metric on $\hat{\Omega}$
through 
\[
d(\hat{\omega}_{1},\hat{\omega}_{2})=\sum_{k=1}^{\infty}\frac{1}{2^{k}}\frac{\left|\hat{\omega}_{1,\xi_{k}}-\hat{\omega}_{2,\xi_{k}}\right|}{1+\left|\hat{\omega}_{1,\xi_{k}}-\hat{\omega}_{2,\xi_{k}}\right|}\,.
\]
We write $\Omega_{n}:=\Omega_{0}^{n}$ and $\N_{n}:=\left\{ k\in\N:\,k\geq n+1\right\} $.
The topology of $\hat{\Omega}$ is generated by the open sets $A\times\Omega_{0}^{\N_{n}}$,
where for some $n>0$, $A\subset\Omega_{n}$ is an open set. In case
$\Omega_{0}$ is compact, the space $\hat{\Omega}$ is compact. Further,
$\hat{\Omega}$ is separable in any case since $\Omega_{0}$ is separable
(see \cite{Kelley1955}).
\begin{lem}
Suppose for every $n\in\N$ there exists a probability measure $\P_{n}$
on $\Omega_{n}$ such that for every measurable $A_{n}\subset\Omega_{n}$
it holds $\P_{n+k}\of{A_{n}\times\Omega^{k}}=\P_{n}\of{A_{n}}$. Then
$\P$ defined as follows defines a probability measure on $\Omega$:
\[
\P\of{A_{n}\times\Omega_{0}^{\N_{n}}}:=\P_{n}\of{A_{n}}\,.
\]
\end{lem}

\begin{proof}
We consider the ring
\[
\cR=\bigcup_{n\in\N}\left\{ A\times\Omega_{0}^{\N_{n}}\,:\;A\subset\Omega_{n}\text{ is measurable}\right\} 
\]
and make the observation that $\P$ is additive and positive on $\cR$
and $\P(\emptyset)=0$. Next, let $\left(A_{j}\right)_{j\in\N}$ be
an increasing sequence of sets in $\cR$ such that $A:=\bigcup_{j}A_{j}\in\cR$.
Then, there exists $\tilde{A}_{1}\subset\Omega_{0}^{n}$ such that
$A_{1}=\tilde{A}_{1}\times\Omega_{0}^{\N_{n}}$ and since $A_{1}\subset A_{2}\subset\dots\subset A$,
for every $j>1$, we conclude $A_{j}=\tilde{A_{j}}\times\Omega_{0}^{\N_{n}}$
for some $\tilde{A}_{j}\subset\Omega_{n}$. Therefore, $\P(A_{j})=\P_{n}(\tilde{A}_{j})\to\P_{n}(\tilde{A})=\P(A)$
where $A=\tilde{A}\times\Omega_{0}^{\N_{n}}$. We have thus proved
that $\P:\cR\to[0,1]$ can be extended to a measure on the Borel-$\sigma$-Algebra
on $\Omega$ (See \cite[Theorem 6-2]{Berberian1965}).
\end{proof}
We define for $z\in\Zd$ the mapping
\[
\hat{\tau}_{z}:\,\hat{\Omega}\to\hat{\Omega}\,,\qquad\hat{\omega}\mapsto\hat{\tau}_{z}\hat{\omega}\,,\quad\mbox{where }\left(\hat{\tau}_{z}\hat{\omega}\right)_{\xi_{i}}=\hat{\omega}_{\xi_{i}+z}\mbox{ component wise}\,.
\]

\begin{rem}
\label{rem:standard-omega-0}In this paper, we consider particularly
$\Omega_{0}=\left\{ 0,1\right\} $. Then $\hat{\Omega}:=\Omega_{0}^{\Zd}$
is equivalent to the power set of $\Zd$ and every $\hat{\omega}\in\hat{\Omega}$
is a sequence of $0$ and $1$ corresponding to a subset of $\Zd$.
Shifting the set $\hat{\omega}\subset\Zd$ by $z\in\Zd$ corresponds
to an application of $\hat{\tau}_{z}$ to $\hat{\omega}\in\hat{\Omega}$.
\end{rem}

Now, let $\bP(\omega)$ be a stationary ergodic random open set and
let $r>0$. Recalling (\ref{eq:Pr}) the map $\omega\mapsto\bP_{-r}(\omega)$
is measurable due to Lemma \ref{lem:contin-geom-Ops} and we can define
$\X_{r}\of{\bP\of{\omega}}:=2r\Zd\cap\bP_{-\frac{r}{2}}\of{\omega}$.
\begin{lem}
\label{lem:X-r-stationary}If $\bP$ is a stationary ergodic random
open set then the set\nomenclature[Xr]{$\X_r(\omega)=\X_r(\bP(\omega))$}{$=2r\Zd\cap \bP_{-r}(\omega)$, \eqref{eq:def-X_r}}
\begin{equation}
\X=\X_{r}(\omega):=\X_{r}\of{\bP\of{\omega}}:=2r\Zd\cap\bP_{-r}\of{\omega}\label{eq:def-X_r}
\end{equation}
 is a stationary random point process w.r.t. $2r\Zd$.
\end{lem}

\begin{proof}
By a simple scaling we can w.l.o.g. assume $2r=1$ and write $\X=\X_{r}$.
Evidently, $\X$ corresponds to a process on $\Zd$ with values in
$\Omega_{0}=\left\{ 0,1\right\} $ writing $\X(z)=1$ if $z\in\X$
and $\X(z)=0$ if $z\not\in\X$. In particular, we write $\left(\omega,z\right)\mapsto\X\of{\omega,z}$.
This process is stationary as  the shift invariance of $\bP$ induces
a shift-invariance of $\hat{\P}$ with respect to $\hat{\tau}_{z}$.
It remains to observe that the probabilities $\P\of{\X(z)=1}$ and
$\P\of{\X(z)=0}$ induce a random measure on $\hat{\Omega}$ in the
way described in Remark \ref{rem:standard-omega-0}. 
\end{proof}
\begin{rem}
If $\bP$ is mixing one can follow the lines of the proof of Lemma
\ref{lem:erg-and-mix-is-erg} to find that $\X_{r}\of{\bP\of{\omega}}$
is ergodic. However, in the general case $\X_{r}\of{\bP\of{\omega}}$
is not ergodic. This is due to the fact that by nature $\left(\tau_{z}\right)_{z\in\Zd}$
on $\Omega$ has more invariant sets than$\left(\tau_{x}\right)_{x\in\Rd}$.
For sufficiently complex geometries the map $\Omega\to\hat{\Omega}$
is onto.
\end{rem}

\begin{defn}[Jointly stationary]
\label{def:jointly-staionary-points}\nomenclature[jointly stationary]{jointly stationary}{Definitions \ref{def:RACS}, \ref{def:jontly-stat-point-regular} and \ref{def:jointly-staionary-points}}We
call a point process $\X$ with values in $2r\Zd$ to be strongly
jointly stationary with a random set $\bP$ if the functions $\chi_{\bP(\omega)}$,
$\chi_{\X(\omega)}$ are jointly stationary w.r.t. the dynamical system
$\left(\tau_{2rx}\right)_{x\in\Zd}$ on $\Omega$.
\end{defn}

\section{\label{sec:Nonlocal-regularity}Quantifying Nonlocal Regularity Properties
of the Geometry}

\subsection{\label{subsec:Microscopic-Regularity}Microscopic Regularity}

\begin{figure}
\centering{}\includegraphics[width=4cm]{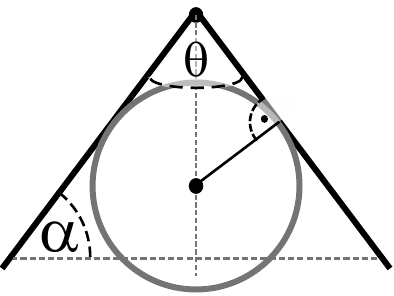}\caption{\label{fig:Ball-and-Cone}How to fit a ball into a cone.}
\end{figure}

\begin{lem}
\label{lem:small-ball-in-cone}Let $\bP$ be a Lipschitz domain. Then
for every $p_{0}\in\partial\bP$ with $\delta(p_{0})>0$ the following
holds: For every $\delta<\delta\left(p_{0}\right)$ and $M:=M_{\delta}(p_{0})>0$
there exists $y\in\bP$ with $\left|p_{0}-y\right|=\frac{\delta}{4}$
such that with $r\left(p_{0}\right):=\frac{\delta}{4\left(1+M\right)}$
it holds $\Ball{r\left(p_{0}\right)}y\subset\Ball{\delta/2}{p_{0}}$.
\end{lem}

\begin{proof}
We can assume that $\partial\bP$ is locally a cone as in Figure \ref{fig:Ball-and-Cone}.
With regard to Figure \ref{fig:Ball-and-Cone}, for $p_{0}\in\partial\bP$
with $\delta$ and $M$ as in the statement we can place a right circular
cone with vertex (apex) $p_{0}$ and axis $\nu$ and an aperture $\theta=\pi-2\arctan M$
inside $\Ball{\delta}{p_{0}}$, where $\alpha=\arctan M$. In other
words, it holds $\tan\left(\alpha\right)=\tan\left(\frac{\pi-\theta}{2}\right)=M$.
Along the axis we may select $y$ with $\left|p_{0}-y\right|=\frac{\delta}{4}$.
Then the distance $R$ of $y$ to the cone is given through 
\[
\left|y-p_{0}\right|^{2}=R^{2}+R^{2}\tan^{2}\left(\frac{\pi-\theta}{2}\right)\quad\Rightarrow\quad R=\frac{\left|y-p_{0}\right|}{\sqrt{1+M^{2}}}\,.
\]
In particular $r\left(p_{0}\right)$ as defined above satisfies the
claim.
\end{proof}

\subsubsection*{Continuity properties of $\delta$, $M$ and $\varrho$}
\begin{lem}
\label{lem:properties-delta-M-regular}Let $\fr>0$, $\bP$ be a Lipschitz
domain and recall (\ref{eq:cover-delta-1}). Then $\partial\bP$ is
$\delta_{\Delta}$-regular in the sense of Definition \ref{def:eta-regular}.
In particular, $\delta_{\Delta}:\,\partial\bP\to\R$ is locally Lipschitz
continuous with Lipschitz constant $4$ and for every $\eps\in\left(0,\frac{1}{2}\right)$
and $\tilde{p}\in\Ball{\eps\delta}p\cap\partial\bP$ it holds 
\begin{equation}
\frac{1-\eps}{1-2\eps}\delta_{\Delta}\of p>\delta_{\Delta}\of{\tilde{p}}>\delta_{\Delta}\of p-\left|p-\tilde{p}\right|>\left(1-\eps\right)\delta_{\Delta}\of p\,.\label{eq:lem:properties-delta-M-regular-3}
\end{equation}
\end{lem}

\begin{rem}
\label{rem:no-global-lipschitz}The latter lemma does \textbf{\emph{not
}}imply global Lipschitz regularity of $\delta_{\Delta}$. It could
be that $2\delta_{\Delta}\of p<\left|p-\tilde{p}\right|<3\delta_{\Delta}\of p$
and $p$ and $\tilde{p}$ are connected by a path inside $\partial\bP$
with the shortest path of length $10\delta_{\Delta}\of p$. Then Lemma
\ref{lem:properties-delta-M-regular} would have to be applied successively
along this path yielding an estimate of $\left|\delta_{\Delta}\of p-\delta_{\Delta}\of{\tilde{p}}\right|\leq40\left|p-\tilde{p}\right|$.
\end{rem}

\begin{proof}[Proof of Lemma \ref{lem:properties-delta-M-regular}]
 It is straight forward to verify that $\left|p-\tilde{p}\right|<\eps\delta_{\Delta}(p)$
implies $\delta_{\Delta}(\tilde{p})>(1-\eps)\delta_{\Delta}(p)$ and
we conclude with Lemma \ref{lem:eta-lipschitz}.
\end{proof}
With regard to Lemma \ref{lem:uniform-extension-lemma}, the relevant
quantity for local extension operators is related to $\delta\of p/\sqrt{4M\of p^{2}+2}$,
where $M(p)$ is the related Lipschitz constant. While we can quantify
$\delta\of p$ in terms of $\delta\of{\tilde{p}}$ and $\left|p-\tilde{p}\right|$,
this does not work for $M\of p$. Hence we cannot quantify $\delta\of p/\sqrt{4M\of p^{2}+2}$
in terms of its neighbors. This drawback is compensated by a variational
trick in the following statement.
\begin{lem}
\label{lem:rho-p-lsc} Let $\bP$ be a Lipschitz domain and let $\delta\leq\delta_{\Delta}$
satisfy (\ref{eq:lem:properties-delta-M-regular-3}) such that $\partial\bP$
is $\delta$-regular. For $p\in\partial\bP$ and let $M_{r}\of p$
be given in (\ref{eq:cover-delta-M}) and define for $n,K\in\N$ \nomenclature[rho]{$\hat\rho(p)$}{$\inf\left\{ \delta\leq\delta(p)\,:\;\sup_{r<\delta}r\sqrt{4M_{r}\of p^{2}+2}^{-1}=\rho\right\} $ \eqref{eq:def-rhohat-of-p}}\nomenclature[rho]{$\rho(p)$}{$\sup_{r<\delta\of p}r\sqrt{4M_{r}\of p^{2}+2}^{-1}$ \eqref{eq:def-rho-of-p}}
\begin{align}
\rho_{n}{\left(p\right)} & :=\sup_{r<\delta\of p}r\sqrt{4M_{r}\of p^{2}+2}^{-n}\,,\label{eq:def-rho-of-p}\\
\hat{\rho}_{n,K}(p) & :=\inf\left\{ \delta\leq\delta(p)\,:\;\sup_{r<2^{-K}\delta}r\sqrt{4M_{r}\of p^{2}+2}^{-n}\geq2^{-K}\rho_{n}(p)\right\} \,.\label{eq:def-rhohat-of-p}
\end{align}
Then for fixed $p\in\partial\bP$ the functions $r\mapsto M_{r}(p)$
is right continuous and monotone increasing (i.e. u.s.c.). Furthermore,
$\rho_{n}$ is positive and locally Lipschitz continuous on $\partial\bP$
with Lipschitz constant $4$ and $\partial\bP$ is $\rho$-regular
in the sense of Definition \ref{def:eta-regular}. In particular,
for $\left|p-\tilde{p}\right|<\eps\rho_{n}{\left(p\right)}$ it holds
\begin{gather}
\frac{1-\eps}{1-2\eps}\rho_{n}\of p>\rho_{n}\of{\tilde{p}}>\rho_{n}\of p-\left|p-\tilde{p}\right|>\left(1-\eps\right)\rho_{n}\of p\,.\label{eq:rho-neighbor-estim}
\end{gather}
Furthermore, $\hat{\rho}_{n,K}\leq\delta$ is well defined.
\end{lem}

\begin{rem}
Like in Remark \ref{rem:no-global-lipschitz} this does \textbf{\emph{not
}}imply global Lipschitz regularity of $\rho_{n}$ or $\hat{\rho}_{n}$.
\end{rem}

\begin{cor}
\label{cor:maximal-extension-order}Every Lipschitz domain $\bP$
has extension order $1$ and symmetric extension order $2$.
\end{cor}

\begin{proof}
This follows from $\hat{\rho}_{n,3}\leq\delta$ and Lemmas \ref{lem:uniform-extension-lemma}
and \ref{lem:uniform-Nitsche-extension-lemma} applied to $\Ball{\frac{1}{8}\hat{\rho}_{n,3}}{p_{0}}$
and $\Ball{\frac{1}{8}\rho_{n}}{p_{0}}$.
\end{proof}

\begin{proof}[Proof of Lemma \ref{lem:rho-p-lsc}]
Right continuity of $r\mapsto M_{r}(p)$ follows because for every
$0<r<R$ because $M$ being Lipschitz constant of $\partial\bP$ in
$\Ball Rp$ implies $M$ being Lipschitz constant of $\partial\bP$
in $\Ball rp$. 

Let $\left|p-\tilde{p}\right|<\eps\rho{\left(p\right)}<\eps\delta(p)$
implying $\delta{\left(\tilde{p}\right)}\geq\left(1-\eps\right)\delta{\left(p\right)}$
by Lemma \ref{lem:properties-delta-M-regular}. For every $\eta>0$
let $r_{\eta}\in\left(\rho(p),\delta\of p\right)$ such that $\rho{\left(p\right)}\leq\left(1+\eta\right)r_{\eta}\sqrt{4M_{r_{\eta}}{\left(p\right)}^{2}+2}^{-n}$.
Since $r_{\eta}>\rho{\left(p\right)}$ and $\left|p-\tilde{p}\right|<\eps\rho{\left(p\right)}$
we find $\Ball{r_{\eta}}p\supset\Ball{\left(1-\eps\right)r_{\eta}}{\tilde{p}}$
 and hence $M_{\left(1-\eps\right)r_{\eta}}{\left(\tilde{p}\right)}\leq M_{r_{\eta}}{\left(p\right)}$.
This implies at the same time that $\partial\bP$ is $\rho$-regular
and that 
\[
\rho{\left(\tilde{p}\right)}\geq\frac{\left(1-\eps\right)r_{\eta}}{\sqrt{4M_{\left(1-\eps\right)r_{\eta}}{\left(\tilde{p}\right)}^{2}+2}^{n}}\geq\frac{\left(1-\eps\right)r_{\eta}}{\sqrt{4M_{r_{\eta}}{\left(p\right)}^{2}+2}^{n}}\geq\frac{\left(1-\eps\right)}{\left(1+\eta\right)}\rho{\left(p\right)}\,.
\]
Since $\eta$ was arbitrary, we conclude $\rho{\left(\tilde{p}\right)}\geq\left(1-\eps\right)\rho{\left(p\right)}$.
Moreover, we find $\left|p-\tilde{p}\right|<\frac{\eps}{1-\eps}\rho{\left(\tilde{p}\right)}$.
And we conclude the first part with Lemma \ref{lem:eta-lipschitz}.

Second, it holds for every $r<\delta$ and $\eps\in(0,1)$ that 
\[
\eps r\sqrt{4M_{r}\of p^{2}+2}^{-n}\leq\eps r\sqrt{4M_{\eps r}\of p^{2}+2}^{-n}
\]
and choosing $\eps=2^{-K}$ and taking the supremum on both sides,
we infer $\hat{\rho}_{n,K}\leq\delta$.
\end{proof}

\begin{cor}
\label{cor:cover-boundary}Let $\fr>0$ and let $\bP\subset\Rd$ be
a locally $\left(\delta,M\right)$-regular open set, where we restrict
$\delta$ by $\delta\left(\cdot\right)\leq\frac{\fr}{4}$. Then there
exists a countable number of points $\left(p_{k}\right)_{k\in\N}\subset\partial\bP$
such that $\partial\bP$ is completely covered by balls $\Ball{\tilde{\rho}\left(p_{k}\right)}{p_{k}}$
where $\tilde{\rho}\left(p\right):=2^{-5}\rho_{n}\left(p\right)$
for some $n\in\N$. Writing 
\[
\tilde{\rho}_{k}:=\tilde{\rho}\of{p_{k}}\,,\qquad\delta_{k}:=\delta\of{p_{k}}\,.
\]
For two such balls with $\Ball{\tilde{\rho}_{k}}{p_{k}}\cap\Ball{\tilde{\rho}_{i}}{p_{i}}\neq\emptyset$
it holds 
\begin{equation}
\begin{aligned} & \frac{15}{16}\tilde{\rho}_{i}\leq\tilde{\rho}_{k}\leq\frac{16}{15}\tilde{\rho}_{i}\\
\text{and}\quad & \frac{31}{15}\min\left\{ \tilde{\rho}_{i},\tilde{\rho}_{k}\right\} \geq\left|p_{i}-p_{k}\right|\geq\frac{1}{2}\max\left\{ \tilde{\rho}_{i},\tilde{\rho}_{k}\right\} \,.
\end{aligned}
\label{eq:cor:cover-boundary-h1}
\end{equation}
Furthermore, there exists $\fr_{k}\geq\frac{\tilde{\rho}_{k}}{32\left(1+M_{\tilde{\rho}(p_{k})}\of{p_{k}}\right)}$
and $y_{k}$ such that $\Ball{\fr_{k}}{y_{k}}\subset\Ball{\tilde{\rho}_{k}/8}{p_{k}}\cap\bP$
and $\Ball{2\fr_{k}}{y_{k}}\cap\Ball{2\fr_{j}}{y_{j}}=\emptyset$
for $k\neq j$.
\end{cor}

\begin{proof}
The existence of the points and Balls satisfying (\ref{eq:cor:cover-boundary-h1})
follows from Theorem \ref{thm:delta-M-rho-covering}, in particular
(\ref{eq:thm:delta-M-rho-covering}). It holds for $\Ball{\tilde{\rho}_{k}}{p_{k}}\cap\Ball{\tilde{\rho}_{i}}{p_{i}}\neq\emptyset$
\[
\left|p_{i}-p_{k}\right|\leq\tilde{\rho_{i}}+\tilde{\rho}_{k}\leq\left(\frac{16}{15}+1\right)\tilde{\rho}_{i}\,.
\]
Lemma \ref{lem:small-ball-in-cone} yields existence of $y_{k}$ such
that $\Ball{\fr_{k}}{y_{k}}\subset\Ball{\tilde{\rho}_{k}/8}{p_{k}}\cap\bP$.
The latter implies $\Ball{\fr_{k}}{y_{k}}\cap\Ball{\fr_{j}}{y_{j}}=\emptyset$
for $k\neq j$.
\end{proof}

\begin{lem}
\label{lem:M-eta}Let $\fr>0$, $\bP\subset\Rd$ be a locally $\left(\delta,M\right)$-regular
open set and let $M_{0}\in(0,+\infty]$ such that for every $p\in\partial\bP$
there exists $\delta>0$, $M<M_{0}$ such that $\partial\bP$ is $(\delta,M)$-regular
in $p$. For $\alpha\in(0,1]$ let $\eta(p)=\alpha\delta_{\Delta}(p)$
from Lemma \ref{lem:properties-delta-M-regular} or $\eta(p)=\alpha\rho_{n}(p)$
from Lemma \ref{lem:rho-p-lsc} and define
\begin{align}
\rmM_{[\eta]}{\left(p\right)} & :=\inf_{\delta>\eta{\left(p\right)}}\inf_{M}\left\{ \,\bP\text{ is }\left(\delta,M\right)\text{-regular in }p\right\} \,.\label{eq:lem:M-eta-1}
\end{align}
Then, for fixed $\xi$, $\rmM_{[\eta]}\of{\cdot}:\;\partial\bP\to\R$
is upper semicontinuous and on each bounded measurable set $A\subset\Rd$
the quantity 
\begin{equation}
\rmM_{[\eta]}(A):=\sup_{p\in\overline{A}\cap\partial\bP}M_{[\eta]}\of p\label{eq:def-M-set-A}
\end{equation}
 \nomenclature[MA]{$M_{[\eta]},\,M_{[\eta],A}$ }{($A$ a set) Equation \eqref{eq:def-M-set-A}, a quantity on $\partial\bP$}\nomenclature[m]{$\fm_{[\eta]}{\left(p,\xi\right)}$}{Lemma \ref{lem:M-eta}}with
$\rmM_{[\eta]}(A)=0$ if $\overline{A}\cap\partial\bP=\emptyset$
is well defined. The functions 
\[
\rmM_{[\eta]}(A,\cdot):\;\Rd\to\R\,,\quad\rmM_{[\eta]}\of{A,x}:=\rmM_{[\eta]}(A+x)\quad\text{with }\rmM_{[\eta]}(A,0)=\rmM_{[\eta]}(A)
\]
are upper semicontinuous.
\end{lem}

\begin{rem}
\label{rem:difference-M-eta-M-eta}Note at this point that $M_{[\eta,r],\Rd}$
defined in (\ref{eq:lem:local-delta-M-construction-estimate-1b})
is a function on $\Rd$ and different from $M_{[\eta]}$. 
\end{rem}

\begin{notation}
The infimum in (\ref{eq:lem:M-eta-1}) is a $\liminf$ for $\delta\searrow\eta(p)$.
We sometimes use the special notation .......................???????????????????????.................
\[
M_{[\eta],\fr}(x):=M_{[\eta],\Ball{\fr}0}(x)\,.
\]
\end{notation}

\begin{proof}[Proof of Lemma \ref{lem:M-eta}]
Let $p,\tilde{p}\in\partial\bP$ with $\left|p-\tilde{p}\right|<\eps\eta(p)$.
Writing $\tilde{\eps}:=\frac{\eps}{1-\eps}$ and $r\left(p,\eps\right):=\left(\frac{1}{1-2\eps}+\eps\right)\eta\of p$
and 
\[
M\of{p,\eps}:=\inf_{M}\left\{ \Ball{r\left(p,\eps\right)}p\cap\partial\bP\text{ is }M\text{-Lipschitz graph}\right\} 
\]
as well as we observe from $\eta$-regularity that $\Ball{\eta\left(\tilde{p}\right)}{\tilde{p}}\subset\Ball{r\left(p,\eps\right)}p$
and $\Ball{\eta\left(p\right)}p\subset\Ball{r\left(\tilde{p},\tilde{\eps}\right)}{\tilde{p}}$.
Hence we find 
\[
\rmM_{[\eta]}\of{\tilde{p}}\leq M\of{p,\eps}\,.
\]
Observing that $M\of{p,\eps}\searrow\rmM_{[\eta]}\of p$ as $\eps\to0$
we find $\limsup_{\tilde{p}\to p}\rmM_{[\eta]}\of{\tilde{p}}\leq\rmM_{[\eta]}\of p$
and $\rmM$ is u.s.c.

Let $x\to0$. First observe that $\rmM_{[\eta]}(A)=\max_{y\in\overline{A}}\rmM_{[\eta]}\of y$.
The set $\overline{A}$ is compact and hence $\overline{A}+x\to\overline{A}$
in the Hausdorff metric as $x\to0$. Let $y_{x}\in\overline{A}+x$
such that $\rmM_{[\eta]}\of{y_{x}}=\rmM_{[\eta]}\left(A,x\right)$.
Since $\overline{A}+x\to\overline{A}$ w.l.o.g. we find $y_{x}\to y$
converges and $y\in\overline{A}$. Hence 
\[
\rmM_{[\eta]}\of y\geq\limsup_{x\to0}\rmM_{[\eta]}\of{y_{x}}=\limsup_{x\to0}\rmM_{[\eta]}\of{A,x}\,.
\]
In particular, $M_{[\eta],A}\of{\cdot}$ is u.s.c. The u.s.c of $\fm_{[\eta]}{\left(p,\xi\right)}$
can be proved similarly.
\end{proof}

\subsubsection*{Measurability and Integrability of Extended Variables}
\begin{lem}
\label{lem:delta-rho-M-measurable}Let $\fr>0$, let $\bP\subset\Rd$
be a Lipschitz domain and let $\eta,r:\,\partial\bP\to\R$ be continuous
such that $\eta\leq\fr$ and $\bP$ is $\eta$- and $r$-regular.
For $\eps\in(0,1]$ let $\eta(p)=\eps\delta(p)$ from Lemma \ref{lem:properties-delta-M-regular}
or $\eta(p)=\eps\rho_{n}(p)$, $n\in\N$, from Lemma \ref{lem:rho-p-lsc}.
Then $\eta_{[r],\Rd}$ from (\ref{eq:lem:local-delta-M-construction-estimate-1})
is measurable and $M_{[\eta,r],\Rd}$ from (\ref{eq:lem:local-delta-M-construction-estimate-1b})
is upper semicontinuous.
\end{lem}

In what follows, we write $A_{\eta,\fr}:=F^{-1}\of{(0,\frac{3}{2}\fr)}$
for 
\[
F:=\inf_{p\in\partial\bP}f_{p}\,,\qquad f_{p}\of x:=\begin{cases}
\eta\of p & \text{if }x\in\Ball{r(p)}p\\
2\fr & \text{else}
\end{cases}\,.
\]

\begin{proof}
Step 1: We write $A=A_{\eta,\fr}$ for simplicity. Let $\left(p_{i}\right)_{i\in\N}\subset\partial\bP$
be a dense subset. If $x\in\Ball{r\of p}p$ for some $p\in\partial\bP$
then also $x\in\Ball{r\of{\tilde{p}}}{\tilde{p}}$ for $\left|p-\tilde{p}\right|$
sufficiently small, by continuity of $\eta$. Hence every $f_{p}$
is upper semicontinuous and it holds $F=\inf_{i\in\N}f_{p_{i}}$.
In particular, $F$ is measurable and so is the set $A$. This implies
$\eta_{[r],\Rd}=\chi_{A}F$ is measurable.

Step 2: We show that for every $a\in\R$ the preimage $M_{[\eta,r],\Rd}^{-1}\of{[a,+\infty)}$
is closed. Let $\left(x_{k}\right)_{k\in\N}$ be a sequence with $M_{[\eta,r],\Rd}\of{x_{k}}\in[a,+\infty)$.
Let $\left(p_{k}\right)\subset\partial\bP$ be a sequence with $\left|x_{k}-p_{k}\right|\leq r\of{p_{k}}$.
W.l.o.g. assume $p_{k}\to p\in\partial\bP$ and $x_{k}\to x\in\Rd$.
Since $r$ is continuous, it follows $\left|x-p\right|\leq r\of p$.
On the other hand $M_{[\eta]}(p)\geq\limsup_{k\to\infty}M_{[\eta]}(p_{k})$
and thus $M_{[\eta,r],\Rd}\of x\geq M_{[\eta,r]}(p)\geq a$.
\end{proof}
\begin{lem}
\label{lem:delta-tilde-construction-estimate}Under the assumptions
of Lemma \ref{lem:delta-rho-M-measurable} let $\tilde{\eta}:=\eta_{[\frac{\eta}{8}],\Rd}$.
Then there exists a constant $C>0$ only depending on the dimension
$d$ such that for every bounded open domain $\bQ$ and $k\in[0,4)$
it holds
\begin{align}
\int_{A_{\eta,\fr}\cap\bQ}\chi_{\tilde{\eta}>0}\tilde{\eta}^{-\alpha} & \leq C\int_{\Ball{\frac{\fr}{4}}{\bQ}\cap\partial\bP}\eta^{1-\alpha}M_{[\frac{\eta}{4}],\Rd}^{d-2}\,,\label{eq:lem:delta-tilde-construction-estimate-1}\\
\int_{A_{\eta,\fr}\cap\bQ}\tilde{\eta}^{-\alpha}M_{[k\frac{\eta}{8},\frac{\eta}{8}],\Rd}^{r} & \leq C\int_{\Ball{\frac{\fr}{4}}{\bQ}\cap\partial\bP}\eta^{1-\alpha}M_{[k\frac{\eta}{8},\frac{\eta}{4}],\Rd}^{r}M_{[\frac{\eta}{4}],\Rd}^{d-2}\,.\label{eq:lem:delta-tilde-construction-estimate-2}
\end{align}
Finally, it holds 
\begin{equation}
x\in\Ball{\frac{1}{8}\eta\of p}p\quad\Rightarrow\quad\eta\of p>\tilde{\eta}\of x>\frac{3}{4}\eta\of p\,.\label{eq:lem:local-delta-M-construction-estimate-2}
\end{equation}
\end{lem}

\begin{rem}
Estimates (\ref{eq:lem:delta-tilde-construction-estimate-1})--(\ref{eq:lem:delta-tilde-construction-estimate-2})
are only rough estimates and better results could be obtained via
more sophisticated calculations that make use of particular features
of given geometries.
\end{rem}

\begin{proof}
We write $A=A_{\eta,\fr}$ for simplicity. Step 1: Given $x\in\Rd$
with $\tilde{\eta}(x)>0$ let 
\begin{equation}
p_{x}\in\argmin\left\{ \eta\of{\tilde{x}}\,:\;\tilde{x}\in\partial\bP\,\text{s.t. }x\in\overline{\Ball{\frac{1}{8}\eta\of{\tilde{x}}}{\tilde{x}}}\right\} \,.\label{eq:lem:delta-tilde-construction-estimate-h-1}
\end{equation}
Such $p_{x}$ exists because $\partial\bP$ is locally compact. We
observe with help of the definition of $p_{x}$, the triangle inequality
and (\ref{eq:thm:delta-M-rho-covering-a})
\[
x\in\Ball{\frac{1}{8}\eta\of p}p\quad\Rightarrow\quad\eta\of{p_{x}}\leq\eta\of p\quad\Rightarrow\quad\left|p-p_{x}\right|<\frac{\eta\of p}{4}\quad\Rightarrow\quad\eta\of{p_{x}}>\frac{3}{4}\eta\of p\,.
\]
The last line particularly implies (\ref{eq:lem:local-delta-M-construction-estimate-2})
and 
\[
\forall p\in\partial\bP\;\forall x\in\Ball{\frac{\eta\of p}{8}}p:\quad\tilde{\eta}\of x>\frac{3\eta\of p}{4}\,.
\]
Step 2: By Theorem \ref{thm:delta-M-rho-covering} we can chose a
countable number of points $\left(p_{k}\right)_{k\in\N}\subset\partial\bP$
such that $\Gamma=\partial\bP$ is completely covered by balls $B_{k}:=\Ball{\xi\of{p_{k}}}{p_{k}}$
where $\xi\of p:=2^{-4}\eta\of p$. For simplicity of notation we
write $\eta_{k}:=\eta\of{p_{k}}$ and $\xi_{k}:=\xi\of{p_{k}}$. Assume
$x\in A$ with $p_{x}\in\Gamma$ given by (\ref{eq:lem:delta-tilde-construction-estimate-h-1}).
Since the balls $B_{k}$ cover $\Gamma$, there exists $p_{k}$ with
$\left|p_{x}-p_{k}\right|<\xi_{k}=2^{-4}\eta_{k}$, implying $\eta\of{p_{x}}<\frac{2^{4}}{2^{4}-1}\eta_{k}$
and hence $\left|x-p_{k}\right|\leq\left(2^{-4}+\frac{2^{-3}2^{4}}{2^{4}-1}\right)\eta_{k}<\frac{3}{16}\eta_{k}$.
Hence we find
\[
\forall x\in A\,\;\exists p_{k}\;:\quad x\in\Ball{\frac{3}{16}\eta_{k}}{p_{k}}\,.
\]
Step 3: For $p\in\Gamma$ with $x\in\Ball{\frac{1}{4}\eta\of p}p\cap\Ball{\frac{1}{8}\eta\of{p_{x}}}{p_{x}}$
we can distinguish two cases:
\begin{enumerate}
\item $\eta\of p\geq\eta\of{p_{x}}$: Then $p_{x}\in\Ball{\frac{3}{8}\eta\of p}p$
and hence $\eta\of{p_{x}}\ge\frac{5}{8}\eta\of p$ by (\ref{eq:thm:delta-M-rho-covering-a}).
\item $\eta\of p<\eta\of{p_{x}}$: Then $p\in\Ball{\frac{3}{8}\eta\of{p_{x}}}{p_{x}}$
and hence$\eta\of{p_{x}}>\frac{1-\frac{3}{8}}{1-\frac{6}{8}}\eta\of p=\frac{5}{2}\eta\of p$
by (\ref{eq:thm:delta-M-rho-covering-a}).
\end{enumerate}
and hence 
\[
x\in\Ball{\frac{1}{4}\eta\of p}p\qquad\Rightarrow\qquad\tilde{\eta}\of x=\eta\of{p_{x}}>\frac{5}{8}\eta\of p\,.
\]
Step 4:  Let $k\in\N$ be fixed and define $B_{k}=\Ball{\frac{1}{4}\eta_{k}}{p_{k}}$,
$M_{k}:=M_{\frac{1}{4}\eta_{k}}(p_{k})$. By construction, every $B_{j}$
with $B_{j}\cap B_{k}\neq\emptyset$ satisfies $\eta_{j}\geq\frac{1}{2}\eta_{k}$
and hence if $B_{j}\cap B_{k}\neq\emptyset$ and $B_{i}\cap B_{j}\neq\emptyset$
we find $\left|p_{j}-p_{i}\right|\geq\frac{1}{4}\eta_{k}$ and $\left|p_{j}-p_{k}\right|\leq3\eta_{k}$.
This implies that 
\[
\exists C>0:\;\forall k\quad\#\left\{ j\,:\;B_{j}\cap B_{k}\neq\emptyset\right\} \leq C\,.
\]
We further observe that the minimal surface of $B_{k}\cap\partial\bP$
is given in case when $B_{k}\cap\partial\bP$ is a cone with opening
angle $\frac{\pi}{2}-\arctan M\of{p_{k}}$. The surface area of $B_{k}\cap\partial\bP$
in this case is bounded by $\frac{1}{d-1}\left|\S^{d-2}\right|\eta_{k}^{d-1}\left(M_{k}+1\right)^{2-d}$.
This particularly implies up to a constant independent from $k$:
\begin{align*}
\int_{A\cap\bQ\cap\bP}\tilde{\eta}^{-\alpha} & \lesssim\sum_{k:\,B_{k}\cap\bQ\neq\emptyset}\int_{A\cap B_{k}\cap\bP}\eta_{k}^{-\alpha}\\
 & \lesssim\sum_{k:\,B_{k}\cap\bQ\neq\emptyset}\int_{A\cap B_{k}\cap\partial\bP}\eta^{1-\alpha}M_{[\frac{\eta}{4}]}^{d-2}\\
 & \lesssim\int_{A\cap\Ball{\frac{\fr}{4}}{\bQ}\cap\partial\bP}\eta^{1-\alpha}M_{[\frac{\eta}{4}]}^{d-2}\,.
\end{align*}
The second integral formula follows in a similar way.
\end{proof}

\subsection{\label{subsec:Mesoscopic-Regularity}Mesoscopic Regularity and Isotropic
Cone Mixing}
\begin{lem}
\label{lem:Alway-mesoscopic-regular}Let $\bP(\omega)$ be a stationary
and ergodic random open set such that 
\[
\P{\left(\bP\cap\I=\emptyset\right)}<1\,.
\]
Then there exists $\fr>0$ and a positive, monotonically decreasing
function $\tilde{f}$ such that almost surely $\bP(\omega)$ is $(\fr,\tilde{f})$-mesoscopic
regular. 
\end{lem}

\begin{proof}
Step 1: For some $\fr>0$ and with positive probability $p_{\fr}>0$
the set $(0,1)^{d}\cap\bP$ contains a ball with radius $5\sqrt{d}\fr$.
Otherwise, for every $r>0$ the set $(0,1)^{d}\cap\bP$ almost surely
does not contain an open ball with radius $r$. In particular with
probability $1$ the set $(0,1)^{d}\cap\bP$ does not contain any
ball. Hence $(0,1)^{d}\cap\bP=\emptyset$ almost surely, contradicting
the assumptions.

Step 2: We define 
\[
\tilde{f}(R):=\P\of{\nexists x:\,\Ball{4\sqrt{d}\fr}x\subset\Ball R0\cap\bP(\omega)}\,.
\]
 The stationary ergodic random measure $\tilde{\mu}_{\omega}(\,\cdot\,):=\lebesgueL\of{\,\cdot\,\cap\bP_{-4\sqrt{d}\fr}\left(\omega\right)}$
has positive intensity $\tilde{\lambda}_{0}>p_{\fr}\left|\S^{d-1}\left(\sqrt{d}\fr\right)^{d}\right|$
and it holds $\tilde{\mu}_{\omega}(A)\neq0$ implies the existence
of $\Ball{4\sqrt{d}\fr}x\subset\bP\cap\Ball{4\sqrt{d}\fr}A$. Assuming
that $\liminf_{R\to\infty}\tilde{f}>0$ there exists for every $R>0$
a set $\Omega_{R}\subset\Omega$ with $\tilde{\mu}_{\omega}\of{\Ball R0}=0$
for every $\omega\in\Omega_{R}$ with $\Omega_{R+1}\subset\Omega_{R}$
and 
\[
\Omega_{\infty}:=\bigcap_{R>0}\Omega_{R}\quad\text{satisfies}\quad\P(\Omega_{\infty})=\liminf_{R\to\infty}\tilde{f}(R)>0\,.
\]
But for almost every $\omega\in\Omega_{\infty}$ it holds by the ergodic
theorem 
\[
\lim_{R\to\infty}\left|\Ball R0\right|^{-1}\tilde{\mu}_{\omega}\of{\Ball R0}\geq\lambda_{0}\,,
\]
which implies the existence of $\Ball{4\sqrt{d}\fr}x\subset\Ball R0\cap\bP(\omega)$,
a contradiction.
\end{proof}

\begin{defn}[Isotropic cone mixing]
\label{def:iso-cone-mix}\nomenclature[Isotropic cone mixing]{Isotropic cone mixing}{Definition \ref{def:iso-cone-mix}}A
random set $\bP(\omega)$ is isotropic cone mixing if there exists
a jointly stationary point process $\X$ in $\Rd$ or $2\fr\Zd$,
$\fr>0$, such that almost surely two points $x,y\in\X$ have mutual
minimal distance $2\fr$ and such that $\Ball{\frac{\fr}{2}}{\X(\omega)}\subset\bP(\omega)$.
Further there exists a function $f(R)$ with $f(R)\to0$ as $\R\to\infty$
and $\alpha\in\left(0,\frac{\pi}{2}\right)$ such that with $\bE:=\left\{ e_{1},\dots e_{d}\right\} \cup\left\{ -e_{1},\dots-e_{d}\right\} $
($\left\{ e_{1},\dots e_{d}\right\} $ being the canonical basis of
$\Rd$) 
\begin{equation}
\P\of{\forall e\in\bE:\;\X\cap\cone_{e,\alpha,R}(0)\neq\emptyset}\geq1-f\of R\,.\label{eq:def-iso-cone-mixing}
\end{equation}
\end{defn}

\begin{lem}[A simple sufficient criterion for (\ref{eq:def-iso-cone-mixing})]
\label{lem:stat-erg-ball}Let $\bP$ be stationary ergodic and $(\fr,\tilde{f})$-regular.
Then $\bP$ is isotropic cone mixing with $f(R)=2d\tilde{f}\of{\left(\left(\tan\alpha\right)^{-1}+1\right)^{-1}R}$
and with 
\begin{equation}
\X(\omega):=\X_{\fr}\of{\bP\of{\omega}}=2\fr\Zd\cap\bP_{-\fr}\of{\omega}=\left\{ x\in2\fr\Zd\,:\;\Ball{\frac{\fr}{2}}x\subset\bP\right\} \label{eq:lem:stat-erg-ball}
\end{equation}
 from Lemma \ref{lem:X-r-stationary}. Vice versa, if $\bP$ is isotropic
cone mixing for $f$ then $\bP$ satisfies (\ref{eq:cri:stat-erg-ball})
with $\tilde{f}=f$.
\end{lem}

\begin{proof}[Proof of Lemma \ref{lem:stat-erg-ball}]
Because of $\P\of{A\cup B}\leq\P\of A+\P\of B$ it holds for $a>1$
\[
\P\of{\exists e\in\bE\,:\;\nexists x\in\Ball R{aRe}:\,\Ball{4\sqrt{d}\fr}x\subset\Ball R{aRe}\cap\bP}\leq2d\tilde{f}(R)\,.
\]
The existence of $\Ball{4\sqrt{d}\fr}x\subset\Ball R{aRe}\cap\bP\of{\omega}$
implies that there exists at least one $x\in\X_{\fr}\left(\bP\left(\omega\right)\right)$
such that $\Ball{\frac{\fr}{2}}x\subset\Ball R{aRe}\cap\bP\of{\omega}$
and we find 
\[
\P\of{\exists e\in\bE\,:\;\nexists x\in\X_{\fr}\of{\bP}:\,\Ball{\frac{\fr}{2}}x\subset\Ball R{aRe}\cap\bP}\leq2d\tilde{f}(R)\,.
\]
In particular, for $\alpha=\arctan a^{-1}$ and $R$ large enough
we discover 
\[
\P\left(\exists e\in\bE\,:\;\X_{\fr}\of{\bP}\cap\cone_{e,\alpha,\left(a+1\right)R}\left(0\right)=\emptyset\right)\leq2d\tilde{f}(R)\,.
\]
The relation (\ref{eq:def-iso-cone-mixing}) holds with $f(R)=2d\tilde{f}\of{\left(a+1\right)^{-1}R}$.

The other direction is evident.
\end{proof}

\subsubsection*{Properties of $\protect\X$}

The formulation of Definition \ref{def:iso-cone-mix} is particularly
useful for the following statement.
\begin{lem}[Size distribution of cells]
\label{lem:Iso-cone-geo-estimate}Let $\bP(\omega)$ be a stationary
and ergodic random open set that is isotropic cone mixing for $\X(\omega)$,
$\fr>0$, $f:\,(0,\infty)\to\R$ and $\alpha\in\left(0,\frac{\pi}{2}\right)$.
Then $\X$ and its Voronoi tessellation have the following properties:
\begin{enumerate}
\item If $G(x)$ is the open Voronoi cell of $x\in\X(\omega)$  with diameter
$d\of x$ then $d$ is jointly stationary with $\X$ and for some
constant $C_{\alpha}>0$ depending only on $\alpha$
\begin{equation}
\P(d(x)>D)<f\of{C_{\alpha}^{-1}\frac{D}{2}}\,.\label{eq:estim-diam-Voronoi}
\end{equation}
\item For $x\in\X(\omega)$  let $\cI\of x:=\left\{ y\in\X\,:\;G\of y\cap\Ball{\fr}{G\of x}\not=\emptyset\right\} $.
Then 
\begin{equation}
\#\cI\of x\leq\left(\frac{4d\of x}{\fr}\right)^{d}\,.\label{eq:estim-diam-Voronoi-2}
\end{equation}
\end{enumerate}
\end{lem}

\begin{proof}
1. W.l.o.g. let $x_{k}=0$. The first part follows from the definition
of isotropic cone mixing: We take arbitrary points $x_{\pm j}\in C_{\pm\e_{j},\alpha,R}(0)\cap\X$.
Then the planes given by the respective equations $\left(x-\frac{1}{2}x_{\pm j}\right)\cdot x_{\pm j}=0$
define a bounded cell around $0$, with a maximal diameter $D(\alpha,R)=2C_{\alpha}R$
which is proportional to $R$. The constant $C_{\alpha}$ depends
nonlinearly on $\alpha$ with $C_{\alpha}\to\infty$ as $\alpha\to\frac{\pi}{2}$.
Estimate (\ref{eq:estim-diam-Voronoi}) can now be concluded from
the relation between $R$ and $D(\alpha,R)$ and from (\ref{eq:def-iso-cone-mixing}).

2. This follows from Lemma \ref{eq:lem:estim-diam-Voronoi-cells}.
\end{proof}
\begin{lem}
\label{lem:estim-E-fa-fb}Let $\X_{\fr}$ be a stationary and ergodic
random point process with minimal mutual distance $2\fr$ for $\fr>0$
and let $f:\,(0,\infty)\to\R$ be such that the Voronoi tessellation
of $\X$ has the property 
\[
\forall x\in\fr\Zd\,:\quad\P(d(x)>D)=f\of D\,.
\]
Furthermore, let $n,s:\,\X_{\fr}\to[1,\infty)$ be measurable and
i.i.d. among $\X_{\fr}$ and let $n,s,d$ be independent from each
other. Let either
\[
G_{n(x)}\of x=\begin{cases}
x+n(x)\left(G\of x-x\right) & \text{or }\\
\Ball{n(x)d(x)}x
\end{cases}
\]
 be the cell $G\of x$ enlarged by the factor $n(x)$ or a ball of
radius $n(x)d(x)$ arround $x$, let $d(x)=\diam G(x)$ and let 
\begin{align*}
\fb_{n}\of y & :=\sum_{x\in\X_{\fr}}\chi_{G_{n}\of x}d(x)^{\eta}s(x)^{\xi}n(x)^{\zeta}\,,
\end{align*}
where $\eta,\xi,\zeta>0$ are fixed a constant. Then $\fb_{n}$ is
jointly stationary with $\X_{\fr}$ and for every $r>1$ there exists
$C\in(0,+\infty)$ such that 
\begin{align}
\E\of{\fb_{n}^{p}} & \leq C\left(\sum_{k,N,S=1}^{\infty}\left(k+1\right)^{d\left(p+1\right)+\eta p+r\left(p-1\right)}\left(S+1\right)^{\xi p+r\left(p-1\right)}\left(N+1\right)^{d\left(p+1\right)+\zeta p+r\left(p-1\right)}\P_{d,k}\P_{n,N}\P_{s,S}\right)\,.\label{eq:lem:estim-E-fa-fb-1}
\end{align}
where 
\begin{align*}
\P_{d,k} & :=\P\of{d(x)\in[k,k+1)}=f\of k-f\of{k+1}\,,\\
\P_{n,N} & :=\P\of{n(x)\in[N,N+1)}\,,\\
\P_{s,S} & :=\P\of{s(x)\in[S,S+1)}\,.
\end{align*}
\end{lem}

\begin{cor}
\label{cor:lem:estim-E-fa-fb}Under the assumptions of Lemma \ref{lem:estim-E-fa-fb}
let additionally $n=const$, $s=const$. Then 
\[
\E\of{\fb^{p}}\leq C\sum_{k,N=1}^{\infty}\left(k+1\right)^{d+(d+\eta+1)p}f(k)\,.
\]
\end{cor}

\begin{proof}[Proof of Lemma \ref{lem:estim-E-fa-fb}]
 We write $\X_{\fr}=\left(x_{i}\right)_{i\in\N}$, $d_{i}=d(x_{i})$,
$n_{i}=n(x_{i})$, $s_{i}:=s(x_{i})$. Let 
\begin{align*}
X_{k,N,S}(\omega) & :=\left\{ x_{i}\in\X_{\fr}\,:\;d_{i}\in[k,k+1),\,n_{i}\in[N,N+1),\,s_{i}\in[S,S+1)\right\} \,,\\
A_{k,N,S} & :=\bigcup_{x\in X_{k,N,S}}G_{n(x)}\of x\,,\qquad A_{k,N}:=\bigcup_{S\in\N}A_{k,N,S}\,,\qquad X_{k,N}:=\bigcup_{S\in\N}X_{k,N,S}\,.
\end{align*}
We observe that the mutual minimal distance implies 
\begin{equation}
\forall x\in\Rd:\quad\#\left\{ x_{i}\in X_{k,N,S}:\,x\in G_{n(x_{i})}\of{x_{i}}\right\} \leq\S^{d-1}\left(N+1\right)^{d}\,\left(k+1\right)^{d}\fr^{-d}\,,\label{eq:lem:estim-E-fa-fb-help-1}
\end{equation}
which follows from the uniform boundedness of cells $G_{n(x)}\of x$,
$x\in X_{k,N}$ and the minimal distance of $\left|x_{i}-x_{j}\right|>2\fr$.
Then, writing $B_{R}:=\Ball R0$ for every $y\in\Rd$ it holds by
stationarity and the ergodic theorem
\begin{align*}
\P\of{y\in G_{n_{i}}\of{x_{i}}\,:\;x_{i}\in X_{k,N,S}} & =\lim_{R\to\infty}\left|B_{R}\right|^{-1}\left|A_{k,N}\cap B_{R}\right|\P_{s,S}\\
 & \leq\lim_{R\to\infty}\left|B_{R}\right|^{-1}\left|B_{R}\cap\bigcup_{x_{i}\in X_{k,N}}G_{n_{i}}\of{x_{i}}\right|\P_{s,S}\\
 & \leq\lim_{R\to\infty}\left|B_{R}\right|^{-1}\sum_{x_{i}\in X_{k,N}\cap B_{R}}\left|\S^{d-1}\right|\left(N+1\right)^{d}\left(k+1\right)^{d}\fr^{-d}\P_{s,S}\\
 & \to\P_{d,k}\P_{n,N}\P_{s,S}\left(N+1\right)^{d}\left|\S^{d-1}\right|\left(k+1\right)^{d}\fr^{-d}\,.
\end{align*}
In the last inequality we made use of the fact that every cell $G_{n(x)}(x)$,
$x\in X_{k,N}$, has volume smaller than $\S^{d-1}\left(N+1\right)^{d}\left(k+1\right)^{d}$.
We note that for $\frac{1}{p}+\frac{1}{q}=1$
\begin{align*}
 & \int_{\bQ}\left(\sum_{x\in\X_{\fr}}\chi_{G_{n}\of x}d(x)^{\eta}s(x)^{\xi}n(x)^{\zeta}\right)^{p}\\
 & \qquad\leq\int_{\bQ}\left(\sum_{k=1}^{\infty}\sum_{N=1}^{\infty}\sum_{S=1}^{\infty}\left(\sum_{x\in X_{k,N,S}}\chi_{G_{n(x)}\of x}\left(k+1\right)^{\eta}(N+1)^{\xi}(S+1)^{\zeta}\right)\right)^{p}\\
 & \qquad\leq\int_{\bQ}\left(\sum_{k,N,S=1}^{\infty}\alpha_{k,N,S}^{q}\right)^{\frac{p}{q}}\left(\sum_{k,N,S=1}^{\infty}\alpha_{k,N,S}^{-p}\left(\sum_{x\in X_{k,N,S}}\chi_{G_{n(x)}\of x}\left(k+1\right)^{\eta}(N+1)^{\xi}(S+1)^{\zeta}\right)^{p}\right)\,.
\end{align*}
Due to (\ref{eq:lem:estim-E-fa-fb-help-1}) we find 
\[
\sum_{x\in X_{k,N,S}}\chi_{G_{n(x)}\of x}\leq\chi_{A_{k,N,S}}\left(N+1\right)^{d}\left(k+1\right)^{d}\left|\S^{d-1}\right|
\]
and obtain for $q=\frac{p}{p-1}$ and $C_{q}:=\left(\sum_{k,N,S=1}^{\infty}\alpha_{k,N,S}^{q}\right)^{\frac{p}{q}}\left|\S^{d-1}\right|^{p}$:
\begin{align*}
\frac{1}{\left|B_{R}\right|}\int_{B_{R}} & \left(\sum_{x\in\X_{\fr}}\chi_{G_{n}\of x}d(x)^{\eta}s(x)^{\xi}n(x)^{\zeta}\right)^{p}\\
 & \leq C_{q}\frac{1}{\left|B_{R}\right|}\int_{B_{R}}\left(\sum_{k,N,S=1}^{\infty}\alpha_{k,N,S}^{-p}\chi_{A_{k,N,S}}\left(N+1\right)^{dp+\zeta p}\left(k+1\right)^{dp+\eta p}(S+1)^{\xi p}\right)\\
 & \to C_{q}\left(\sum_{k,N,S=1}^{\infty}\alpha_{k,N,S}^{-p}\left(k+1\right)^{d\left(p+1\right)+\eta p}\left(N+1\right)^{d\left(p+1\right)+\zeta p}(S+1)^{\xi p}\P_{s,S}\P_{d,k}\P_{n,N}\right)
\end{align*}
For the sum $\sum_{k,N,S=1}^{\infty}\alpha_{k,N,S}^{q}$ to converge,
it is sufficient that $\alpha_{k,N,S}^{q}=\left(k+1\right)^{-r}\left(N+1\right)^{-r}\left(S+1\right)^{-r}$
for some $r>1$. Hence, for such $r$ it holds $\alpha_{k,N,S}=\left(k+1\right)^{-r/q}\left(N+1\right)^{-r/q}\left(S+1\right)^{-r/q}$
and thus (\ref{eq:lem:estim-E-fa-fb-1}).
\end{proof}

\section{\label{sec:Extension-and-Trace-d-M}Extension and Trace Properties
from $\left(\delta,M\right)$-Regularity}

\begin{figure}
 \begin{minipage}[c]{0.5\textwidth} \includegraphics[width=6cm]{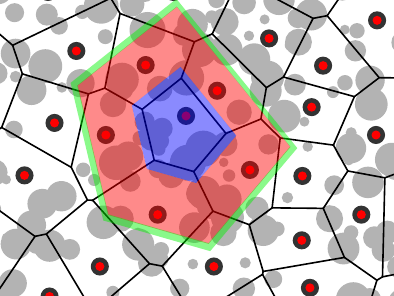}\end{minipage}\hfill   \begin{minipage}[c]{0.45\textwidth}\caption{\label{fig:sketch-extension}Gray: a Poisson ball process. Black balls:
balls of radius $\protect\fr>0$. Red Balls: radius $\frac{\protect\fr}{2}$.
The Voronoi tessellation is generated from the centers of the red
balls. The existence of such tessellations is discussed in Section
\ref{subsec:Mesoscopic-Regularity}. Blue region: $\protect\fA_{1,k}$. }
\end{minipage}
\end{figure}

\subsection{\label{subsec:5-Preliminaries}Preliminaries}

For this whole section, let $\bP$ be a Lipschitz domain which furthermore
satisfies the following assumption.
\begin{rem}
All calculations that follow in the present Section \ref{sec:Extension-and-Trace-d-M}
equally work for arbitrarily distributed radii $\fr_{a}$ associated
to $x_{a}$ and replacing the constant $\fr$, e.g. with 
\[
\cM_{a}u:=\fint_{\Ball{\frac{\fr_{a}}{16}}{x_{a}}}u\,,\qquad\overline{\nablabot_{\cM,a}}u:=\fint_{\Ball{\frac{\fr_{a}}{16}}{x_{a}}}\left(\nabla-\nablas\right)u\,.
\]
However, for simplicity of presentation, we chose to work with constant
$\fr$ from the start.
\end{rem}

\begin{assumption}
\label{assu:mesoscopic-voronoi}Let $\bP$ be an open (unbounded)
set and let $\X_{\fr}=\left(x_{a}\right)_{a\in\N}$ be a set of points
having mutual distance $\left|x_{a}-x_{b}\right|>2\fr$ if $a\neq b$
and with $\Ball{\frac{\fr}{2}}{x_{a}}\subset\bP$ for every $a\in\N$
(e.g. $\X_{\fr}(\bP)$, see (\ref{eq:def-X_r})). We construct from
$\X_{\fr}$ a Voronoi tessellation and denote by $G_{a}:=G(x_{a})$
the Voronoi cell corresponding to $x_{a}$ with diameter $d_{a}$
with $\fA_{1,a}:=\Ball{\frac{\fr}{2}}{G_{a}}$. Let $\tilde{\Phi}_{0}\in C^{\infty}(\R;[0,1])$
be monotone decreasing with $\tilde{\Phi}_{0}'>-\frac{4}{\fr}$, $\tilde{\Phi}_{0}(x)=1$
if $x\leq0$ and $\tilde{\Phi}_{0}(x)=0$ for $x\geq\frac{\fr}{2}$.
We define on $\Rd$ the Lipschitz functions 
\begin{equation}
\tilde{\Phi}_{a}(x):=\tilde{\Phi}_{0}\left(\dist\left(x,G_{a}\right)\right)\quad\text{and}\quad\Phi_{a}(x):=\tilde{\Phi}_{a}(x)\left(\sum_{b}\tilde{\Phi}_{b}(x)\right)^{-1}\,.\label{eq:def-Phi-i}
\end{equation}
\end{assumption}

Lemma \ref{lem:estim-diam-Voronoi-cells} implies 
\begin{equation}
\forall x\in\Ball{\frac{\fr}{2}}{G_{a}}\,:\quad\#\left\{ b\,:\;x\in\fA_{1,b}\right\} \leq\left(\frac{4d_{a}}{\fr}\right)^{d}\label{eq:assu:mesoscopic-regular-1}
\end{equation}
and thus (\ref{eq:def-Phi-i}) yields for some $C$ depending only
on $\tilde{\Phi}_{0}$ that
\begin{equation}
\left|\nabla\Phi_{a}\right|\leq Cd_{a}^{d}\quad\text{and}\quad\forall k:\,\left|\nabla\Phi_{k}\right|\chi_{\fA_{1,a}}\leq Cd_{a}^{d}\,.\label{eq:estiamte-nabla-Phi-i}
\end{equation}

\begin{defn}[Weak Neighbors]
\label{def:neighbors}\nomenclature[neighbors weak]{$x_i\sim\sim x_j$}{$x_i$ and $x_j$ are weak neighbors or weakly connected, see Definition \ref{def:neighbors}}Under
the Assumption \ref{assu:mesoscopic-voronoi}, two points $x_{a},x_{b}\in\X_{\fr}$
are called to be  weakly connected (or weak neighbors), written $a\sim\sim b$
or $x_{a}\sim\sim x_{b}$ if $\Ball{\frac{\fr}{2}}{G_{a}}\cap\Ball{\frac{\fr}{2}}{G_{b}}\neq\emptyset$.
For $\bQ\subset\Rd$ open we say $\fA_{1,a}\sim\sim\bQ$ if $\Ball{\frac{\fr}{2}}{\fA_{1,a}}\cap\bQ\neq\emptyset$.
We then define \nomenclature[XrQ]{$\X_\fr(\bQ)$}{\eqref{eq:Xr-Q-Q-simsim}}\nomenclature[Qsim]{$\bQ^{\sim\sim}$}{\eqref{eq:Xr-Q-Q-simsim}}
\begin{equation}
\X_{\fr}(\bQ):=\left\{ x_{a}\in\X_{\fr}:\;\fA_{1,a}\sim\sim\bQ\neq\emptyset\right\} \,,\quad\bQ^{\sim\sim}:=\bigcup_{\fA_{1,a}\sim\sim\bQ}\fA_{1,a}\,.\label{eq:Xr-Q-Q-simsim}
\end{equation}
\end{defn}

In view of Assumption \ref{assu:mesoscopic-voronoi} we bound $\delta_{\Delta}$
by $\fr>0$ and recall (\ref{eq:lem:properties-delta-M-regular-3}).
As announced in the introduction, we apply Corollary \ref{cor:cover-boundary}
for $n\in\N$ (we study mostly $n=1$ and $n=2$ in the following)
to obtain a complete covering of $\partial\bP$ by balls $\Ball{\tilde{\rho}_{n}\left(p_{i}^{n}\right)}{p_{i}^{n}}$,
$\left(p_{i}^{n}\right)_{k\in\N}$, where $\tilde{\rho}_{n}\of p:=2^{-5}\rho_{n}\of p$.
Recalling (\ref{eq:def-rho-of-p})--(\ref{eq:def-rhohat-of-p}) we
define with $\tilde{\rho}_{n,i}:=\tilde{\rho}_{n}\of{p_{i}^{n}}$,
$\hat{\rho}_{n,i}:=\hat{\rho}_{n,3}\of{p_{i}^{n}}$\nomenclature[A1]{$A_{1,i},A_{2,i},A_{3,i}$}{Equation \eqref{eq:A123-k}}
and 
\begin{equation}
A_{1,i}^{n}:=\Ball{\tilde{\rho}_{n,i}}{p_{i}^{n}}\,,\quad A_{2,i}^{n}:=\Ball{3\tilde{\rho}_{n,i}}{p_{i}^{n}}\,,\quad A_{3,i}^{n}:=\Ball{\hat{\rho}_{n,i}}{p_{i}^{n}}\,,\quad B_{n,i}:=\Ball{\frac{1}{8}\tilde{\rho}_{n,i}}{p_{i}^{n}}\,,\label{eq:A123-k}
\end{equation}
where we recall the construction of $\fr_{n,\alpha,i}$ and $y_{n,\alpha,i}$
in (\ref{eq:def-fr-n-alph})--(\ref{eq:condi-fr-n-i}) and note that
$\Ball{\tilde{\rho}_{n,i}}{p_{i}^{n}}\supset\Ball{\fr_{n,\alpha,i}}{y_{n,\alpha,i}}$
independent from $\alpha$.
\begin{lem}
\label{lem:properties-local-rho-convering}For $n\in\N$, $\alpha\in[0,1]$
and any two balls $A_{1,i}^{n}\cap A_{1,j}^{n}\neq\emptyset$ either
$A_{1,i}^{n}\subset A_{2,j}^{n}$ or $A_{1,j}^{n}\subset A_{2,i}^{n}$
and
\begin{equation}
A_{1,i}^{n}\cap A_{1,j}^{n}\neq\emptyset\quad\Rightarrow\quad\Ball{\frac{1}{2}\tilde{\rho}_{n,i}}{p_{i}}\subset A_{2,j}^{n}\text{ and }\Ball{\frac{1}{2}\tilde{\rho}_{n,j}}{p_{j}}\subset A_{2,i}^{n}\,.\label{eq:lem:properties-local-rho-convering-1}
\end{equation}
Furthermore, there exists a constant $C$ depending only on the dimension
$d$ and some $\hat{d}\in[0,d]$ such that 
\begin{align}
\forall k\quad &  & \#\left\{ j\,:\;A_{1,j}^{n}\cap A_{1,i}^{n}\neq\emptyset\right\} +\#\left\{ j\,:\;A_{2,j}^{n}\cap A_{2,i}^{n}\neq\emptyset\right\}  & \leq C\,,\label{eq:lem:properties-local-rho-convering-2}\\
\forall x\quad &  & \#\left\{ j\,:\;x\in A_{1,j}^{n}\right\} +\#\left\{ j\,:\;x\in A_{2,j}^{n}\right\}  & \leq C+1\,,\label{eq:lem:properties-local-rho-convering-3}\\
\forall x\quad &  & \#\left\{ j:\,x\in\overline{\Ball{\hat{\rho}_{n,j}}{p_{j}}}\right\}  & <C(1+M_{[\frac{3\delta}{8},\frac{\delta}{8}],\Rd}(x))^{n\hat{d}}\,.\label{eq:lem:properties-local-rho-convering-4}
\end{align}
Finally, there exist non-negative functions $\phi_{n,0}$ and $\left(\phi_{n,i}\right)_{k\in\N}$
independent from $\alpha$ such that for $k\geq1$: $\support\phi_{n,i}\subset A_{1,i}^{n}$,
$\phi_{n,i}|_{B_{n,j}}\equiv0$ for $k\neq j$. Further, $\phi_{n,0}\equiv0$
on all $B_{n,i}$ and on $\partial\bP$ and $\sum_{k=0}^{\infty}\phi_{n,i}\equiv1$
and there exists $C$ depending only on $d$ such that for all $k\in\N$
it holds 
\begin{equation}
x\in A_{1,i}^{n}\qquad\Rightarrow\qquad\forall j\in\N\cup\left\{ 0\right\} :\quad\left|\nabla\phi_{n,j}(x)\right|\leq C\tilde{\rho}_{n,i}^{-1}\,.\label{eq:bound-nabla-phi-j}
\end{equation}
\end{lem}

\begin{rem}
\label{rem:lem:properties-local-rho-convering}We usually can improve
$\hat{d}$ to at least $\hat{d}=d-1$. To see this assume $\partial\bP$
is flat on the scale of $\delta$. Then all points $p_{i}$ lie on
a $d-1$-dimensional plane and we can thus improve the argument in
the following proof to $\hat{d}=d-1$.
\end{rem}

\begin{proof}
(\ref{eq:lem:properties-local-rho-convering-1}) follows from (\ref{eq:cor:cover-boundary-h1})$_{2}$.
For improved readability we drop the indeces $n$ and $\alpha$.

Let $k\in\N$ be fixed. By construction in Corollary \ref{cor:cover-boundary},
every $A_{1,j}$ with $A_{1,j}\cap A_{1,k}\neq\emptyset$ satisfies
$\tilde{\rho}_{j}\geq\frac{1}{2}\tilde{\rho}_{k}$ and hence if $A_{1,j}^ {}\cap A_{1,k}\neq\emptyset$
and $A_{1,i}\cap A_{1,k}\neq\emptyset$ we find $\left|p_{j}-p_{i}\right|\geq\frac{1}{4}\tilde{\rho}_{k}$
and $\left|p_{j}-p_{k}\right|\leq3\tilde{\rho}_{k}$. This implies
(\ref{eq:lem:properties-local-rho-convering-2})--(\ref{eq:lem:properties-local-rho-convering-3})
for $A_{1,j}$ and the statement for $A_{2,j}$ follows analogously.

For two points $p_{i}$, $p_{j}$ such that $x\in A_{3,i}\cap A_{3,j}$
it holds due to the triangle inequality $\left|p_{i}-p_{j}\right|\leq\max\left\{ \frac{1}{4}\hat{\rho}_{i},\frac{1}{4}\hat{\rho}_{j}\right\} $.
Let $\X(x):=\left\{ p_{i}\in\X:\,x\in\overline{\Ball{\frac{1}{8}\hat{\rho}_{i}}{p_{i}}}\right\} $
and choose $\tilde{p}(x)=\tilde{p}\in\X(x)$ such that $\delta_{\text{m}}:=\delta\of{\tilde{p}}$
is maximal. Then $\X(x)\subset\Ball{\frac{1}{4}\delta_{\text{m}}}{\tilde{p}}$
and every $p_{i}\in\X(x)$ satisfies $\delta_{\text{m}}>\delta_{i}>\frac{1}{3}\delta_{\text{m}}$.
Correspondingly, $\tilde{\rho}_{i}>\frac{1}{3}\delta_{\text{m}}2^{-5}\tilde{M}_{\frac{\delta_{i}}{8}}^{-n}>\frac{1}{3}\delta_{\text{m}}2^{-5}\tilde{M}_{\frac{3\delta_{m}}{8}}^{-n}$
for all such $p_{i}$. In view of (\ref{eq:cor:cover-boundary-h1})
this lower local bound of $\tilde{\rho}_{i}$ implies a lower local
bound on the mutual distance of the $p_{i}$. Since this distance
is proportional to $\delta_{\text{m}}\tilde{M}_{\frac{3\delta_{m}}{8}}^{-n}$,
this implies (\ref{eq:lem:properties-local-rho-convering-4}) with
$\hat{d}=d$. This is by the same time the upper estimate on $\hat{d}$.

Let $\phi:\,\R\to\R$ be symmetric, smooth, monotone on $(0,\infty)$
with $\phi'\leq2$ and $\phi=0$ on $(1,\infty)$. For each $k$ we
consider a radially symmetric smooth function $\hat{\phi}_{k}\of x:=\phi\of{\frac{\left|x-p_{k}\right|^{2}}{\tilde{\rho}_{k}}}$
and an additional function $\tilde{\phi}_{0}\left(x\right)=\dist\of{\,x,\,\partial\bP\cup\bigcup_{k}B_{n,k}\,}$.
In a similar way we may modify $\tilde{\phi}_{k}:=\hat{\phi}_{k}\,\dist\of{x,\bigcup_{j\neq k}B_{n,j}}$
such that $\tilde{\phi}_{k}|_{B_{n,j}}\equiv0$ for $j\neq k$. Then
we define $\phi_{k}:=\tilde{\phi}/\left(\tilde{\phi}_{0}+\sum_{j}\tilde{\phi}_{j}\right)$.
Note that by construction of $\fr_{k}$ and $y_{k}$ we find $\phi_{k}|_{B_{k}}\equiv1$
and $\sum_{k\geq1}\phi_{k}\equiv1$ on $\partial\bP$.

Estimate (\ref{eq:bound-nabla-phi-j}) follows from (\ref{eq:lem:properties-local-rho-convering-2}).
\end{proof}

\subsection{\label{subsec:Microscopic-Extension-dm}Extensions preserving the
Gradient norm via $\left(\delta,M\right)$-Regularity of $\partial\protect\bP$}

By Lemma \ref{lem:uniform-extension-lemma} in case $n=1$ there exist
local extension operator 
\begin{equation}
\cU_{n,i}:\,\,W^{1,p}\of{\bP\cap A_{3,i}^{n}}\,\to\,W^{1,p}\of{\Ball{\frac{1}{8}\rho_{n,i}}{p_{i}^{n}}\backslash\bP}\,\hookrightarrow\,W^{1,p}\of{A_{2,i}^{n}\backslash\bP}\label{eq:very-local-extension}
\end{equation}
which is linear continuous with bounds 
\begin{align}
\left\Vert \nabla\cU_{n,i}u\right\Vert _{L^{p}(A_{2,i}^{n}\backslash\bP)} & \leq2M_{n,i}\left\Vert \nabla u\right\Vert _{L^{p}\left(A_{3,i}^{n}\cap\bP\right)}\,,\label{eq:lem:local-delta-M-extension-estimate-h2}\\
\left\Vert \cU_{n,i}u\right\Vert _{L^{p}(A_{2,i}^{1}\backslash\bP)} & \leq\left\Vert u\right\Vert _{L^{p}\left(A_{3,i}^{1}\cap\bP\right)}\,.\label{eq:lem:local-delta-M-extension-estimate-h2-b}
\end{align}
Of course, higher $n>1$ are always valid, but the result becomes
worse, as we will see. However, in case $\partial\bP$ is locally
always in the upper half plane, the case $n=0$ is also valid, improving
the estimates of the extension operators significantly. This phenomenon
is acknowledged through the Definition \ref{def:extension-order}
of the extension order.
\begin{defn}
\label{def:cU-Q}Using Notation \ref{nota:extension-1} for every
$\bQ\subset\Rd$ let 
\begin{align}
\cU_{n,\alpha,\bQ}:\,C^{1}\of{\overline{\bP\cap\Ball{\frac{\fr}{2}}{\bQ}}} & \to C^{1}\of{\overline{\bQ\backslash\bP}}\,,\nonumber \\
u & \mapsto\chi_{\bQ\backslash\bP}\sum_{i\neq0}\sum_{a}\Phi_{a}\left(\phi_{n,i}\left(\cU_{n,i}\of{u-\tau_{n,\alpha,i}u}+\tau_{n,\alpha,i}u-\cM_{a}u\right)+\cM_{a}u\right)\,.\label{eq:def-extension-op}
\end{align}
\end{defn}

Due to the defintions, we find 
\begin{equation}
\tau_{n,\alpha,i}\cM_{a}u=\cM_{a}u\,.\label{eq:permutation-tau-cM}
\end{equation}

\begin{lem}
\label{lem:local-delta-M-extension-estimate} Let $\bP\subset\Rd$
be a Lipschitz domain (i.e. locally $\left(\delta,M\right)$-regular)
with $\delta_{\Delta}$ bounded by $\fr>0$ and let Assumption \ref{assu:M-alpha-bound}
hold and let $\hat{d}$ be the constant from (\ref{eq:lem:properties-local-rho-convering-4}).
Then for every bounded open $\bQ\subset\Rd$ with $\Ball{10\fr}0\subset\bQ$
and $1\leq r<p$ the linear operator 
\begin{align*}
\cU_{n,\alpha,\bQ}:\,W^{1,p}\left(\bP\cap\Ball{\frac{\fr}{2}}{\bQ}\right) & \to W^{1,r}\left(\bQ\right)
\end{align*}
is continuous and writing 
\[
f_{\alpha,n,\hat{d}}(M,\cdot\,):=\left(\left(1+M_{[\frac{3\delta}{8},\frac{\delta}{8}],\Rd}\right)^{n\hat{d}}\left(1+M_{[\frac{1}{8}\delta],\Rd}\right)^{r}\left(1+M_{[\tilde{\rho}_{n}],\Rd}\right)^{\alpha(d-1)}\right)^{\frac{p}{p-r}}
\]
the operator $\cU_{n,\alpha,\bQ}$ satisfies for some $C$ not depending
on $\bP$
\begin{align}
\fint_{\bQ}\left|\nabla\left(\cU_{n,\alpha,\bQ}u\right)\right|^{r} & \leq C\left(\fint_{\Ball{\fr}{\bQ}}f_{\alpha,n,\hat{d}}(M)\right)^{r\frac{p-r}{p}}\left(\frac{1}{\left|\bQ\right|}\int_{\Ball{\fr}{\bQ}\cap\bP}\left|\nabla u\right|^{p}\right)^{\frac{r}{p}}\nonumber \\
 & \quad+C\frac{1}{\left|\bQ\right|}\int_{\bQ\backslash\bP}\sum_{a}\Phi_{a}\left|\sum_{i\neq0}\rho_{1,i}^{-1}\chi_{A_{1,i}}\left(\tau_{n,\alpha,i}u-\cM_{a}u\right)\right|^{r}\label{lem:local-delta-M-extension-estimate-estim-1}\\
 & \quad+\frac{1}{\left|\bQ\right|}\int_{\bQ}\left|\sum_{l=1}^{d}\sum_{a:\,\partial_{l}\Phi_{a}>0}\sum_{b:\,\partial_{l}\Phi_{b}<0}\frac{\partial_{l}\Phi_{a}\left|\partial_{l}\Phi_{b}\right|}{D_{l+}^{\Phi}}\left(\cM_{a}u-\cM_{b}u\right)\right|^{r}\\
\fint_{\bQ}\left|\cU_{n,\alpha,\bQ}u\right|^{r} & \leq C_{0}\left(\frac{1}{\left|\bQ\right|}\int_{\Ball{\frac{\fr}{2}}{\bQ}\cap\bP}(1+M_{[\frac{3\delta}{8},\frac{\delta}{8}],\Rd})^{\frac{p\hat{d}}{p-r}}\right)^{\frac{p-r}{p}}\left(\frac{1}{\left|\bQ\right|}\int_{\Ball{\fr}{\bQ}\cap\bP}\left|u\right|^{p}\right)^{\frac{r}{p}}\,,\label{lem:local-delta-M-extension-estimate-estim-1-b}
\end{align}
where
\begin{equation}
D_{l+}^{\Phi}:=\sum_{\substack{a\neq0:\,\partial_{l}\Phi_{a}<0}
}\left|\partial_{l}\Phi_{a}\right|\,.\label{lem:local-delta-M-extension-estimate-estim-def-D-l-phi}
\end{equation}
\end{lem}

\begin{rem*}
Since the covering $A_{1,i}$ is locally finite we find 
\[
\left|\sum_{i\neq0}\rho_{1,i}^{-1}\chi_{A_{1,i}}\left(\tau_{n,\alpha,i}u-\cM_{a}u\right)\right|^{r}\leq\sum_{i\neq0}\rho_{1,i}^{-r}\chi_{A_{1,i}}\left|\tau_{n,\alpha,i}u-\cM_{a}u\right|^{r}\,.
\]
\end{rem*}

\subsection{Extensions preserving the Symmetric Gradient norm via $\left(\delta,M\right)$-Regularity
of $\partial\protect\bP$}

By Lemmas \ref{lem:rho-p-lsc} and \ref{lem:uniform-Nitsche-extension-lemma}
in case $n=2$ the local extension operator 
\begin{equation}
\cU_{n,k}:\,\,W^{1,p}\of{\bP\cap A_{3,k}^{n}}\,\to\,W^{1,p}\of{\Ball{\frac{1}{8}\rho_{n,k}}{p_{k}^{n}}\backslash\bP}\,\hookrightarrow\,W^{1,p}\of{A_{2,k}^{n}\backslash\bP}\label{eq:very-local-extension-sym}
\end{equation}
is linear continuous with bounds 
\begin{align}
\left\Vert \nablas\cU_{n,k}u\right\Vert _{L^{p}(\Ball{\frac{1}{8}\rho_{n,k}}{p_{k}^{n}}\backslash\bP)} & \leq C\tilde{M}_{n,k}^{2}\left\Vert \nablas u\right\Vert _{L^{p}\left(A_{3,k}^{n}\cap\bP\right)}\,.\label{eq:lem:local-delta-M-extension-estimate-h2-sym}
\end{align}
Like in Section \ref{subsec:Microscopic-Extension-dm} lower values
of $n$ are possible, acknowledged by Definition \ref{def:extension-order}
of symmetric extension order.
\begin{defn}
Using the notation of Definition \ref{def:cU-Q-sym} let 
\begin{align}
\cU_{n,\alpha,\bQ}:\,C^{1}\of{\overline{\bP\cap\Ball{\frac{\fr}{2}}{\bQ}}} & \to C^{1}\of{\overline{\bQ\backslash\bP}}\,,\nonumber \\
u & \mapsto\chi_{\bQ\backslash\bP}\sum_{k}\sum_{a}\Phi_{a}\left(\phi_{n,k}\left(\cU_{n,k}\of{u-\tau_{n,\alpha,k}^{\fs}u}+\tau_{n,\alpha,k}^{\fs}u-\cM_{a}^{\fs}u\right)+\cM_{a}^{\fs}u\right)\label{eq:def-extension-op-sym}
\end{align}
where $\cU_{n,k}$ are the extension operators on $A_{3,k}^{n}$ given
by the symmetric extension order of $\bP$.
\end{defn}

By definition we verify $\nablas\left(u-\tau_{n,\alpha,i}^{\fs}u\right)=\nablas u$
as well as 
\[
\fint_{\Ball{\fr_{n,\alpha,i}}{y_{n,\alpha,i}}}\left(\nabla-\nablas\right)\left(u-\tau_{n,\alpha,i}^{\fs}u\right)=0\,,\qquad\fint_{\Ball{\fr_{n,\alpha,i}}{y_{n,\alpha,i}}}\left(u-\tau_{n,\alpha,i}^{\fs}u\right)=0
\]
and similarly for $\cM_{a}^{\fs}u$. Furthermore, it holds 
\begin{equation}
\tau_{n,\alpha,i}^{\fs}\cM_{a}^{\fs}u=\cM_{a}^{\fs}u\,.\label{eq:permutation-tau-cM-fs}
\end{equation}

\begin{lem}
\label{lem:local-delta-M-extension-estimate-sym} Let $\bP\subset\Rd$
be a locally $\left(\delta,M\right)$-regular open set with delta
bounded by $\fr>0$ and let Assumption \ref{assu:M-alpha-bound} hold
and let $\hat{d}$ be the constant from (\ref{eq:lem:properties-local-rho-convering-4}).
Then for every bounded open $\bQ\subset\Rd$, $1\leq r<p$ the operator
\begin{align*}
\cU_{n,\bQ}:\,W^{1,p}\left(\bP\cap\Ball{\frac{\fr}{2}}{\bQ}\right) & \to W^{1,r}\left(\bQ\right)
\end{align*}
is linear, well defined and with 
\[
f_{\alpha,n,\hat{d}}^{\fs}(M,\cdot\,):=\left(\left(1+M_{[\frac{3\delta}{8},\frac{\delta}{8}],\Rd}\right)^{\hat{d}}\left(1+M_{[\frac{1}{8}\delta],\Rd}\right)^{2r}\left(1+M_{[\tilde{\rho}_{n}],\Rd}\right)^{\alpha(d-1)}\right)^{\frac{p}{p-r}}
\]
satisfies 
\begin{align}
\fint_{\bQ}\left|\nablas\left(\cU_{2,\bQ}u\right)\right|^{r} & \leq C\left(\fint_{\Ball{\fr}{\bQ}}f_{\alpha,n,\hat{d}}^{\fs}(M)\right)^{r\frac{p-r}{p}}\left(\frac{1}{\left|\bQ\right|}\int_{\Ball{\fr}{\bQ}\cap\bP}\left|\nablas u\right|^{p}\right)^{\frac{r}{p}}\nonumber \\
 & \quad+C\frac{1}{\left|\bQ\right|}\int_{\bQ\backslash\bP}\sum_{a}\Phi_{a}\left|\sum_{i\neq0}\rho_{1,i}^{-1}\chi_{A_{1,i}}\left(\tau_{n,\alpha,i}^{\fs}u-\cM_{a}^{\fs}u\right)\right|^{r}\label{lem:local-delta-M-extension-estimate-estim-sym-1}\\
 & \quad+\frac{1}{\left|\bQ\right|}\int_{\bQ}\left|\sum_{l=1}^{d}\sum_{a:\,\partial_{l}\Phi_{a}>0}\sum_{b:\,\partial_{l}\Phi_{b}<0}\frac{\partial_{l}\Phi_{a}\left|\partial_{l}\Phi_{b}\right|}{D_{l+}^{\Phi}}\left(\cM_{a}^{\fs}u-\cM_{b}^{\fs}u\right)\right|^{r}\\
\fint_{\bQ}\left|\cU_{\bQ}u\right|^{r} & \leq C_{0}\left(\frac{1}{\left|\bQ\right|}\int_{\Ball{\frac{\fr}{2}}{\bQ}\cap\bP}(1+M_{[\frac{3\delta}{8},\frac{\delta}{8}],\Rd})^{\frac{2p\hat{d}}{p-r}}\right)^{\frac{p-r}{p}}\left(\frac{1}{\left|\bQ\right|}\int_{\Ball{\fr}{\bQ}\cap\bP}\left|u\right|^{p}\right)^{\frac{r}{p}}\,,\label{lem:local-delta-M-extension-estimate-estim-sym-1-b}
\end{align}
where $D_{l+}^{\Phi}$ is given by (\ref{lem:local-delta-M-extension-estimate-estim-def-D-l-phi})
\end{lem}

\subsection{Support}
\begin{thm}
\label{thm:support}For both operators given in (\ref{eq:def-extension-op})
and (\ref{eq:def-extension-op-sym}) the following holds: For every
bounded open set $\bQ$ with $0\in\bQ$ and $n_{0},n_{1}\in\N$ let
\[
\forall M>1:\qquad\tilde{\bQ}_{M}:=\bigcup_{x_{a}\in\X_{\fr}\cap M\bQ}\Ball{\fr}{G_{a}}\,.
\]
If the mesoscopic regularity function $\tilde{f}$ of $\bP$ satisfies
$\tilde{f}(D)\leq CD^{-\frac{d-1}{\alpha}+\beta}$ for some $C>0$,
$\alpha\in(0,1)$ and $\beta>1$ then there exists almost surely $M_{0}>1$
such that for every $M>M_{0}$ it holds $\tilde{\bQ}_{M}\subset\Ball{M^{\alpha}}{M\bQ}$.
\end{thm}

\begin{proof}
We consider two balls $\Ball r0\subset\bQ\subset\Ball R0$ with $r>0$. 

We write $\bQ_{M}:=M\bQ$ and $\B_{M,\alpha,\bQ}:=\Ball{M^{\alpha}}{\bQ_{M}}$
for $\alpha\in(0,1)$ with $\B_{M,\alpha,\bQ}^{\complement}:=\Rd\backslash\B_{M,\alpha,\bQ}$.
For $k\in\N$ we introduce 
\[
\bQ_{M,k}:=\left\{ x\in\bQ_{M}:\,\dist(x,\partial\bQ_{M})\in[k,k)\right\} 
\]
and find 
\[
\P\of{\tilde{\bQ}_{M}\subset\B_{M,\alpha,\bQ}}=1-\sum_{k}\P\of{\exists x_{a}\in\bQ_{M,k}\cap\X_{r}:\,\Ball{\fr}{G_{a}}\cap\B_{M,\alpha,\bQ}^{\complement}\neq\emptyset}.
\]
On the other hand, 
\begin{align*}
 & \P\of{\exists x_{a}\in\bQ_{M,k}\cap\X_{r}:\,\Ball{\fr}{G_{a}}\cap\B_{M,\alpha,\bQ}^{\complement}\neq\emptyset}\\
 & \qquad\qquad\leq\P\of{\exists x_{a}\in\bQ_{M,k}\cap\X_{r}:\,\Ball{2d_{a}}{x_{a}}\cap\B_{M,\alpha,\bQ}^{\complement}\neq\emptyset}\\
 & \qquad\qquad\leq C\partial\bQ_{M}\P\of{d_{a}>\frac{k}{2}+M^{\alpha}}\\
 & \qquad\qquad\leq CM^{d-1}\left(\frac{k}{2}+M^{\alpha}\right)^{-\left(\frac{d-1}{\alpha}+\beta_{1}+\beta_{2}\right)}\leq CM^{-\beta_{1}}\left(\frac{k}{2}\right)^{-\beta_{2}}
\end{align*}
where $C$ depends only on the minimal mutual distance of the points,
i.e. $\fr$, and the shape of $\bQ$. Now, since $\beta>1$ we can
choose $\beta_{2}>1$ and find
\[
\P\of{\tilde{\bQ}_{M}\subset\B_{M,\alpha,\bQ}}\geq1-CM^{-\beta_{1}}\,.
\]
Since the right hand side converges to $1$ as $M\to\infty$, we can
conclude.
\end{proof}

\subsection{Proof of Lemmas \ref{lem:local-delta-M-extension-estimate} and \ref{lem:local-delta-M-extension-estimate-sym}}
\begin{lem}
\label{lem:conv-sum-0}Let $\alpha_{i}$, $u_{i}$, $i=1\dots n$,
be a family of real numbers such that $\sum_{i}\alpha_{i}=0$ and
let $\alpha_{+}:=\sum_{i:\,\alpha_{i}>0}\alpha_{i}$. Then 
\[
\sum_{i}\alpha_{i}u_{i}=\sum_{i:\,\alpha_{i}>0}\sum_{j:\,\alpha_{j}<0}\frac{\alpha_{i}\left|\alpha_{j}\right|}{\alpha_{+}}\left(u_{i}-u_{j}\right)\,.
\]
\end{lem}

\begin{proof}
\begin{align*}
\sum_{i}\alpha_{i}u_{i} & =\sum_{i:\,\alpha_{i}>0}\alpha_{i}u_{i}+\sum_{j:\,\alpha_{j}<0}\alpha_{j}u_{j}\\
 & =\sum_{i:\,\alpha_{i}>0}\alpha_{i}\sum_{j:\,\alpha_{j}<0}\frac{-\alpha_{j}}{\alpha_{+}}u_{i}+\sum_{j:\,\alpha_{j}<0}\alpha_{j}\sum_{i:\,\alpha_{i}>0}\frac{\alpha_{i}}{\alpha_{+}}u_{j}\\
 & =\sum_{i:\,\alpha_{i}>0}\sum_{j:\,\alpha_{j}<0}\frac{\alpha_{i}\left|\alpha_{j}\right|}{\alpha_{+}}\left(u_{i}-u_{j}\right)\,.
\end{align*}
\end{proof}

\begin{proof}[Proof of Lemma \ref{lem:local-delta-M-extension-estimate}]
 For improved readability, we drop the indeces $n$ and $\alpha$
in the following.

We prove Lemma \ref{lem:local-delta-M-extension-estimate}, i.e. (\ref{lem:local-delta-M-extension-estimate-estim-1})
as (\ref{lem:local-delta-M-extension-estimate-estim-1-b}) can be
derived in a similar but shorter way. Lemma \ref{lem:local-delta-M-extension-estimate-sym}
can be proved in a similar way with some inequalities used below being
replaced by the ``symmetrized'' counterparts. We will make some
comments towards this direction in Step 4 of this proof.

For shortness of notation (and by abuse of notation) we write
\[
\fint_{\bP\cap\bQ}g:=\frac{1}{\left|\bQ\right|}\int_{\bP\cap\bQ}g\,,\qquad\fint_{\bQ\backslash\bP}g:=\frac{1}{\left|\bQ\right|}\int_{\bQ\backslash\bP}g
\]
and similar for integrals over $\Ball{\frac{\fr}{2}}{\bQ}\cap\bP$
and $\Ball{\frac{\fr}{2}}{\bQ}\backslash\bP$. For simplicity of notation,
we further drop the index $1$ in the subsequent calculations.

We introduce the quantities 
\[
\tilde{M}_{\tilde{\rho},i}:=M_{\tilde{\rho}(p_{i})}(p_{i})\,,\qquad\tilde{M}_{\delta,1,i}:=M_{\frac{1}{8}\delta(p_{i})}(p_{i})\,,\qquad\tilde{M}_{\delta,2,i}:=M_{\frac{3}{8}\delta(p_{i})}(p_{i})
\]
note that $\tilde{\rho}_{i}\leq\frac{1}{8}\delta_{i}$ as well as
$\sqrt{4M_{i}^{2}+2}\leq2\tilde{M}_{i}$. Writing 
\begin{align*}
u_{i} & :=\cU_{i}\left(u-\tau_{i}u\right)+\tau_{i}u &  & \text{on }A_{2,i}\\
u_{i,a} & :=\cU_{i}\left(u-\tau_{i}u\right)+\tau_{i}u-\cM_{a}u &  & \text{on }A_{2,i}\cap\fA_{1,a}
\end{align*}
 on $A_{2,i}$, The integral over $\nabla\left(\cU_{\bQ}u\right)$
can be estimated via
\begin{gather}
\fint_{\bQ\backslash\bP}\left|\nabla\left(\cU_{\bQ}u\right)\right|^{r}\leq C_{r}\left(I_{1}+I_{2}+I_{3}\right)\label{eq:lem:dm-general-estimate-1}\\
I_{1}:=\fint_{\bQ\backslash\bP}\left|\sum_{i\neq0}\sum_{a}\Phi_{a}\phi_{i}\nabla u_{i,a}\right|^{r}\,,\qquad I_{2}:=\fint_{\bQ\backslash\bP}\left|\sum_{i\neq0}\sum_{a}u_{i,a}\Phi_{a}\nabla\phi_{i}\right|^{r}\,,\nonumber \\
I_{3}:=\fint_{\bQ\backslash\bP}\left|\sum_{i\neq0}\sum_{a}u_{i,a}\phi_{i}\nabla\Phi_{a}\right|^{r}\,.
\end{gather}
\emph{Step 1:} Using (\ref{eq:ex-order-estimate-1}) and $\nabla u_{i,a}=\nabla u_{i}$
as well as $\sum_{a}\Phi_{a}=1$ we conclude 
\begin{align*}
I_{1}=\fint_{\bQ\backslash\bP}\left|\sum_{i\neq0}\phi_{i}\nabla u_{i}\right|^{r} & \leq\fint_{\bQ\backslash\bP}\sum_{i\neq0}\phi_{i}\left|\nabla u_{i}\right|^{r}\leq\fint_{\bQ\backslash\bP}\sum_{i\neq0}\chi_{A_{1,i}}\left|\nabla u_{i}\right|^{r}\\
 & \leq C\,\sum_{i\neq0}\fint_{\bQ}\chi_{A_{2,i}}\left|\nabla u_{i}\right|^{r}\leq C\,\sum_{i\neq0}\tilde{M}_{\delta,2,i}^{r}\fint_{\Ball{\frac{\fr}{2}}{\bQ}\cap\bP}\chi_{A_{3,i}}\left|\nabla u\right|^{r}\,.
\end{align*}
It only remains to estimate $\sum_{i}\chi_{A_{3,i}}(x)$. After a
Hölder estimate and using $\tilde{M}_{\delta,2,i}\leq1+M_{[\frac{1}{8}\delta],\Rd}$
on $A_{3,i}$, we obtain 
\begin{align}
\sum_{i\neq0}\tilde{M}_{\delta,2,i}^{r} & \fint_{\bQ\cap\bP}\chi_{A_{3,i}}\left|\nabla u\right|^{r}\leq\fint_{\Ball{\frac{\fr}{2}}{\bQ}\cap\bP}\sum_{i\neq0}\chi_{A_{3,i}}\left(1+M_{[\frac{1}{8}\delta],\Rd}\right)^{r}\left|\nabla u\right|^{r}\nonumber \\
 & \leq\left(\fint_{\Ball{\frac{\fr}{2}}{\bQ}\cap\bP}\left(\sum_{i\neq0}\chi_{A_{3,i}}\right)^{\frac{p}{p-r}}\left(1+M_{[\frac{1}{8}\delta],\Rd}\right)^{\frac{rp}{p-r}}\right)^{\frac{p-r}{p}}\left(\fint_{\Ball{\frac{\fr}{2}}{\bQ}\cap\bP}\left|\nabla u\right|^{p}\right)^{\frac{r}{p}}\,.\label{eq:thm:dm-Ext-sto-connected-help-5-0}
\end{align}
\emph{Step 2}: Concerning $I_{2}$, we first observe that for each
$j\neq0$ it holds 
\begin{equation}
\chi_{A_{1,j}\backslash\bP}u_{j,a}\nabla\phi_{0}+\chi_{A_{1,j}\backslash\bP}\sum_{i\neq0}u_{j,a}\nabla\phi_{i}=0\,.\label{eq:proof-main-sum-zero}
\end{equation}
We use $\sum_{j\in\N}\chi_{A_{1,j}}\geq\chi_{A_{1,i}}$ for every
$i\in\N$ together with (\ref{eq:proof-main-sum-zero}) and (\ref{eq:lem:properties-local-rho-convering-2})
to obtain 
\begin{align*}
\fint_{\bQ\backslash\bP}\Phi_{a}\left|\sum_{i\neq0}u_{i,a}\nabla\phi_{i}\right|^{r} & \leq C\fint_{\bQ\backslash\bP}\Phi_{a}\left|\sum_{j\neq0}\chi_{A_{1,j}}\sum_{i\neq0}\left(u_{i,a}-u_{j,a}\right)\nabla\phi_{i}+u_{j,a}\nabla\phi_{0}\right|^{r}\\
 & \stackrel{\eqref{eq:lem:properties-local-rho-convering-2}}{\leq}C\fint_{\bQ\backslash\bP}\Phi_{a}\left(\sum_{j\neq0}\chi_{A_{1,j}}\sum_{i:\,A_{1,i}\cap A_{1,j}\neq\emptyset}\left|u_{i,a}-u_{j,a}\right|^{r}\left|\nabla\phi_{i}\right|^{r}+\left|\sum_{j\neq0}\chi_{A_{1,j}}u_{j,a}\nabla\phi_{0}\right|^{r}\right)\,.
\end{align*}
Note that 
\begin{equation}
\forall a,b,i,j:\qquad u_{i,a}-u_{j,a}=u_{i,b}-u_{j,b}=u_{i}-u_{j}\,.\label{eq:diff-ui-uj}
\end{equation}
Furthermore $u_{i}$ and $u_{j}$ are defined on $A_{2,i}$ and $A_{2,j}$
respectively and $u_{i}=u_{j}$ on $\Ball{\fr_{j}}{p_{j}}$ and $\Ball{\fr_{i}}{p_{i}}$
because of (\ref{eq:lem:properties-local-rho-convering-1}). Furthermore,
both functions can be extended from $A_{2,i}$ and $A_{2,j}$ to $\tilde{u}_{i}$
and $\tilde{u}_{j}$ on $\Ball{4\tilde{\rho}_{i}}{p_{i}}$ and $\Ball{4\tilde{\rho}_{j}}{p_{j}}$
respectively using Lemma \ref{lem:simple-extension} such that for
some $C$ independent from $i,j$ 
\[
k=i,j:\quad\norm{\nabla\tilde{u}_{k}}_{L^{r}\of{\Ball{4\tilde{\rho}_{k}}{p_{k}}}}\leq C\norm{\nabla\tilde{u}_{k}}_{L^{r}\of{A_{2,k}}}\,.
\]
Since now $\tilde{u}_{i}=\tilde{u}_{j}$ on $\Ball{\fr_{j}}{p_{j}}$
and $\Ball{\fr_{i}}{p_{i}}$ we chose $k(i,j)$ such that for $\tilde{M}_{k(i,j)}=1+\min\left\{ M_{\tilde{\rho},i},M_{\tilde{\rho},j}\right\} $
and it holds by the Poincaré inequality (\ref{eq:lem:Poincare-ball-2}),
the microscopic regularity $\alpha$ and the estimate (\ref{eq:rho-neighbor-estim})
\[
\int_{A_{1,i}\cap A_{1,j}}\left|u_{i,a}-u_{j,a}\right|^{r}\left|\nabla\phi_{i}\right|^{r}\leq C\rho_{i}^{-r}\int_{A_{1,k(i,j)}}\left|\tilde{u}_{i}-\tilde{u}_{j}\right|^{r}\leq C\tilde{M}_{k(i,j)}^{\alpha(d-1)}\int_{A_{2,k(i,j)}}\left|\nabla\left(\tilde{u}_{i}-\tilde{u}_{j}\right)\right|^{r}\,.
\]
We obtain with microscopic regularity $\alpha$, the finite covering
(\ref{eq:lem:properties-local-rho-convering-3}) and the proportionality
(\ref{eq:cor:cover-boundary-h1}) that
\begin{align*}
\fint_{\bQ\backslash\bP}\sum_{a}\Phi_{a}\chi_{A_{1,j}} & \sum_{i:\,A_{1,i}\cap A_{1,j}\neq\emptyset}\left|u_{i,a}-u_{j,a}\right|^{r}\left|\nabla\phi_{i}\right|^{r}=\fint_{\bQ\backslash\bP}\chi_{A_{1,j}}\sum_{i:\,A_{1,i}\cap A_{1,j}\neq\emptyset}\left|\tilde{u}_{i}-\tilde{u}_{j}\right|^{r}\left|\nabla\phi_{i}\right|^{r}\\
 & \leq\frac{C}{\left|\bQ\right|}\sum_{i:\,A_{1,i}\cap A_{1,j}\neq\emptyset}\tilde{M}_{k(i,j)}^{\alpha(d-1)}\int_{A_{2,j}}\left|\nabla\left(\tilde{u}_{i}-\tilde{u}_{j}\right)\right|^{r}\\
 & \leq\frac{C}{\left|\bQ\right|}\sum_{i:\,A_{1,i}\cap A_{1,j}\neq\emptyset}\tilde{M}_{k(i,j)}^{\alpha(d-1)}\left(\int_{A_{2,i}}\left|\nabla\tilde{u}_{i}\right|^{r}+\int_{A_{2,j}}\left|\nabla\tilde{u}_{j}\right|^{r}\right)\\
 & \leq\frac{C}{\left|\bQ\right|}\sum_{i:\,A_{1,i}\cap A_{1,j}\neq\emptyset}\left(\int_{A_{3,i}\cup A_{3,j}}\tilde{M}_{[\frac{1}{8}\delta],\Rd}^{r}\left(1+M_{[\tilde{\rho}],\Rd}\right)^{\alpha(d-1)}\left|\nabla u\right|^{r}\right)\,.
\end{align*}
Next we estimate from (\ref{eq:bound-nabla-phi-j})
\[
\fint_{\bQ\backslash\bP}\sum_{a}\Phi_{a}\left|\sum_{j\neq0}\chi_{A_{1,j}}u_{j,a}\nabla\phi_{0}\right|^{r}\leq C\fint_{\bQ\backslash\bP}\sum_{a}\Phi_{a}\left(\sum_{j\neq0}\rho_{j}^{-r}\chi_{A_{1,j}}\left|\cU_{j}\left(u-\tau_{j}u\right)\right|^{r}+\left|\sum_{j\neq0}\rho_{j}^{-1}\chi_{A_{1,j}}\left(\tau_{j}u-\cM_{a}u\right)\right|^{r}\right)\,.
\]
Using once more Assumption \ref{assu:M-alpha-bound} and 
\begin{equation}
\nabla\cU_{j}\left(u-\tau_{j}u\right)=\nabla\left(\cU_{j}\left(u-\tau_{j}u\right)+\tau_{j}u\right)=\nabla u_{j}\label{eq:nabla-U-j-u-j-equality}
\end{equation}
 and $\sum_{a}\Phi_{a}=1$ we infer from (\ref{eq:lem:Poincare-ball-2})
\[
C\fint_{\bQ\backslash\bP}\sum_{a}\Phi_{a}\sum_{j\neq0}\rho_{j}^{-r}\chi_{A_{1,j}}\left|\cU_{j}\left(u-\tau_{j}u\right)\right|^{r}\leq\frac{C}{\left|\bQ\right|}\sum_{j\neq0}\left(1+M_{\tilde{\rho},j}\right)^{\alpha(d-1)}\int_{A_{2,j}}\left|\nabla u_{j}\right|^{r}\,.
\]
Now we make use of the extension estimate (\ref{eq:ex-order-estimate-1})
to find
\[
\int_{A_{2,j}}\left|\nabla u_{j}\right|^{r}\leq CM_{\delta,1,j}^{r}\int_{A_{3,j}\cap\bP}\left|\nabla u\right|^{r}
\]
which in total implies for $f_{12}(M)=\left(1+M_{[\frac{1}{8}\delta],\Rd}\right)^{\frac{rp}{p-r}}\left(1+M_{[\tilde{\rho}],\Rd}\right)^{\frac{p\alpha(d-1)}{p-r}}$
\begin{align*}
I_{1}+I_{2} & \leq C\left(\fint_{\Ball{\frac{\fr}{2}}{\bQ}\cap\bP}\left(\sum_{i\neq0}\chi_{A_{3,i}}\right)^{\frac{p}{p-r}}f_{12}(M)\right)^{\frac{p-r}{p}}\left(\fint_{\Ball{\frac{\fr}{2}}{\bQ}\cap\bP}\left|\nabla u\right|^{p}\right)^{\frac{r}{p}}\\
 & \quad+C\fint_{\bQ\backslash\bP}\sum_{a}\Phi_{a}\left|\sum_{j\neq0}\rho_{j}^{-1}\chi_{A_{1,j}}\left(\tau_{j}u-\cM_{a}u\right)\right|^{r}\,.
\end{align*}
Making use of (\ref{eq:lem:properties-local-rho-convering-4}) we
find 
\[
\left|\sum_{i\neq0}\chi_{A_{3,i}}\right|\leq\left(1+M_{[\frac{3\delta}{8},\frac{\delta}{8}],\Rd}\right)^{\hat{d}}\,,
\]
and it only remains to estimate $I_{3}$.

Step 3: We observe with help of $\sum_{a}\nabla\Phi_{a}=0$ and $\sum_{i\neq0}\phi_{i}=\phi_{0}$
that 
\[
\sum_{i\neq0}\sum_{a}u_{i,a}\phi_{i}\nabla\Phi_{a}=\sum_{i\neq0}u_{i}\phi_{i}\sum_{a}\nabla\Phi_{a}+\sum_{a}\cM_{a}u\nabla\Phi_{a}=\sum_{a}\cM_{a}u\nabla\Phi_{a}\,.
\]
and Lemma \ref{lem:conv-sum-0} yields 
\begin{align*}
I_{3} & =\fint_{\bQ\backslash\bP}\left|\phi_{0}\sum_{a}\cM_{a}u\nabla\Phi_{a}\right|^{r}\\
 & \leq\fint_{\bQ\backslash\bP}\left|\sum_{l=1}^{d}\sum_{a:\,\partial_{l}\Phi_{a}>0}\sum_{b:\,\partial_{l}\Phi_{b}<0}\frac{\partial_{l}\Phi_{a}\left|\partial_{l}\Phi_{b}\right|}{D_{l+}^{\Phi}}\left(\cM_{a}u-\cM_{b}u\right)\right|^{r}\,.
\end{align*}
Step 4: Concerning the proof of Lemma \ref{lem:local-delta-M-extension-estimate-sym}
we follow the above lines with the following modifications. 

We use the Nitsche extension operators. Hence, instead of (\ref{eq:ex-order-estimate-1})
we use (\ref{eq:ex-order-estimate-2}). The local extended functions
are called 
\begin{align*}
u_{i} & :=\cU_{i}\left(u-\tau_{i}^{\fs}u\right)+\tau_{i}^{\fs}u &  & \text{on }A_{2,i}\\
u_{i,a} & :=\cU_{i}\left(u-\tau_{i}^{\fs}u\right)+\tau_{i}^{\fs}u-\cM_{a}^{\fs}u &  & \text{on }A_{2,i}\cap\fA_{1,a}
\end{align*}
and (\ref{eq:diff-ui-uj}) remains valid. We find it worth mentioning
that $\nablas\left(\tau_{i}^{\fs}u-\cM_{a}^{\fs}u\right)=0$ and hence
\[
\nablas\left(\phi_{i}\Phi_{a}u_{i,a}\right)=\frac{1}{2}\left(\nabla(\phi_{i}\Phi_{a})\otimes u_{i,a}+u_{i,a}\otimes\nabla(\phi_{i}\Phi_{a})\right)+\phi_{i}\Phi_{a}\nablas\cU_{2,i}\left(u-\tau_{i}^{\fs}u\right)\,.
\]
We furthermore replace Lemma \ref{lem:simple-extension} by Lemma
\ref{lem:uniform-korn-extension} and the Poincaré inequality (\ref{eq:lem:Poincare-ball-2})
by (\ref{eq:lem:Poincare-ball-2-Korn}). Finally we observe that (\ref{eq:nabla-U-j-u-j-equality})
is replaced by 
\[
\nablas\cU_{j}\left(u-\tau_{j}u\right)=\nablas\left(\cU_{j}\left(u-\tau_{j}u\right)+\tau_{j}u\right)=\nablas u_{j}
\]
\end{proof}

\subsection{\label{subsec:Proof-traces}Traces on $\left(\delta,M\right)$-Regular
Sets, Proof of Theorem \ref{thm:uniform-trace-estimate-1}}
\begin{proof}
We use the covering of $\partial\bP$ by $B_{i}:=A_{1,i}^{1}$ and
set $\tilde{\rho}_{i}:=\tilde{\rho}_{1,i}$, $\hat{\rho}_{i}:=\hat{\rho}_{i,5}(p_{k}^{1})$
and write $M_{i}=M_{\hat{\rho}_{i}}(p_{k}^{1})$, $\hat{B}_{i}:=\Ball{\hat{\rho}_{i}}{p_{k}^{1}}$.
Due to Lemma \ref{lem:basic-trace} we find locally
\begin{equation}
\norm{\cT u}_{L^{p_{0}}(\partial\bP\cap B_{k})}\leq C_{p_{0},p_{0}}\tilde{\rho}_{k}^{-\frac{1}{p_{0}}}\sqrt{4M_{k}^{2}+2}^{\frac{1}{p_{0}}+1}\norm u_{W^{1,p_{0}}\left(\hat{B}_{k}\right)}\,.\label{eq:lem:uniform-trace-estimate-2-2}
\end{equation}
We thus obtain 
\begin{multline*}
\frac{1}{\left|\bQ\right|}\int_{\bQ\cap\partial\bP}\left|\sum_{k}\phi_{k}\cT_{k}u\right|^{r}\\
\leq\left(\frac{1}{\left|\bQ\right|}\int_{\Ball{\frac{1}{4}}{\bQ}\cap\partial\bP}\sum_{k}\chi_{B_{k}}\tilde{\rho}_{k}^{-\frac{1}{p_{0}-r}}\right)^{\frac{p_{0}-r}{p_{0}}}\left(\frac{1}{\left|\bQ\right|}\sum_{k}\int_{\Ball{\frac{1}{4}}{\bQ}\cap\partial\bP}\chi_{B_{k}}\tilde{\rho}_{k}\left|\cT_{k}u\right|^{p_{0}}\right)^{\frac{r}{p_{0}}}
\end{multline*}
which yields by the uniform local bound of the covering, $\tilde{\eta}$
defined in Lemma \ref{lem:delta-tilde-construction-estimate}, twice
the application of (\ref{eq:lem:local-delta-M-construction-estimate-2})
and (\ref{eq:lem:uniform-trace-estimate-2-2})
\begin{align*}
\frac{1}{\left|\bQ\right|}\int_{\bQ\cap\partial\bP}\left|\sum_{k}\phi_{k}\cT_{k}u\right|^{r} & \leq C\left(\frac{1}{\left|\bQ\right|}\int_{\bQ\cap\partial\bP}\rho_{5,\Rd}^{-\frac{1}{p_{0}-r}}\right)^{\frac{p_{0}-r}{p_{0}}}\cdot\\
 & \qquad\cdot\left(\frac{1}{\left|\bQ\right|}\int_{\bQ\cap\bP}\sum_{k}\chi_{\hat{B}_{k}}\sqrt{4M_{k}^{2}+2}^{\frac{1}{p_{0}}+1}\left(\left|\nabla u\right|^{p_{0}}+\left|u\right|^{p_{0}}\right)\right)^{\frac{r}{p_{0}}}\,.
\end{align*}
With Hölders inequality and replacing $M_{k}$ by $M_{[\frac{1}{32}\delta],\Rd}$,
the last estimate leads to (\ref{eq:lem:uniform-trace-estimate-1-1}).
The second estimate goes analogue since the local covering by $A_{2,k}$
is finite.
\end{proof}

\section{\label{subsec:The-Issue-of-Connectedness}The Issue of Connectedness}
\begin{rem}
The following Lemmas \ref{lem:6-4} and \ref{lem:6-5} also hold with
$\tau_{i}$ and $\cM_{a}$ replaced by $\tau_{i}^{\fs}$ and $\cM_{a}^{\fs}$
respectively.
\end{rem}

\begin{lem}
\label{lem:6-4}Under Assumptions \ref{assu:M-alpha-bound}, \ref{assu:mesoscopic-voronoi}
let $\left(f_{j}\right)_{j\in\N}$ be non-negative and have support
$\support f_{j}\supset\Ball{\frac{\fr}{2}}{x_{j}}$ and let $\sum_{j\in\N}f_{j}\equiv1$.
Writing $\X(\bQ):=\left\{ x_{j}:\;\support f_{j}\cap\bQ\neq\emptyset\right\} $,
and 
\begin{align*}
F_{s,\iota}^{1}(\bQ) & :=\left(\frac{1}{\left|\bQ\right|}\int_{\bP\cap\bQ_{\fr}\cap\Rd_{3}}\left|\tilde{\rho}_{\Rd}\right|^{-\frac{sr}{s-r}}\tilde{M}^{2-\iota}\right)^{\frac{s-r}{s}}\\
F_{s,\tilde{s},\iota}^{2}(\bQ) & :=\left(\frac{1}{\left|\bQ\right|}\int_{\Ball{\frac{\fr}{2}}{\bQ}\backslash\bP}\tilde{M}^{\frac{\left(\iota-2\right)(\tilde{s}-r)}{r(s-\tilde{s})}}\right)^{r\frac{s-\tilde{s}}{\tilde{s}s}}\\
F_{s}^{3}(\bQ,u) & :=\left(\frac{1}{\left|\bQ\right|}\int_{\bP\cap\bQ_{\fr}}\sum_{x_{a}\in\X(\bQ)}\Phi_{a}\left|\sum_{\substack{i\neq0:\,\partial_{l}\phi_{i}\partial_{l}\phi_{0}<0}
}\chi_{A_{1,i}}\left(\tau_{i}u-\cM_{a}u\right)\right|^{s}\right)^{\frac{r}{s}}\\
F_{s}^{3,\fs}(\bQ,u) & :=\left(\frac{1}{\left|\bQ\right|}\int_{\bP\cap\bQ_{\fr}}\sum_{x_{a}\in\X(\bQ)}\Phi_{a}\left|\sum_{\substack{i\neq0:\,\partial_{l}\phi_{i}\partial_{l}\phi_{0}<0}
}\chi_{A_{1,i}}\left(\tau_{i}^{\fs}u-\cM_{a}^{\fs}u\right)\right|^{s}\right)^{\frac{r}{s}}
\end{align*}
for every $l=1,\dots d$ and $r<\tilde{s}<s$ it holds 
\begin{align*}
\frac{1}{\left|\bQ\right|}\int_{\bQ\backslash\bP}\sum_{a}\Phi_{a}\left|\sum_{i\neq0}\rho_{1,i}^{-1}\chi_{A_{1,i}}\left(\tau_{n,\alpha,i}u-\cM_{a}u\right)\right|^{r} & \leq\begin{cases}
F_{s,2}^{1}(\bQ)\,F_{s}^{3}(\bQ)\\
F_{s,d}^{1}(\bQ)\,F_{s,\tilde{s},d}^{2}(\bQ)\,F_{s}^{3}(\bQ,u)
\end{cases}\,,
\end{align*}
and
\begin{align*}
\frac{1}{\left|\bQ\right|}\int_{\bQ\backslash\bP}\sum_{a}\Phi_{a}\left|\sum_{i\neq0}\rho_{1,i}^{-r}\chi_{A_{1,i}}\left(\tau_{n,\alpha,i}^{\fs}u-\cM_{a}^{\fs}u\right)\right|^{r} & \leq\begin{cases}
F_{s,2}^{1}(\bQ)\,F_{s}^{3,\fs}(\bQ)\\
F_{s,d}^{1}(\bQ)\,F_{s,\tilde{s},d}^{2}(\bQ)\,F_{s}^{3,\fs}(\bQ,u)
\end{cases}\,.
\end{align*}
\end{lem}

\begin{proof}
We find from Hölder's and Jensen's inequality
\begin{align*}
 & \frac{1}{\left|\bQ\right|}\int_{\bP\cap\bQ}\sum_{\substack{i\neq0:\,\partial_{l}\phi_{i}\partial_{l}\phi_{0}<0}
}\sum_{a}\rho_{1,i}^{-r}\frac{\left|\partial_{l}\phi_{i}\right|}{D_{l+}}\chi_{\fA_{1,a}}\left|\tau_{i}u-\cM_{a}u\right|^{r}\\
 & \qquad\leq\begin{cases}
F_{s,2}^{1}(\bQ)\,F_{s}^{3}(\bQ)\\
F_{s,d}^{1}(\bQ)\,F_{s,\tilde{s},d}^{2}(\bQ)\,F_{s}^{3}(\bQ)
\end{cases}\,.
\end{align*}
The second part follows accordingly.
\end{proof}
\begin{lem}
\label{lem:6-5}Under Assumptions \ref{assu:M-alpha-bound}, \ref{assu:mesoscopic-voronoi}
for every $l=1,\dots d$ and $\tilde{\alpha}>0$ it holds
\begin{align*}
\frac{1}{\left|\bQ\right|}\int_{\bP\cap\bQ} & \left|\sum_{k:\,\partial_{l}\Phi_{k}>0}\sum_{j:\,\partial_{l}\Phi_{j}<0}\frac{\partial_{l}\Phi_{k}\left|\partial_{l}\Phi_{j}\right|}{D_{l+}^{\Phi}}\left(\cM_{k}u-\cM_{j}u\right)\right|^{r}\\
 & \leq\left(\frac{1}{\left|\bQ\right|}\int_{\bP\cap\bQ}\left(\sum_{j:\,\partial_{l}\Phi_{j}<0}d_{j}^{\frac{\tilde{\alpha}s+drs}{s-r}}\chi_{\nabla\Phi_{j}\neq0}\right)^{\frac{s}{s-r}}\right)^{\frac{s-r}{s}}\cdot\dots\\
 & \quad\dots\cdot\left(\frac{1}{\left|\bQ\right|}\int_{\bP\cap\bQ}\sum_{k:\,\partial_{l}\Phi_{k}>0}\sum_{j:\,\partial_{l}\Phi_{j}<0}\chi_{\nabla\Phi_{j}\neq0}\frac{d_{j}^{-\tilde{\alpha}\frac{s}{r}}\left|\partial_{l}\Phi_{k}\right|}{D_{l+}^{\Phi}}\left|\cM_{k}u-\cM_{j}u\right|^{s}\right)^{\frac{r}{s}}\,,
\end{align*}
with the similar formula holding for $\cM_{\bullet}$ replaced by
$\cM_{\bullet}^{\fs}$.
\end{lem}

\begin{proof}
We observe with help of (\ref{eq:estiamte-nabla-Phi-i}) and with
Lemma \ref{lem:Iso-cone-geo-estimate}.2) 
\begin{align}
\forall x:\quad\sup_{k}\left|\partial_{l}\Phi_{k}\right|(x) & \leq\sup\left\{ \left|\nabla\Phi_{k}(x)\right|\,:\;x\in\Ball{\frac{\fr}{2}}{G_{k}}\right\} \nonumber \\
 & \leq C\sup\left\{ d_{k}^{d}\,:\;x\in G_{k}\right\} \,,\label{eq:thm:all-in-help-1}\\
\sup_{x\in\Ball{\frac{\fr}{2}}{G_{j}}}\left|\partial_{l}\Phi_{j}\right|(x) & \leq Cd_{j}^{d}\,.\label{eq:thm:all-in-help-2}
\end{align}
We write 
\[
I:=\frac{1}{\left|\bQ\right|}\int_{\bP\cap\bQ}\left|\sum_{k:\,\partial_{l}\Phi_{k}>0}\sum_{j:\,\partial_{l}\Phi_{j}<0}\frac{\partial_{l}\Phi_{k}\left|\partial_{l}\Phi_{j}\right|}{D_{l+}^{\Phi}}\left(2-\phi_{0}\right)\left(\cM_{k}u-\cM_{j}u\right)\right|^{r}
\]
and find 
\begin{align*}
I & \leq C\frac{1}{\left|\bQ\right|}\int_{\bP\cap\bQ}\sum_{k:\,\partial_{l}\Phi_{k}>0}\sum_{j:\,\partial_{l}\Phi_{j}<0}\frac{\left|\partial_{l}\Phi_{k}\right|^{r}\left|\partial_{l}\Phi_{j}\right|}{D_{l+}^{\Phi}}\left|\cM_{k}u-\cM_{j}u\right|^{r}\\
 & \leq C\frac{1}{\left|\bQ\right|}\int_{\bP\cap\bQ}\left(\sum_{k:\,\partial_{l}\Phi_{k}>0}\sum_{j:\,\partial_{l}\Phi_{j}<0}\frac{d_{j}^{\alpha\frac{s}{s-r}}\left|\partial_{l}\Phi_{k}\right|^{\frac{sr}{s-r}}\left|\partial_{l}\Phi_{j}\right|}{D_{l+}^{\Phi}}\right)^{\frac{s-r}{s}}\cdot\dots\\
 & \qquad\dots\cdot\left(\sum_{k:\,\partial_{l}\Phi_{k}>0}\sum_{j:\,\partial_{l}\Phi_{j}<0}\chi_{\nabla\Phi_{j}\neq0}\frac{d_{j}^{-\alpha\frac{s}{r}}\left|\partial_{l}\Phi_{k}\right|}{D_{l+}^{\Phi}}\left|\cM_{k}u-\cM_{j}u\right|^{s}\right)^{\frac{r}{s}}\,.
\end{align*}
Now we make use of (\ref{eq:thm:all-in-help-1}) and once more of
Lemma \ref{lem:Iso-cone-geo-estimate}.2) to obtain for the first
bracket on the right hand side an estimate of the form 
\[
\left|\partial_{l}\Phi_{k}\right|^{\frac{sr}{s-r}}\left|\partial_{l}\Phi_{j}\right|\leq\left|\partial_{l}\Phi_{k}\right|\left|\partial_{l}\Phi_{k}\right|^{\frac{sr}{s-r}-1}\left|\partial_{l}\Phi_{j}\right|\leq C\left|\partial_{l}\Phi_{k}\right|d_{j}^{d\frac{sr-s+r}{s-r}}d_{j}^{d}\leq C\left|\partial_{l}\Phi_{k}\right|d_{j}^{d\frac{sr}{s-r}}\,,
\]
which implies 
\begin{align*}
\sum_{k:\,\partial_{l}\Phi_{k}>0}\sum_{j:\,\partial_{l}\Phi_{j}<0}\frac{d_{j}^{\alpha\frac{s}{s-r}}\left|\partial_{l}\Phi_{k}\right|^{\frac{sr}{s-r}}\left|\partial_{l}\Phi_{j}\right|}{D_{l+}^{\Phi}} & \leq C\sum_{k:\,\partial_{l}\Phi_{k}>0}\sum_{j:\,\partial_{l}\Phi_{j}<0}\frac{d_{j}^{\alpha\frac{s}{s-r}}d_{j}^{\frac{dsr}{s-r}}\left|\partial_{l}\Phi_{k}\right|}{D_{l+}^{\Phi}}\\
 & \leq C\sum_{j:\,\partial_{l}\Phi_{j}<0}d_{j}^{\alpha\frac{s}{s-r}}d_{j}^{\frac{dsr}{s-r}}\chi_{\nabla\Phi_{j}\neq0}\,,
\end{align*}
where we used $\sum\left|\partial_{l}\Phi_{k}\right|=D_{l+}^{\Phi}$.
From Hölder's inequality the Lemma follows.
\end{proof}

\section{\label{sec:Sample-Geometries}Sample Geometries}

\subsection{\label{subsec:Matern-Process}Delaunay Pipes for a Matern Process}

For two points $x,y\in\Rd$, we denote 
\[
P_{r}(x,y):=\left\{ y+z\in\Rd:\,0\leq z\cdot(x-y)\leq\left|x-y\right|^{2},\,\left|z-z\cdot(x-y)\frac{x-y}{\left|x-y\right|}\right|<r\right\} \,,
\]
the cylinder (or pipe) around the straight line segment connecting
$x$ and $y$ with radius $r>0$.

Recalling Example \ref{exa:poisson-point-proc} we consider a Poisson
point process $\X_{\pois}(\omega)=\left(x_{i}(\omega)\right)_{i\in\N}$
with intensity $\lambda$ (recall Example \ref{exa:poisson-point-proc})
and construct a hard core Matern process $\X_{\mat}$ by deleting
all points with a mutual distance smaller than $d\fr$ for some $\fr>0$
(refer to Example \ref{exa:Matern}). From the remaining point process
$\X_{\mat}$ we construct the Delaunay triangulation $\D(\omega):=\D(X_{\mat}(\omega))$
and assign to each $(x,y)\in\D$ a random number $\delta(x,y)$ in
$(0,\fr)$ in an i.i.d. manner from some probability distribution
$\delta(\omega)$. We finally define 
\[
\bP(\omega):=\bigcup_{(x,y)\in\D(\omega)}P_{\delta(x,y)}(x,y)\bigcup_{x\in\X_{\mat}}\Ball{\frac{\fr}{2}}x
\]
 the family of all pipes generated by the Delaunay grid ``smoothed''
by balls with the fix radius $\fr$ around each point of the generating
Matern process.

Since the Matern process is mixing and $\delta$ is mixing, Lemma
\ref{lem:erg-and-mix-is-erg} yields that the whole process is still
ergodic. We start with a trivial observation.
\begin{cor}
\label{cor:pipes-micro-regularity}The microscopic regularity of $\bP$
is $\alpha=0$ (Def. \ref{assu:M-alpha-bound}) and it holds $\hat{d}=d-1$
in Lemma \ref{lem:properties-local-rho-convering}. Furthermore both
the extension order and the symmetric extension order are $n=0$.
\end{cor}

\begin{proof}
This follows from the fact that $\partial\bP$ can be locally represented
as a graph in the upper half space with $\bP$ filling the lower half
space. 
\end{proof}
\begin{lem}
For the Voronoi tessellation $\left(G_{a}\right)_{a\in\N}$ corresponding
to $\X_{\mat}$ holds 
\[
\P(d_{a}\geq D)\leq\exp\of{-\lambda\left|\S^{d-1}\right|\left(4D\right)^{d}\left(1-e^{-\lambda\left|\S^{d-1}\right|\left(d\fr\right)^{d}}\right)}\,.
\]
\end{lem}

\begin{proof}
For the underlying Poisson point process $\X_{\pois}$ it holds for
the void probability inside a ball $\Ball Rx$
\[
\P\of{\X_{\pois}(\Ball Rx)=0}=\P_{R,0}:=e^{-\lambda\left|\S^{d-1}\right|R^{d}}\,.
\]
The probability for a point $x\in\X_{\pois}$ to be removed is thus
$1-\P_{d\fr,0}$ and is i.i.d distributed among points of $\X_{\pois}$.
The total probability to not find any point of $\X_{\mat}$ is thus
given by not finding a point of $\X_{\pois}$ plus the probability
that all points of $\X_{\pois}$ are removed, i.e. 
\begin{align*}
\P\of{\X_{\mat}(\Ball Rx)=0} & =\sum_{n=0}^{\infty}e^{-\lambda\left|A\right|}\frac{\lambda^{n}\left|A\right|^{n}}{n!}\left(1-\P_{d\fr,0}\right)^{n}\\
 & =\exp\of{-\lambda\left|A\right|+\lambda\left|A\right|\left(1-\P_{d\fr,0}\right)}=e^{-\lambda\left|A\right|\left(1-\P_{d\fr,0}\right)}\,.
\end{align*}
From here one concludes.
\end{proof}
\begin{rem}
The family of balls $\Ball{\fr}x$ can also be dropped from the model.
However, this would imply we had to remove some of the points from
$\X_{\mat}$ for the generation of the Voronoi cells. This would cause
technical difficulties which would not change much in the result,
as the probability for the size of Voronoi cells would still decrease
exponentially.
\end{rem}

\begin{lem}
\label{lem:matern-delaunay-distribution}$\X_{\mat}$ is a point process
for $\bP(\omega)$ that satisfies Assumption \ref{assu:mesoscopic-voronoi}
and $\bP$ is isotropic cone mixing for $\X_{\mat}$ with exponentially
decreasing $f(R)\leq Ce^{-R^{d}}$ and it holds $n=0$ and $\alpha=0$.
Furthermore, assume there exists $C_{\delta},a_{\delta}>0$ such that
$\P(\delta(x,y)<\delta_{0})\leq C_{\delta}e^{-a_{\delta}\frac{1}{\delta_{0}}}$,
then $\P(\tilde{M}>M_{0})\leq Ce^{-aM_{0}}$ for some $C,a>0$. If
$\P(\delta(x,y)<\delta_{0})\leq C_{\delta}\delta_{0}^{\beta}$ then
for every $R\in(0,\infty)$ it holds 
\begin{equation}
\E\of{M_{[\frac{\delta}{2}],\Rd}^{R}}+\E\of{\tilde{\delta}_{\Rd}^{-(\beta+d-1)}}<C\E\of{\left|x-y\right|},\label{eq:lem:matern-delaunay-distribution}
\end{equation}
where $\E\of{\left|x-y\right|}$ is the expectation of the length
of pipes.
\end{lem}

\begin{proof}
\emph{Isotropic cone mixing:} For $x,y\in2d\fr\Zd$ the events $\left(x+[0,1]^{d}\right)\cap\X_{\mat}$
and $\left(y+[0,1]^{d}\right)\cap\X_{\mat}$ are mutually independent.
Hence
\[
\P\of{\left(k2dr\,[-1,1]^{d}\right)\cap\X_{\mat}=\emptyset}\leq\P\of{[-1,1]^{d}\cap\X_{\mat}=\emptyset}^{k^{d}}\,.
\]
Hence the open set $\bP$ is isotropic cone mixing for $\X=\X_{\mat}$
with exponentially decaying $f(R)\leq Ce^{-R^{d}}$.

\emph{Estimate on the distribution of $M$:} By definition of the
Delaunay triangulation, two pipes intersect only if they share one
common point $x\in\X_{\mat}$.

Given three points $x,y,z\in\X_{\mat}$ with $x\sim y$ and $x\sim z$,
the highest local Lipschitz constant on $\partial\left(P_{\delta(x,y)}(x,y)\cup P_{\delta(x,z)}(x,z)\right)$
is attained in 
\[
\tilde{x}=\arg\max\left\{ \left|x-\tilde{x}\right|:\,\tilde{x}\in\partial P_{\delta(x,y)}(x,y)\cap\partial P_{\delta(x,z)}(x,z)\right\} \,.
\]
It is bounded by 
\[
\max\left\{ \arctan\left(\frac{1}{2}\sphericalangle\left((x,y),(x,z)\right)\right),\,\frac{1}{\delta(x,y)},\,\frac{1}{\delta(x,z)}\right\} \,,
\]
where $\alpha:=\sphericalangle\left((x,y),(x,z)\right)$ in the following
denotes the angle between $(x,y)$ and $(x,z)$, see Figure \ref{fig:Voronoi-delaunay-proof}.
If $d_{x}$ is the diameter of the Voronoi cell of $x$, we show that
a necessary (but not sufficient) condition that the angle $\alpha$
can be smaller than some $\alpha_{0}$ is given by 
\begin{equation}
d_{x}\geq C\frac{1}{\sin\alpha_{0}}\,,\label{eq:lem:matern-delaunay-distribution-1}
\end{equation}
where $C>0$ is a constant depending only on the dimension $d$. Since
for small $\alpha$ we find $M\approx\frac{1}{\sin\alpha}$, and since
the distribution for $d_{x}$ decays subexponentially, also the distribution
for $M$ at the junctions of two pipes decays subexponentially. However,
inside the pipes, we find $\Delta(p)=2\delta(x,y)$ and hence $\delta_{\Delta}(p)=\delta(x,y)$.
Due to the cylindric structure, we furthermore find essential boundedness
of $M$. This also implies $\alpha=n=0$ inside the pipes. At the
junction of Balls and pipes we find $\partial\bP$ to be in the upper
half of the local plane approximation and hence also here $\alpha=n=0$
can be chosen (see also Remarks \ref{rem:ext-order} and \ref{rem:ext-order-sym}).

Concerning the expectation of $M_{[\frac{\delta}{2}],\Rd}$ and $\delta_{\Rd}$,
we only have to accound for the pipes by the above argumentation since
the other contribution to $M$ is exponentially distributed. In particular,
we find for one single pipe $P_{\delta(x,y)}(x,y)$ that
\[
\int_{P_{\delta(x,y)}(x,y)}\delta_{\Rd}^{-\alpha-d+1}\leq C\left|x-y\right|\delta(x,y)^{-\alpha}\,,
\]
and hence (\ref{eq:lem:matern-delaunay-distribution}) due to the
independence of length and diameter. It thus remains to proof (\ref{eq:lem:matern-delaunay-distribution-1}).
\begin{figure}
 \begin{minipage}[c]{0.6\textwidth} \includegraphics[width=6.5cm]{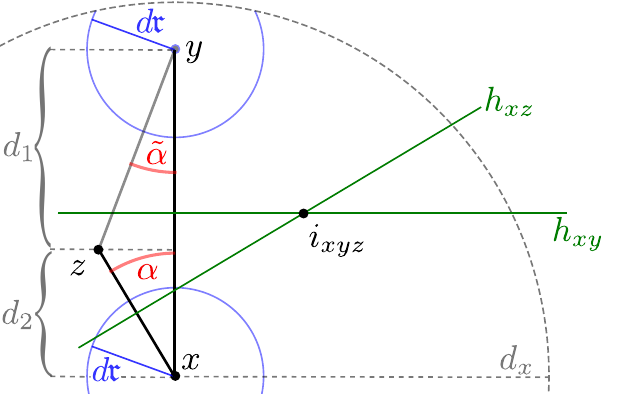}\end{minipage}\hfill   \begin{minipage}[c]{0.35\textwidth}\caption{\label{fig:Voronoi-delaunay-proof}Sketch of the proof of Lemma \ref{lem:matern-delaunay-distribution}
and estimate (\ref{eq:lem:matern-delaunay-distribution-1})\emph{.}}
\end{minipage}
\end{figure}

\emph{Proof of (\ref{eq:lem:matern-delaunay-distribution-1}):} Given
an angle $\alpha>0$ and $x\in\X_{\mat}$ we derive a lower bound
for the diameter of $G(x)$ such that for two neighbors $y,z$ of
$x$ it can hold $\sphericalangle\left((x,y),(x,z)\right)\leq\alpha$.
With regard to Figure \ref{fig:Voronoi-delaunay-proof}, we assume
$\left|x-y\right|\geq\left|x-z\right|$.

Writing $d_{x}:=d(x)$ the diameter of $G(x)$ and $\tilde{\alpha}=\sphericalangle\left((x,z),(z,y)\right)$,
w.l.o.g let $y=(d_{1}+d_{2},0,\dots,0)$, where $d_{1}+d_{2}<d_{x}$
and $d_{1}=\left|y-z\right|\cos\tilde{\alpha}$. Hence we can assume
$z=(d_{2},-\left|y-z\right|\sin\tilde{\alpha},0\dots0)$ and in what
follows, we focus on the first two coordinates only. The boundaries
between the cells $x$ and $z$ and $x$ and $y$ lie on the planes
\[
h_{xz}(t)=\frac{1}{2}z+t\left(\begin{array}{c}
\left|y-z\right|\sin\tilde{\alpha}\\
d_{2}
\end{array}\right)\,,\quad h_{xy}(s)=\frac{1}{2}y+s\left(\begin{array}{c}
0\\
1
\end{array}\right)
\]
respectively. The intersection of these planes has the first two coordinates
\[
i_{xyz}:=\left(\frac{d_{1}+d_{2}}{2},-\frac{1}{2}\left|y-z\right|\sin\tilde{\alpha}+\frac{1}{2}\frac{d_{1}d_{2}}{\left|y-z\right|\sin\tilde{\alpha}}\right)\,.
\]
Using the explicit form of $d_{2}$, the latter point has the distance
\[
\left|i_{xyz}\right|^{2}=\frac{1}{4}\left|y-z\right|^{2}+\frac{1}{4}d_{2}^{2}+\frac{1}{4}\frac{d_{2}^{2}\cos^{2}\tilde{\alpha}}{\sin^{2}\tilde{\alpha}}
\]
to the origin $x=0$. Using $\left|y-z\right|\sin\tilde{\alpha}=\left|z\right|\sin\alpha$
and $d_{2}=\left|y\right|-\left|z\right|\cos\alpha$ we obtain 
\[
\left|i_{xyz}\right|^{2}=\frac{1}{4}\left(\left|y-z\right|^{2}\left(1+\frac{\left(\left|y\right|-\left|z\right|\cos\alpha\right)^{2}\cos^{2}\tilde{\alpha}}{\left|z\right|^{2}\sin^{2}\alpha}\right)+\left(\left|y\right|-\left|z\right|\cos\alpha\right)^{2}\right)\,.
\]
Given $y$, the latter expression becomes small for $\left|y-z\right|$
small, with the smallest value being $\left|y-z\right|=d\fr$. But
then 
\[
\cos^{2}\tilde{\alpha}=1-\sin^{2}\tilde{\alpha}=1-\frac{\left(\left|z\right|\sin\alpha\right)^{2}}{\left|y-z\right|^{2}}
\]
and hence the distance becomes 
\[
\left|i_{xyz}\right|^{2}=\frac{1}{4}\left(\left(d\fr\right)^{2}\left(1+\frac{\left(\left|y\right|-\left|z\right|\cos\alpha\right)^{2}\left(\left(d\fr\right)^{2}+\left|z\right|^{2}\sin^{2}\alpha\right)}{\left(d\fr\right)^{2}\left|z\right|^{2}\sin^{2}\alpha}\right)+\left(\left|y\right|-\left|z\right|\cos\alpha\right)^{2}\right)\,.
\]
We finally use $\left|y\right|=\left|z\right|\cos\alpha-\sqrt{\left(d\fr\right)^{2}-\left|z\right|^{2}\sin^{2}\alpha}$
and obtain 
\[
\left|i_{xyz}\right|^{2}=\frac{1}{4}\left(\left(d\fr\right)^{2}\left(1+\frac{\left(\left(d\fr\right)^{4}-\left|z\right|^{4}\sin^{4}\alpha\right)}{\left(d\fr\right)^{2}\left|z\right|^{2}\sin^{2}\alpha}\right)+\left(\left(d\fr\right)^{2}-\left|z\right|^{2}\sin^{2}\alpha\right)\right)\,.
\]
The latter expression now needs to be smaller than $d_{x}$. We observe
that the expression on the right hand side decreases for fixed $\alpha$
if $\left|z\right|$ increases.

On the other hand, we can resolve $\left|z\right|(y)=\left|y\right|\cos\alpha-\sqrt{\left|y\right|^{2}\sin^{2}\alpha+\left(d\fr\right)^{2}}$.
From the conditions $\left|y\right|\leq d_{x}$ and $\left|i_{xyz}\right|\leq d_{x}$,
we then infer (\ref{eq:lem:matern-delaunay-distribution-1}).
\end{proof}
\begin{thm}
\label{thm:Delaunay-final}Assuming $\E\of{\delta^{-s-d}+\delta^{1+s-2d}}^{\frac{p}{p-s}}<\infty$
and using the notation of Lemma \ref{lem:6-4} the above constructed
$\bP$ has the property that for $1\leq r<s<p$ there almost surely
exists $C>0$ such that for every $n\in\N$ and every $u\in W_{0,\partial(n\bQ)}^{1,p}(\bP\cap n\bQ)$
\begin{align*}
\left(\frac{1}{\left|n\bQ\right|}\int_{\bP\cap n\bQ}\sum_{k:\,\partial_{l}\Phi_{k}>0}\sum_{j:\,\partial_{l}\Phi_{j}<0}\frac{d_{j}^{-\tilde{\alpha}\frac{s}{r}}\left|\partial_{l}\Phi_{j}\right|}{D_{l+}^{\Phi}}\left|\cM_{k}u-\cM_{j}u\right|^{s}\right)^{\frac{r}{s}}+F_{s}^{3}(n\bQ,u) & \leq C\left(\frac{1}{\left|n\bQ\right|}\int_{\bP\cap n\bQ}\left|\nabla u\right|^{p}\right)^{\frac{r}{p}}\,,
\end{align*}
and for every $u\in\bW_{0,\partial(n\bQ)}^{1,p}(\bP\cap n\bQ)$
\begin{align*}
\left(\frac{1}{\left|n\bQ\right|}\int_{\bP\cap n\bQ}\sum_{k:\,\partial_{l}\Phi_{k}>0}\sum_{j:\,\partial_{l}\Phi_{j}<0}\frac{d_{j}^{-\tilde{\alpha}\frac{s}{r}}\left|\partial_{l}\Phi_{j}\right|}{D_{l+}^{\Phi}}\left|\cM_{k}^{\fs}u-\cM_{j}^{\fs}u\right|^{s}\right)^{\frac{r}{s}}+F_{s}^{3,\fs}(n\bQ,u) & \leq C\left(\frac{1}{\left|n\bQ\right|}\int_{\bP\cap n\bQ}\left|\nablas u\right|^{p}\right)^{\frac{r}{p}}\,.
\end{align*}
\end{thm}

\begin{lem}
\label{lem:matern-delaunay-bounded}For every bounded open set $\bQ$
with $0\in\bQ$ and $n_{0},n_{1}\in\N$ let 
\[
\forall M>1:\qquad\tilde{\bQ}_{M,n_{0},n_{1}}:=\bigcup_{\substack{x_{a}\in\X_{\mat}\\
\Ball{n_{0}d_{a}}{x_{a}}\cap M\bQ\neq\emptyset
}
}\Ball{n_{1}d_{a}}{x_{a}}\,.
\]
Then for fixed $n_{0}$ and $n_{1}$ there almost surely exists $r>0$
such that for every $M>1$ it holds $\tilde{\bQ}_{M,n_{0},n_{1}}\subset Mr\bQ$
\end{lem}

\begin{proof}
There exists $r_{0}<R$ such that $\Ball{r_{0}}0\subset\bQ\subset\Ball R0$
we assume w.l.o.g $\bQ=\Ball R0$. We denote $\bQ_{M}:=M\bQ$ and
observe that $\frac{\left|\partial\bQ_{M}\right|}{\left|\bQ_{M}\right|}\leq C\,M^{-1}$
where $\left|\partial\bQ_{M}\right|:=\hausdorffH^{d-1}(\partial\bQ_{M})$.
For 
\[
\bQ_{M,a,b}:=\left\{ x\in\Rd\backslash\bQ_{M}:\,a<\dist(x,\bQ_{M})<b\right\} \,,
\]
we observe that $\#\of{\bQ_{M,a,b}\cap\X_{\mat}}\leq CM^{d-1}(b-a)$
due to the minimal mutual distance. The probability that at least
one $x\in\bQ_{M,a,b}\cap\X_{\mat}$ satisfies $\Ball{n_{0}d(x)}x\cap\bQ_{M}\neq\emptyset$
is given by 
\begin{align*}
\P\of{\bQ_{M},a,b}: & =\P\of{\exists x\in\bQ_{M,a,b}\cap\X_{\mat}:\;\Ball{n_{0}d(x)}x\cap\bQ_{M}\neq\emptyset}\\
 & =\sum_{k=1}^{\infty}k\,\P\of{k=\#\bQ_{M,a,b}\cap\X_{\mat}}\,\P\of{d>\frac{a}{n_{0}}}\\
 & \leq\P\of{d>\frac{a}{n_{0}}}e^{-\lambda\left|\bQ_{M,a,b}\right|}\sum_{k=1}^{\infty}\frac{\lambda^{k}\left|\bQ_{M,a,b}\right|^{k}}{\left(k-1\right)!}=\P\of{d>\frac{a}{n_{0}}}\lambda\left|\bQ_{M,a,b}\right|\,.
\end{align*}
Now let $r>0$ and observe $\left|\bQ_{M,a,b}\right|\leq C\,(b-a)\,\left(b+MR\right)^{d-1}$
while $\P\of{d>\frac{a}{n_{0}}}\leq Ce^{-\alpha\,a^{d}}$. Then the
probability that there exists $x\in\X_{\mat}\backslash\bQ_{rM}$ such
that $\Ball{n_{0}d(x)}x\cap\bQ_{M}$ is smaller than 
\begin{align*}
\sum_{k=0}^{\infty}\P & \of{\bQ_{M},(r-1)M+k,(r-1)M+k+1}\\
 & =\sum_{k=0}^{\infty}\P\of{d>\frac{(r-1)M+k}{n_{0}}}\lambda\left((rM+k+1)^{d}-(rM+k)^{d}\right)\\
 & \leq e^{-\alpha((r-1)M)^{d}}(rM)^{d}\sum_{k=0}^{\infty}e^{-\alpha k^{d}}\lambda(k+2)^{d}\,,
\end{align*}
and the right hand side tends uniformly to $0$ as $r\to\infty$. 
\end{proof}

\begin{proof}[Proof of Theorem \ref{thm:Delaunay-final}]
In what follows, we will mostly perform the calculations for $\tau_{i}^{\fs}$
and $\cM_{a}^{\fs}$ since these calculations are more involved and
drop $n$ except for the last Step 4. 

We first estimate the difference $\left|\cM_{a}^{\fs}u-\cM_{b}^{\fs}u\right|$
for two directly neighbored points $x_{a}\sim x_{b}$ of the Delaunay
grid. These are connected through a cylindric pipe 
\[
P_{\delta,a,b}=P(x_{a},x_{b},\delta(a,b)):=\conv\of{\Ball{\delta(a,b)}{x_{a}}\cup\Ball{\delta(a,b)}{x_{b}}}
\]
with round ends and of thickness $\delta(a,b)$ and total length $\left|x_{a}-x_{b}\right|+2\delta(a,b)<2\left|x_{a}-x_{b}\right|$
and we first introduce the new averages in the spirit of (\ref{eq:lem:Mixed-Korn-zylinder-cM})
\[
\cM_{a}^{\delta}u:=\fint_{\Ball{\delta}{x_{a}}}u\,,\qquad\cM_{a}^{\fs,\delta}u(x):=\overline{\nabla_{a,\delta}^{\bot}}u\of{x-a}+\fint_{\Ball{\delta}{x_{a}}}u\,.
\]
As for (\ref{eq:permutation-tau-cM}) and (\ref{eq:permutation-tau-cM-fs})
we obtain 
\[
\cM_{a_{1}}^{\delta_{1}}\cM_{a_{2}}^{\delta_{2}}u=\cM_{a_{2}}^{\delta_{2}}u\,,\qquad\cM_{a_{1}}^{\fs,\delta_{1}}\cM_{a_{2}}^{\fs,\delta_{2}}u=\cM_{a_{2}}^{\fs,\delta_{2}}u\,.
\]
For every $i,a\in\N$ with $p_{i}\in\Ball{\fr}{G_{a}}$ there exists
almost surely $a_{i}\in\N$ such that $p_{i}$ and $x_{a_{i}}$ are
connected in $\bP$ through a straight line segment  (i.e. $p_{i}$
lies on the boundary of one of the pipes emerging at $x_{a_{i}}$
or in $\Ball{\fr}{x_{a_{i}}}$) and 
\begin{align*}
\left|\tau_{i}u-\cM_{a}u\right|^{s} & \leq2^{s}\left(\left|\tau_{i}u-\cM_{a_{i}}u\right|^{s}+\left|\cM_{a_{i}}u-\cM_{a}u\right|^{s}\right)\,,\\
\left|\tau_{i}^{\fs}u(x)-\cM_{a}^{\fs}u(x)\right|^{s} & \leq2^{s}\left(\left|\tau_{i}^{\fs}u(x)-\cM_{a_{i}}^{\fs}u(x)\right|^{s}+\left|\cM_{a_{i}}^{\fs}u(x)-\cM_{a}^{\fs}u(x)\right|^{s}\right)\,.
\end{align*}
The second term is of ``mesoscopic type'', while the first term
is of local type. We will study both types of terms separately.

\emph{Step 1: }Using (\ref{eq:lem:Mixed-Korn-zylinder-4})--(\ref{eq:lem:Mixed-Korn-zylinder-5}),
we observe for neighbors $a\sim b$
\begin{align}
\left|\cM_{a}^{\fs}u-\cM_{b}^{\fs}u\right|^{s} & \leq\sum_{k=a,b}\left|\cM_{k}^{\fs}u-\cM_{k}^{\fs,\delta(a,b)}u\right|^{s}+\left|\cM_{a}^{\fs,\delta(a,b)}u-\cM_{b}^{\fs,\delta(a,b)}u\right|^{s}\nonumber \\
 & \leq CF_{s}^{\fs,1}\of{x,\delta(a,b)}\left(\left|x-x_{a}\right|^{s}+\left|x-x_{b}\right|^{s}\right)\nonumber \\
 & \qquad\left(\left\Vert \nablas u\right\Vert _{L^{s}(\Ball{\frac{\fr}{16}}{\left\{ x_{a},x_{b}\right\} })}^{s}+\left|x_{a}-x_{b}\right|^{2s}\left\Vert \nablas u\right\Vert _{L^{s}(P_{\delta,a,b})}^{s}\right)\,.\label{eq:thm:Delaunay-final-help-1}
\end{align}
where 
\begin{equation}
F_{s}^{\fs,q}(x,\delta):=\left(\delta^{-d}+\delta^{-s-d}+\delta^{1+s-2d}\right)^{q}\,.\label{eq:def-F-sq-s}
\end{equation}

\emph{Step 2: }For reasons that we will encounter below, we define
\[
\mathrm{I}_{\alpha}:=\frac{1}{\left|\bQ\right|}\int_{\bP\cap\bQ}\sum_{a:\,\Ball{4d_{a}}{x_{a}}\cap\bQ\neq\emptyset}\chi_{\Ball{\fr}{G_{a}}}\sum_{\substack{b:\,\Ball{4d_{b}}{x_{b}}\cap\bQ\neq\emptyset\\
d_{b}\leq d_{a},\,\left|x_{a}-x_{b}\right|\leq3d_{a}
}
}d_{a}^{-\alpha\frac{s}{r}}\left|\cM_{a}^{\fs}u-\cM_{b}^{\fs}u\right|^{s}\,.
\]
 Assume $\chi_{\Ball{\fr}{G_{a}}}\chi_{A_{1,i}}\not\equiv0$. Then
it holds $p_{i}\in\Ball{2d_{a_{i}}}{x_{a_{i}}}$ which implies 
\begin{align}
\frac{1}{\left|\bQ\right|}\int_{\bP\cap\bQ_{\fr}}\sum_{\substack{i\neq0}
}\sum_{x_{a}\in\X(\bQ)}f_{a} & \frac{\left|\partial_{l}\phi_{i}\right|}{D_{l+}}\left|\cM_{a}^{\fs}u-\cM_{a_{i}}^{\fs}u\right|^{s}\nonumber \\
 & \leq\frac{1}{\left|\bQ\right|}\int_{\bP\cap\bQ_{\fr}}\sum_{x_{a}\in\X(\bQ)}\sum_{\substack{x_{b}\in\X_{\mat}\\
\Ball{2d_{b}}{x_{b}}\cap\Ball{\fr}{G_{a}}\neq\emptyset
}
}\sum_{i:\,x_{a_{i}}=x_{b}}f_{a}\chi_{A_{1,i}}\left|\cM_{a}^{\fs}u-\cM_{b}^{\fs}u\right|^{s}\nonumber \\
 & \leq\frac{1}{\left|\bQ\right|}C\int_{\bP\cap\bQ_{\fr}}\sum_{\substack{x_{b}\in\X_{\mat}}
}\sum_{\substack{x_{a}\in\X(\bQ)\\
\Ball{2d_{b}}{x_{b}}\cap\Ball{\fr}{G_{a}}\neq\emptyset
}
}\chi_{\Ball{\fr}{G_{a}}}\left|\cM_{a}^{\fs}u-\cM_{b}^{\fs}u\right|^{s}\,.\label{eq:thm:Delaunay-final-help-2}
\end{align}
Hence, we encounter the conditions $\Ball{\fr}{G_{a}}\cap\bQ\neq\emptyset$
and $\Ball{2d_{b}}{x_{b}}\cap\Ball{\fr}{G_{a}}\neq\emptyset$ as well
as 
\[
\left|x_{a}-x_{b}\right|\leq3\max\left\{ d_{a},d_{b}\right\} \,.
\]
In particular, we conclude the symmetric condition
\[
\Ball{4d_{a}}{x_{a}}\cap\bQ\neq\emptyset\,,\quad\Ball{4d_{a}}{x_{b}}\cap\bQ\neq\emptyset\,,\qquad\Ball{2d_{a}}{x_{a}}\cap\Ball{2d_{b}}{x_{b}}\neq\emptyset
\]
and 
\begin{equation}
\text{R.H.S of }\eqref{eq:thm:Delaunay-final-help-2}\leq\mathrm{I}_{0}\,.\label{eq:lem:Delaunay-final-help-3}
\end{equation}
Similarly 
\begin{align}
\frac{1}{\left|\bQ\right|}\int_{\bP\cap\bQ}\sum_{a:\,\partial_{l}\Phi_{a}>0}\sum_{b:\,\partial_{l}\Phi_{b}<0} & \frac{d_{b}^{-\alpha\frac{s}{r}}\left|\partial_{l}\Phi_{b}\right|}{D_{l+}^{\Phi}}\left|\cM_{a}u-\cM_{b}u\right|^{s}\leq I_{\alpha}\,.\label{eq:lem:Delaunay-final-help-4}
\end{align}
\emph{Step 3:} We now derive an estimate for $I_{\alpha}$. For pairs
$(a,b)$ with $d_{b}\leq d_{a},\,\left|x_{a}-x_{b}\right|\leq3d_{a}$
let $y_{a,b}:=\left(y_{1},\dots,y_{n(a,b)}\right)$ be a discrete
path on the Delaunay grid of $\X_{\mat}$ with length smaller than
$2\left|x_{a}-x_{b}\right|$ (this exists due to \cite{xia2013stretch})
that connects $x_{a}$ and $x_{b}$. By the minimal mutual distance
of points, this particularly implies that $n(a,b)\leq6d_{a}/2\fr$
and the path lies completely within $\Ball{4.5\,d_{a}}{x_{a}}$. Because
\begin{align*}
\left|\cM_{a}^{\fs}u-\cM_{b}^{\fs}u\right|^{s} & \leq n(a,b)^{s}\sum_{k=1}^{n(a,b)-1}\left|\cM_{y_{k}}^{\fs}u-\cM_{y_{k+1}}^{\fs}u\right|^{s}\\
 & \leq6d_{a}^{s}/2\fr\sum_{k=1}^{n(a,b)-1}\left|\cM_{y_{k}}^{\fs}u-\cM_{y_{k+1}}^{\fs}u\right|^{s}
\end{align*}
it holds with (\ref{eq:thm:Delaunay-final-help-1})
\begin{align*}
\left|\left(\cM_{a}^{\fs}u-\cM_{b}^{\fs}u\right)(x)\right|^{s} & \leq Cd_{a}^{s}\int_{\Ball{6d_{a}}{x_{a}}}\left(\sum_{e\sim f}F_{s}^{\fs,1}\of{\delta(e,f)}\left(\left|x-x_{e}\right|^{s}+\left|x-x_{f}\right|^{s}\right)\right.\cdot\\
 & \qquad\qquad\cdot\left.\left|x_{e}-x_{f}\right|^{2s}\left(\chi_{\Ball{\frac{\fr}{16}}{x_{e}}}+\chi_{\Ball{\frac{\fr}{16}}{x_{f}}}+d_{a}^{s-1}\chi_{P_{\delta,e,f}}\right)\right)\left|\nablas u\right|^{s}
\end{align*}
We make use of $\left|x-x_{e}\right|^{s}\leq2^{s}\left(\left|x-x_{a}\right|^{s}+\left|x_{a}-x_{e}\right|^{s}\right)\leq2^{s}\left(\left|x-x_{a}\right|^{s}+d_{a}^{s}\right)$
and $\left|x_{e}-x_{f}\right|^{2s}\leq Cd_{a}^{2s}$ and $B_{e,f}:=\Ball{\frac{\fr}{16}}{\{x_{e},x_{f}\}}\cup P_{\delta,e,f}$
to find 
\[
\left|\left(\cM_{a}^{\fs}u-\cM_{b}^{\fs}u\right)(x)\right|^{s}\leq Cd_{a}^{4s}\int_{\Ball{6d_{a}}{x_{a}}}\left(\sum_{e\sim f}F_{s}^{\fs,1}\of{\delta(e,f)}\left(\left|x-x_{a}\right|^{s}+d_{a}^{s}\right)\,\chi_{B_{e,f}}\right)\left|\nablas u\right|^{s}\,.
\]
In the integrals $I_{\alpha}$, any of the integrals $\int\chi_{\Ball{\fr}{G_{a}}}\left|\cM_{a}^{\fs}u-\cM_{b}^{\fs}u\right|^{s}$
has $\left|x-x_{a}\right|<2d_{a}$ and we can use an estimate of the
form 
\[
\left|\left(\cM_{a}^{\fs}u-\cM_{b}^{\fs}u\right)(x)\right|^{s}\leq Cd_{a}^{5s}\int_{\Ball{6d_{a}}{x_{a}}}\left(\sum_{e\sim f}F_{s}^{\fs,1}\of{\delta(e,f)}\chi_{B_{e,f}}\right)\left|\nablas u\right|^{s}\,.
\]
With this estimate, and using 
\[
\#\left\{ b:\,\Ball{4d_{b}}{x_{b}}\cap\bQ\neq\emptyset,\,d_{b}\leq d_{a},\,\left|x_{a}-x_{b}\right|\leq3d_{a}\right\} \leq Cd_{a}^{d}
\]
the integral $I_{\alpha}$ can be controlled through 
\[
I_{\alpha}\leq\frac{1}{\left|\bQ\right|}\int_{\bP}\sum_{a:\,\Ball{4d_{a}}{x_{a}}\cap\bQ\neq\emptyset}d_{a}^{2d+5s-\alpha\frac{s}{r}}\chi_{\Ball{6d_{a}}{x_{a}}}\left(\sum_{e\sim f}F_{s}^{\fs,1}\of{\delta(e,f)}\chi_{B_{e,f}}\right)\left|\nablas u\right|^{s}\,.
\]
Denoting 
\begin{align*}
f(\omega) & :=\sum_{a}d_{a}^{2d+5s-\alpha\frac{s}{r}}\chi_{\Ball{6d_{a}}{x_{a}}}\,,\\
f(\omega,\bQ) & :=\sum_{a:\,\Ball{4d_{a}}{x_{a}}\cap\bQ\neq\emptyset}d_{a}^{2d+5s-\alpha\frac{s}{r}}\chi_{\Ball{6d_{a}}{x_{a}}}\,,\\
g(\omega) & :=\sum_{e\sim f}F_{s}^{\fs,1}\of{\delta(e,f)}\chi_{B_{e,f}}\,,
\end{align*}
and using $u\equiv0$ outside $\bQ$, we observe 
\[
I_{\alpha}\leq\left(\frac{1}{\left|\bQ\right|}\int_{\bP}f(\omega,\bQ)^{\frac{p}{p-s}}\,g(\omega)^{\frac{p}{p-s}}\right)^{\frac{p-s}{p}}\left(\frac{1}{\left|\bQ\right|}\int_{\bP\cap\bQ}\left|\nablas u\right|^{p}\right)^{\frac{s}{p}}\,.
\]
 \emph{Step 4: }Since every quantity related to the distribution
of $d_{a}$ is distributed exponentially, we can be very generous
with this variable. We observe 
\[
\frac{1}{\left|\bQ\right|}\int_{\bP\cap\bQ_{\fr}}\sum_{a}f_{a}\sum_{\substack{i\neq0}
}\frac{\left|\partial_{l}\phi_{i}\right|}{D_{l+}}\left|\tau_{i}^{\fs}u-\cM_{a_{i}}^{\fs}u\right|^{s}\leq\frac{1}{\left|\bQ\right|}\int_{\bP\cap\bQ_{\fr}}\sum_{\substack{i\neq0}
}\chi_{A_{1,i}}\left|\tau_{i}^{\fs}u-\cM_{a_{i}}^{\fs}u\right|^{s}
\]
but for every fixed $x$ (and using that $x\in\Ball{2d_{a_{i}}}{x_{a_{i}}}$)
using again Jensens inequality 
\[
\int_{A_{1,i}}\left|\tau_{i}^{\fs}u(x)-\cM_{a_{i}}^{\fs}u(x)\right|^{s}\leq C\int_{\Ball{\fr_{i}}{y_{i}}}\left(\left|\nabla\left(u-\cM_{a_{i}}^{\fs}u\right)\right|^{s}d_{a_{i}}^{s}+\left|u-\cM_{a_{i}}^{\fs}u\right|^{s}\right)\,.
\]
Having this in mind, we may sum over all $y_{i}$ to find 
\begin{align*}
\frac{1}{\left|\bQ\right|} & \int_{\bP\cap\bQ_{\fr}}\sum_{a}f_{a}\sum_{\substack{i\neq0}
}\frac{\left|\partial_{l}\phi_{i}\right|}{D_{l+}}\left|\tau_{i}^{\fs}u-\cM_{a_{i}}^{\fs}u\right|^{s}\\
 & \leq\frac{1}{\left|\bQ\right|}\int_{\bQ\cap\bP}C\sum_{a}\chi_{2d_{a}}\sum_{b\sim a}d_{a}^{s}\,\chi_{P(x_{a},x_{b},\delta(a,b))}\left(\left|\nabla\left(u-\cM_{a}^{\fs}u\right)\right|^{s}+\left|u-\cM_{a}^{\fs}u\right|^{s}\right)\,.
\end{align*}
With the splitting $u-\cM_{a}^{\fs}u=u-\cM_{a}^{\fs,\delta(a,b)}u+\cM_{a}^{\fs,\delta(a,b)}u-\cM_{a}^{\fs}u$
and Lemmas \ref{lem:Mixed-Korn-Zylinder} and \ref{lem:general-Korn-Poincar=0000E9}
it follows with $F_{s}^{\fs,1}$ from (\ref{eq:def-F-sq-s})
\begin{align*}
\frac{1}{\left|\bQ\right|} & \int_{\bP\cap\bQ_{\fr}}\sum_{a}f_{a}\sum_{\substack{i\neq0}
}\frac{\left|\partial_{l}\phi_{i}\right|}{D_{l+}}\left|\tau_{i}^{\fs}u-\cM_{a_{i}}^{\fs}u\right|^{s}\\
 & \leq\frac{1}{\left|\bQ\right|}C\sum_{a}\sum_{b\sim a}F_{s}^{\fs,1}\of{\delta(a,b)}\left(\left\Vert \nabla u\right\Vert _{L^{s}(\Ball{\frac{\fr}{16}}{x_{a}}\cup\Ball{\frac{\fr}{16}}{x_{b}})}^{s}+\left(2d_{a}\right)^{s}\left\Vert \nabla u\right\Vert _{L^{s}(P(x_{a},x_{b},\delta(a,b)))}^{s}\right)d_{a}^{d+s}
\end{align*}
by a restructuration, the right hand side is bounded by 
\begin{align*}
\frac{1}{\left|\bQ\right|}\int_{\bQ\cap\bP}C\left(\sum_{a}\chi_{2d_{a}}d_{a}^{3s+d}\right) & g(\omega)\left|\nabla^{s}u\right|^{s}\\
 & \leq C\left(\frac{1}{\left|\bQ\right|}\int_{\bQ\cap\bP}f_{1}(\omega)^{\frac{p}{p-s}}Bg(\omega)^{\frac{p}{p-s}}\right)^{\frac{p-s}{p}}\left(\frac{1}{\left|\bQ\right|}\int_{\bQ\cap\bP}\left|\nabla^{s}u\right|^{p}\right)^{\frac{s}{p}}\,,
\end{align*}
where
\begin{align*}
f_{1}(\omega) & :=\sum_{a}\chi_{2d_{a}}\left(2d_{a}\right)^{s+d}
\end{align*}
Step 4: We can replace in the above calculations $\bQ$ by $n\bQ$.
By Lemma \ref{lem:matern-delaunay-bounded} we can extend $f(\omega,n\bQ)$
to $f(\omega)|_{Rn\bQ}$ for some fixed $R>1$ and on $R\bQ$ we can
use standard ergodic theory. Hence, the expressions in $\delta$ and
$d_{a}$ converge to a constant as $n\to\infty$ provided 
\begin{equation}
\E\of{\left(f_{1}g\right)^{\frac{p}{p-s}}+\left(fg\right)^{\frac{p}{p-s}}}<\infty\,.\label{eq:thm:Delaunay-final-help-10}
\end{equation}
However, $f$, $f_{1}$ and $g$ are stationary by definition and
$f$ and $g$ or $f_{1}$ and $g$ are independent. Since $f$ and
$f_{1}$ clearly have finite expectation by the exponential distribution
of $d_{a}$ and Lemma \ref{lem:estim-E-fa-fb}, we only mention that
due to the strong mixing of $\delta$ and its independence from the
distribution of connections 
\[
\E(g^{\frac{p}{p-s}})\leq\E\of{\sum_{e\sim f}\chi_{B_{e,f}}}\,\E\of{\of{\delta^{-s-d}+\delta^{1+s-2d}}^{\frac{p}{p-s}}}
\]
and thus (\ref{eq:thm:Delaunay-final-help-10}) holds.
\end{proof}
The work \cite{xia2013stretch} which we used in the last proof also
opens the door to demonstrate the following result which will be used
in part III of this series to prove regularity properties of the homogenized
equation.
\begin{thm}
\label{thm:overlay-count-1}For fixed $y_{0}\in\X_{\mat}$ and every
$\tilde{y}\in\X_{\mat}$ let $P(y_{0},\tilde{y})=\left(y_{0},y_{1}(\tilde{y}),\dots,y_{N}(\tilde{y})\right)_{N\in\N}$
with $y_{N}(\tilde{y})=\tilde{y}$ be the shortest path of points
in $\X_{\mat}$ connecting $y_{0}$ and $\tilde{y}$ in $\bP$ and
having length $L(y_{0},\tilde{y})$. Then there exists 
\begin{align*}
\gamma_{y_{0,}\tilde{y}}:\,[0,L(y_{0},\tilde{y})]\times\Ball{\frac{\fr}{16}}0 & \to\bP\\
(t,z) & \mapsto\gamma_{y_{0,}\tilde{y}}(t,z)
\end{align*}
such that $\gamma_{y_{0,}\tilde{y}}(t,\cdot)$ is invertible for every
$t$ and $\norm{\partial_{t}\gamma_{y_{0,}\tilde{y}}}_{\infty}\leq2$.
For $R>1$ let 
\[
N_{y_{0},R}(x):=\#\left\{ \tilde{y}\in\Ball R{y_{0}}\cap\X_{\mat}:\,\exists t:\,x\in\gamma_{y_{0},\tilde{y}}\of{t,\Ball{\frac{\fr}{16}}0}\right\} \,.
\]
Then there exists $C>0$ such that for every $y_{0}$ it holds
\[
N_{y_{0},R}(x)\leq C\left(R^{d}-\left(\frac{x}{2}\right)^{d}\right)\quad\text{for }\left|x-y_{0}\right|<2R\,,\qquad N_{y_{0},R}(x)=0\quad\text{else.}
\]
\end{thm}

\begin{proof}
The function $\gamma_{y_{0,}\tilde{y}}$ consists basically of pipes
connecting $y_{i}(\tilde{y})$ with $y_{i+1}(\tilde{y})$ that conically
become smaller within the ball $\Ball{\frac{\fr}{2}}{y_{i}(\tilde{y})}$
before entering the pipe and vice versa in $\Ball{\frac{\fr}{2}}{y_{i+1}(\tilde{y})}$.
Defining 
\[
N_{y_{0},r,R}(x):=\#\left\{ \tilde{y}\in\of{\Ball R{y_{0}}\backslash\Ball r{y_{0}}}\cap\X_{\mat}:\,\exists t:\,x\in\gamma_{y_{0},\tilde{y}}\of{t,\Ball{\frac{\fr}{16}}0}\right\} 
\]
\cite{xia2013stretch} implies $N_{y_{0},r,R}(x)=0$ for all $\left|x-y_{0}\right|>2R$
but also due to the minimal mutual distance $N_{y_{0},r,R}(x)\leq CR^{d-1}(R-r)$,
where $C$ depends only on $\fr$ and $d$. Hence writing $\left\lfloor x\right\rfloor :=\min\left\{ n\in\N:\,n+1>x\right\} $
we can estimate for every $K\in\N$ 
\[
N_{y_{0},K}(x)\leq\sum_{k=0}^{K-1}N_{y_{0},k,k+1}(x)\leq C\sum_{k=\left\lfloor \frac{x}{2}\right\rfloor }^{K-1}\left(k+1\right)^{d-1}\leq C\left(K^{d}-\left\lfloor \frac{x}{2}\right\rfloor ^{d}\right)\,.
\]
\end{proof}
We close this section by proving Theorem \ref{thm:Pipes-Model}.
\begin{proof}[\textbf{Proof of Theorem \ref{thm:Pipes-Model}}]
The statement on the support is provided by Theorem \ref{thm:support}
and the fact that we restrict to functions with support in $m\bQ$.
Hence in the following we can apply all cited results to $\Ball{m^{1-\beta}}{m\bQ}$
instead of $m\bQ$. According to Lemmas \ref{lem:local-delta-M-extension-estimate}
and \ref{lem:6-4}--\ref{lem:6-5} and to Theorem \ref{thm:Delaunay-final}
we need only need to ensure
\[
\E\of{\delta^{-s-d}+\delta^{1+s-2d}}^{\frac{p}{p-s}}+\E\left(1+M_{[\frac{1}{8}\delta],\Rd}\right)^{r}+\E\left|\tilde{\rho}_{\Rd}\right|^{-\frac{sr}{s-r}}<\infty\,,
\]
since $d_{a}$ is distributed exponentially and the corresponding
terms are bounded as long as $r\neq s\neq p$. We note that the exponential
distribution of $M$ allows us to restrict to the study of $\delta$
and $\tilde{\rho}$. 

According to Lemma \ref{lem:matern-delaunay-distribution} it is sufficient
that $\max\left\{ \frac{p\left(s+d\right)}{p-s},\frac{p(2d-s-1)}{p-s}\right\} \leq\beta$
and $\frac{sr}{s-r}\leq\beta+d-1$. 
\end{proof}

\subsection{\label{subsec:Boolean-Model-for}Boolean Model for the Poisson Ball
Process}

The following argumentation will be strongly based on the so called
void probability. This is the probability $\P_{0}(A)$ to not find
any point of the point process in a given open set $A$ and is given
by (\ref{eq:PoisonPointPoc-Prob}) i.e. $\P_{0}(A):=e^{-\lambda\left|A\right|}$.
The void probability for the ball process is given accordingly by
\[
\P_{0}(A):=e^{-\lambda\left|\overline{\Ball 1A}\right|}\,,\qquad\overline{\Ball 1A}:=\left\{ x\in\Rd\,:\;\dist\of{x,A}\leq1\right\} \,,
\]
which is the probability that no ball intersects with $A\subset\Rd$.
\begin{thm}
\label{thm:boolean-delta-M-distrib}Let $\bP\left(\omega\right):=\bigcup_{i}B_{i}(\omega)$
(or $\bP\left(\omega\right):=\Rd\backslash\bigcup_{i}B_{i}(\omega)$)
and define 
\begin{align*}
\tilde{\delta}\of x & :=\min\left\{ \delta\of{\tilde{x}}\,:\;\tilde{x}\in\partial\bP\,\text{s.t. }x\in\Ball{\frac{1}{8}\delta\of{\tilde{x}}}{\tilde{x}}\right\} \,,
\end{align*}
where $\min\emptyset:=0$ for convenience. Then $\partial\bP$ is
almost surely locally $\left(\delta,M\right)$ regular and for every
$\gamma<1$, $\beta<d+2$ it holds
\[
\E\of{\delta^{-\gamma}}+\E\of{\tilde{\delta}^{-\gamma-1}}+\E\of{\tilde{M}_{[0]}^{\beta}}<\infty\,.
\]
Furthermore, it holds $\hat{d}\le d-1$ and $\alpha=0$ in inequalities
(\ref{eq:lem:properties-local-rho-convering-4}) and (\ref{lem:properties-local-rho-convering}).
Furthermore the extension order and symmetric extension order are
both $n=0$. If $\bP\left(\omega\right):=\Rd\backslash\overline{\bigcup_{i}B_{i}}(\omega)$
the above holds with $\alpha$ replaced by $1$ and with extension
order $n=1$ and symmetric extension order $n=2$.
\end{thm}

\begin{rem}
We observe that the union of balls has better properties than the
complement.
\end{rem}

\begin{proof}
We study only $\bP\left(\omega\right):=\bigcup_{i}B_{i}(\omega)$
since $\Rd\backslash\overline{\bigcup_{i}B_{i}}(\omega)$ is the complement
sharing the same boundary. Hence, in case $\bP(\omega)=\Rd\backslash\overline{\bigcup_{i}B_{i}}(\omega)$,
all calculations remain basically the same. However, in the first
case, it is evident that $\alpha=0$ and $n=0$ because the geometry
has only cusps and no dendrites and we refer to Remarks \ref{rem:ext-order}
and \ref{rem:ext-order-sym}. 

In what follows, we use that the distribution of balls is mutually
independent. That means, given a ball around $x_{i}\in\X_{\pois}$,
the set $\X_{\pois}\backslash\left\{ x_{i}\right\} $ is also a Poisson
process. W.l.o.g. , we assume $x_{i}=x_{0}=0$ with $B_{0}:=\Ball 10$.
First we note that $p\in\partial B_{0}\cap\partial\bP$ if and only
if $p\in\partial B_{0}\backslash\bP$, which holds with probability
$\P_{0}\of{\Ball 1p}=\P_{0}\of{B_{0}}$. This is a fixed quantity,
independent from $p$.

Now assuming $p\in\partial B_{0}\backslash\bP$, the distance to the
closest ball besides $B_{0}$ is denoted 
\[
r(p)=\dist(p,\partial\bP\backslash\partial B_{0})
\]
with a probability distribution 
\[
\P_{\dist}(r):=\P_{0}\of{\Ball{1+r}p}/\P_{0}\of{\Ball 1p}\,.
\]
It is important to observe that $\partial B_{0}$ is $r$-regular
in the sense of Lemma \ref{lem:eta-lipschitz}. Another important
feature in view of Lemma \ref{lem:properties-delta-M-regular} is
$r(p)<\Delta(p)$. In particular, $\delta(p)>\frac{1}{2}r(p)$ and
$\partial B_{0}$ is $(\delta,1)$-regular in case $\delta<\sqrt{\frac{1}{2}}$.
Hence, in what follows, we will derive estimates on $r^{-\gamma}$,
which immediately imply estimates on $\delta^{-\gamma}$.

\textbf{Estimate on $\gamma$:} A lower estimate for the distribution
of $r(p)$ is given by 
\begin{equation}
\P_{\dist}(r):=\P_{0}\of{\Ball{1+r}p}/\P_{0}\of{\Ball 1p}\approx1-\lambda\left|\S^{d-1}\right|r\,.\label{eq:PP-lower-estim-delta}
\end{equation}
This implies that almost surely for $\gamma<1$
\[
\limsup_{n\to\infty}\frac{1}{(2n)^{d}}\int_{(-n,n)^{d}\cap\partial\bP}r(p)^{-\gamma}\,\d\cH^{d-1}(p)<\infty\,,
\]
i.e. $\E\of{\delta^{-\gamma}}<\infty$.

\textbf{Intersecting balls:} Now assume there exists $x_{i}$, $i\neq0$
such that $p\in\partial B_{i}\cap\partial B_{0}$. W.l.o.g. assume
$x_{i}=x_{1}:=(2x,0,\dots,0)$ and $p=\left(\sqrt{1-x^{2}},0,\dots,0\right)$.
Then 
\[
\delta(p)\leq\delta_{0}(p):=2\sqrt{1-x^{2}}
\]
 and $p$ is at least $M(p)=\frac{x}{\sqrt{1-x^{2}}}$-regular. Again,
a lower estimate for the probability of $r$ is given by (\ref{eq:PP-lower-estim-delta})
on the interval $(0,\delta_{0})$. Above this value, the probability
is approximately given by $\lambda\left|\S^{d-1}\right|\delta_{0}$
(for small $\delta_{0}$i.e.~$x\approx1$). We introduce as a new
variable $\xi=1-x$ and obtain from $1-x^{2}=\xi(1+x)$ that 
\begin{equation}
\delta_{0}\leq C\xi^{\frac{1}{2}}\quad\text{and}\quad M(p)\leq C\xi^{-\frac{1}{2}}\,.\label{eq:delta-0-xi}
\end{equation}

\textbf{No touching:} At this point, we observe that $M$ is almost
surely locally finite. Otherwise, we would have $x=1$ and for every
$\eps>0$ we had $x_{1}\in\Ball{2+\eps}{x_{0}}\backslash\Ball{2-\eps}{x_{0}}$.
But 
\[
\P_{0}\of{\Ball{2+\eps}{x_{0}}\backslash\Ball{2-\eps}{x_{0}}}\approx1-\lambda2\left|\S^{d-1}\right|\eps\;\to\;1\qquad\text{as }\eps\to0\,.
\]
Therefore, the probability that two balls ``touch'' (i.e. that $x=1$)
is zero. The almost sure local boundedness of $M$ now follows from
the countable number of balls.

\textbf{Extension to $\tilde{\delta}$:} We again study each ball
separately. Let $p\in\partial B_{0}\backslash\overline{\bP}$ with
tangent space $T_{p}$ and normal space $N_{p}$. Let $x\in N_{p}$
and $\tilde{p}\in\partial B_{0}$ such that $x\in\Ball{\frac{1}{8}\delta(\tilde{p})}{\tilde{p}}$,
then also $p\in\Ball{\frac{1}{8}\delta(\tilde{p})}{\tilde{p}}$ and
$\delta(p)\in(\frac{7}{8},\frac{7}{6})\delta(\tilde{p})$ and $\delta(\tilde{p})\in(\frac{7}{8},\frac{7}{6})\delta(p)$
by Lemma \ref{lem:eta-lipschitz}. Defining 
\[
\tilde{\delta}_{i}\of x:=\min\left\{ \delta\of{\tilde{x}}\,:\;\tilde{x}\in\partial B_{i}\backslash\bP\,\text{s.t. }x\in\Ball{\frac{1}{8}\delta\of{\tilde{x}}}{\tilde{x}}\right\} \,,
\]
we find 
\[
\tilde{\delta}^{-\gamma}\leq\sum_{i}\chi_{\tilde{\delta}_{i}>0}\tilde{\delta}_{i}^{-\gamma}\,.
\]
Studying $\delta_{0}$ on $\partial B_{0}$ we can assume $M\leq M_{0}$
in (\ref{eq:lem:delta-tilde-construction-estimate-1}) and we find
\[
\int_{\bP}\chi_{\tilde{\delta}_{0}>0}\tilde{\delta}_{0}^{-\gamma-1}\leq C\int_{\partial B_{0}\backslash\bP}\delta^{-\gamma}\,.
\]
Hence we find
\[
\int_{\bP}\tilde{\delta}^{-\gamma-1}\leq\sum_{i}\int_{\bP}\chi_{\tilde{\delta}_{i}>0}\tilde{\delta}_{i}^{-\gamma-1}\leq\sum_{i}C\int_{\partial B_{i}\backslash\bP}\delta^{-\gamma}\,.
\]

\textbf{Estimate on $\beta$:} For two points $x_{i},x_{j}\in\X_{\pois}$
let $\mathrm{Circ}_{ij}:=\partial B_{i}\cap\partial B_{j}$ and $\Ball{\frac{1}{8}\tilde{\delta}}{\mathrm{Circ}_{ij}}:=\bigcup_{p\in\mathrm{Circ}_{ij}}\Ball{\frac{1}{8}\tilde{\delta}(p)}p$.
For the fixed ball $B_{i}=B_{0}$ we write $\mathrm{Circ}_{0j}$ and
obtain $\left|\mathrm{Circ}_{0j}\right|\leq C\delta_{0}^{d}$ with
$\delta_{0}$ from (\ref{eq:delta-0-xi}). Therefore, we find 
\[
\int_{\mathrm{Circ}_{0j}}(1+M(p))^{\beta}\leq\delta_{0}^{d}(1+M(p))^{\beta}\leq C\xi^{-\frac{1}{2}(\beta-d)}\,.
\]

We now derive an estimate for $\E\of{\int_{\Ball{1+\fr}0}\tilde{M}^{\beta}}.$
To this aim, let $q\in(0,1)$. Then $x\in\Ball{2-q^{k+1}}0\backslash\Ball{2-q^{k}}0$
implies $\xi\geq q^{k+1}$ and 
\begin{align*}
\int_{\Ball{1+\fr}0}\tilde{M}^{\beta} & \leq C+\sum_{k=1}^{\infty}\sum_{x_{j}\in\Ball{2-q^{k+1}}0\backslash\Ball{2-q^{k}}0}\int_{\mathrm{Circ}_{0j}}(1+M(p))^{\beta}\\
 & \leq C+\sum_{k=1}^{\infty}\sum_{x_{j}\in\Ball{2-q^{k+1}}0\backslash\Ball{2-q^{k}}0}C\left(q^{k+1}\right)^{-\frac{1}{2}(\beta-d)}
\end{align*}
The only random quantity in the latter expression is $\#\left\{ x_{j}\in\Ball{2-q^{k+1}}0\backslash\Ball{2-q^{k}}0\right\} $.
Therefore, we obtain with $\E\of{\X(A)}=\lambda\left|A\right|$ that
\begin{align*}
\E\of{\int_{\Ball{1+\fr}0}\tilde{M}^{\beta}} & \leq C\left(1+\sum_{k=1}^{\infty}\left(q^{k}-q^{k+1}\right)\left(q^{k+1}\right)^{-\frac{1}{2}(\beta-d)}\right)\\
 & \leq C\left(1+\sum_{k=1}^{\infty}\left(q^{k}\right)^{-\frac{1}{2}(\beta-d-2)}\right)\,.
\end{align*}
Since the point process has finite intensity, this property carries
over to the whole ball process and we obtain the condition $\beta<d+2$
in order for the right hand side to remain bounded.

\textbf{Estimate on $\hat{d}$:} We have to estimate the local maximum
number of $A_{3,k}$ overlapping in a single point in terms of $\tilde{M}$.
We first recall that $\hat{\rho}(p)\approx8\tilde{M}(p)\tilde{\rho}(p)$.
Thus large discrepancy between $\hat{\rho}$ and $\tilde{\rho}$ occurs
in points where $\tilde{M}$ is large. This is at the intersection
of at least two balls. Despite these ``cusps'', the set $\partial\bP$
consists locally on the order of $\hat{\rho}$ of almost flat parts.
Arguing like in Lemma \ref{lem:properties-local-rho-convering} resp.
Remark \ref{rem:lem:properties-local-rho-convering} this yields $\hat{d}\leq d-1$.
\end{proof}
It remains to verify bounded average connectivity of the Boolean set
$\bP_{\infty}$ or its complement. Associated with the connected component
$\X_{\pois,\infty}$ there is a graph distance 
\[
\forall x,y\in\X_{\pois,\infty}\quad d(x,y):=\inf\left\{ l(\gamma):\,\gamma\text{ path in }\X_{\pois,\infty}\text{ from }x\text{ to }y\right\} \,.
\]
Using this distance, we shall rely on the following concept. 
\begin{defn}[Statistical Strech Factor]
For $x\in\X_{\pois,\infty}$ and $R>\fr$ we denote 
\[
S(x,R):=\max_{\substack{y\in\X_{\pois,\infty}\cap\Ball Rx}
}\frac{d(x,y)}{R}\,,\qquad S(x):=\sup_{R>\fr}S(x,R)\,,
\]
the statistical local strech factor $S(x,R)$ and statistical (global)
strech factor $S(x)$.
\end{defn}

\begin{lem}
\label{lem:strech-factor-boolean}There exists $S_{0}>1$ depending
only on $d$ and $\lambda$ such that for $x\in\X_{\pois,\infty}$
it holds 
\[
\forall S>S_{0}:\qquad\P\of{S(x)>S}\leq\frac{2\mu}{\nu}e^{-\frac{\nu}{2\mu}S}\,.
\]
\end{lem}

In order to prove this, we will need the following large deviation
result.
\begin{thm}[{Shape Theorem \cite[Thm 2.2]{yao2011large}}]
\label{thm:large-dev-pois}Let $\lambda>\lambda_{c}$. Then there
exist positive contants $\mu$, $\nu$ and $k_{0}$ such that the
following holds: For every $k>k_{0}$
\[
\P\of{S(0,k)>\mu}\leq e^{-\nu k}\,.
\]
\end{thm}

\begin{proof}[Proof of Lemma \ref{lem:strech-factor-boolean}]
We have 
\begin{align*}
S(0,k) & >\alpha\mu & \qquad & \Leftrightarrow\qquad & \exists x,y\in\Ball k0:\quad d(x,y) & \geq\alpha\mu k\,,\\
S(0,\alpha k) & >\mu & \qquad & \Leftrightarrow\qquad & \exists x,y\in\Ball{\alpha k}0:\quad d(x,y) & \geq\alpha\mu k\,,
\end{align*}
i.e. 
\[
\P\of{S(0,k)>\alpha\mu}\leq\P\of{S(0,\alpha k)>\mu}\leq e^{-\frac{\nu}{\mu}(\alpha\mu)k}\,.
\]
One quickly verifies for $k\in\N$ that $S(0,k)\leq S$ and $S(0,k+1)\leq S$
implies $S(0,k+r)\leq2S$ for all $r\in(0,1)$. Hence we find 
\[
\P\of{S(x)>S}\leq\sum_{k\in\N}\P\of{S(0,k)>\frac{S}{2}}\leq\sum_{k\in\N}e^{-\frac{\nu}{2\mu}Sk}\leq\frac{2\mu}{\nu}e^{-\frac{\nu}{2\mu}S}\,.
\]
\end{proof}
While the choice of the points $\left(p_{i}\right)_{i\in\N}\subset\partial\bP$
is clearly specified in Section \ref{subsec:5-Preliminaries}, there
is lots of room in the choice and construction of $\X_{\fr}$. In
what follows, we choose $\X_{\fr}$ in the form (\ref{eq:def-X_r}).
Then we find the following:
\begin{thm}
\label{thm:boolean-final}Under the above assumptions on the construction
of $\bP_{\infty}$, as well as $p>d$ and using the notation of Lemma
\ref{lem:6-4}, for every $1\leq r<s<p$ there almost surely exists
$C>0$ such that for every $n\in\N$ and every $u\in W_{0,\partial(n\bQ)}^{1,p}(\bP_{\infty}\cap n\bQ)$
\begin{multline*}
\left(\frac{1}{\left|n\bQ\right|}\int_{\bP_{\infty}\cap n\bQ}\sum_{k:\,\partial_{l}\Phi_{k}>0}\sum_{j:\,\partial_{l}\Phi_{j}<0}\frac{d_{j}^{-\tilde{\alpha}\frac{s}{r}}\left|\partial_{l}\Phi_{j}\right|}{D_{l+}^{\Phi}}\left|\cM_{k}u-\cM_{j}u\right|^{s}\right)^{\frac{r}{s}}+F_{s}^{3}(n\bQ,u)\\
\leq C\left(\frac{1}{\left|n\bQ\right|}\int_{\bP_{\infty}\cap n\bQ}\left|\nabla u\right|^{p}\right)^{\frac{r}{p}}\,,
\end{multline*}
and for every $u\in\bW_{0,\partial(n\bQ)}^{1,p}(\tilde{\bP}\cap n\bQ)$
\begin{multline*}
\left(\frac{1}{\left|n\bQ\right|}\int_{\bP_{\infty}\cap n\bQ}\sum_{k:\,\partial_{l}\Phi_{k}>0}\sum_{j:\,\partial_{l}\Phi_{j}<0}\frac{d_{j}^{-\tilde{\alpha}\frac{s}{r}}\left|\partial_{l}\Phi_{j}\right|}{D_{l+}^{\Phi}}\left|\cM_{k}^{\fs}u-\cM_{j}^{\fs}u\right|^{s}\right)^{\frac{r}{s}}+F_{s}^{3,\fs}(n\bQ,u)\\
\leq C\left(\frac{1}{\left|n\bQ\right|}\int_{\bP_{\infty}\cap n\bQ}\left|\nablas u\right|^{p}\right)^{\frac{r}{p}}\,.
\end{multline*}
\end{thm}

\begin{lem}
\label{lem:X-pois-covery-finite}Let $\X_{\pois}$ be a Poisson point
process with finite intensity. Generate a Voronoi tessellation from
$\X_{\pois}$ and for each $x_{a}\in\X_{\pois}$ let $d_{a}$ be the
diameter of the corresponding Voronoi cell. Then for each $n\in\N$
the following function has finite expectation
\[
f_{n}:=\sum_{a}\chi_{\Ball{nd_{a}}{x_{a}}}\,.
\]
 
\end{lem}

Note that this statement is not covered by Lemma \ref{lem:estim-E-fa-fb}
due to the lack of a minimal distance between the points.
\begin{proof}
Given the condition $0\in\X_{\pois}$ we observe 
\[
\E\of{\chi_{\Ball{nd_{0}}0}}(x)\leq\sum_{k=0}^{\infty}\P\of{d_{0}\in[k,k+1)}\chi_{\Ball{k+1}0}(x)\,.
\]
Since $\P\of{d_{0}\in[k,k+1)}\leq e^{-\alpha k}$ for some $\alpha>0$,
we infer 
\[
\E\of{\chi_{\Ball{nd_{0}}0}}\,(x)\leq Ce^{-\alpha\left|x\right|}\,.
\]
From here, we conclude with the exponentially in $N$ decreasing probability
to find more than $N$ points within $[0,1]^{d}$:
\[
\E\of{\sum_{x_{a}\in\X_{\pois}\cap[0,1]^{d}}\chi_{\Ball{nd_{a}}{x_{a}}}}(x)\leq Ce^{-\beta\left|x\right|}\,,
\]
for some $\beta>0$. Summing up over all cubes we infer 
\[
\E(f_{n})(0)\leq C\sum_{k\in\Zd}e^{-\beta\left|x-k\right|}\leq C\sum_{N\in\N}N^{d-1}e^{-\beta N}<\infty\,.
\]
\end{proof}
Similar to the proof of Theorem \ref{thm:Delaunay-final} it will
be necessary to introduce the following quantity for $y\in\X_{\pois,\infty}$
based on (\ref{eq:lem:Mixed-Korn-zylinder-cM}):
\[
\cM_{y}^{\fs}u(x):=\overline{\nabla_{y,\fr}^{\bot}}u\of{x-y}+\fint_{\Ball{\fr}y}u\,.
\]

An important property of $\cM_{y}^{\fs}$ is the following.
\begin{lem}
\label{lem:thm:poisson-final-help}Let $y_{1},y_{2}\in\X_{\pois,\infty}$
with $\left|y_{1}-y_{2}\right|<2$ and $\delta:=\frac{1}{2}\sup\left\{ r:\,\Ball r{\frac{1}{2}\left(y_{1}+y_{2}\right)}\subset\tilde{\bP}\right\} $.
Then there exists $f:\,\Ball 1{\left\{ y_{1},y_{2}\right\} }\to\R$
such that
\begin{align*}
\left|\cM_{y_{1}}^{\fs}u(x)-\cM_{y_{2}}^{\fs}u(x)\right|^{s} & \leq C\norm{f\nablas u}_{L_{\Ball 1{\left\{ y_{1},y_{2}\right\} }}^{s}}^{s}\,,\\
\left|\cM_{y_{1}}u(x)-\cM_{y_{2}}u(x)\right|^{s} & \leq C\norm{f\nablas u}_{L_{\Ball 1{\left\{ y_{1},y_{2}\right\} }}^{s}}^{s}\,,
\end{align*}
and 
\begin{equation}
\int_{\Ball 1{\left\{ y_{1},y_{2}\right\} }}\left|f\right|^{\frac{sp}{p-s}}\leq C\delta^{\frac{s(p-d)}{p-s}-1}\,.\label{eq:lem:thm:poisson-final-help-1}
\end{equation}
Furthermore for some fixed $C>0$ and for every $y\in\X_{\pois,\infty}$
\begin{align}
\int_{\Ball 1y}\sum_{i}\chi_{\Ball{\tilde{\rho}_{i}}{p_{i}}}\left|\tau_{i}^{\fs}u-\cM_{y}^{\fs}u\right|^{s}+\sum_{x_{a}\in\X_{\fr}}\chi_{\Ball{\frac{\fr}{16}}{x_{a}}}\left|\cM_{a}^{\fs}u-\cM_{y}^{\fs}\right|^{s} & \leq C\norm{\nablas u}_{L^{s}(\Ball 1y)}^{s}\,.\label{eq:lem:thm:poisson-final-help-2}\\
\int_{\Ball 1y}\sum_{i}\chi_{\Ball{\tilde{\rho}_{i}}{p_{i}}}\left|\tau_{i}u-\cM_{y}u\right|^{s}+\sum_{x_{a}\in\X_{\fr}}\chi_{\Ball{\frac{\fr}{16}}{x_{a}}}\left|\cM_{a}u-\cM_{y}\right|^{s} & \leq C\norm{\nablas u}_{L^{s}(\Ball 1y)}^{s}\,.
\end{align}
\end{lem}

\begin{proof}
We only treat the vector valued case, the other is proved similarly
using results from Section \ref{subsec:Poincar=0000E9-Inequalities}.
W.l.o.g let $y_{1}=y\be_{1}$ and $y_{2}=-y\be_{1}$. Let $n=\min\left\{ n\in\N:\,\Ball{2^{-n}\fr}0\subset\tilde{\bP}\right\} $,
i.e. $\delta\in(2^{-n-1}\fr,2^{-n}\fr)$. Furthermore, let $\alpha_{k}:=2\fr\sum_{j=1}^{k}2^{-(n-j)}$
for $k=1,\dots,n$ and $\alpha_{-k}=-\alpha_{k}$ with $\alpha_{0}=0$.
Using (\ref{eq:lem:Mixed-Korn-zylinder-cM}), for every number $j=-n,\dots,n$
let further 
\[
\cM_{j}^{\fs}u:=\cM_{\alpha_{j}\be_{1}}^{\fs,\fr2^{-\left(n-\left|j\right|\right)}}\,.
\]
Then for $j\geq0$ we find from Lemma \ref{lem:Mixed-Korn-Zylinder}
\begin{align*}
\left|\cM_{j}^{\fs}u(x)-\cM_{j+1}^{\fs}u(x)\right|^{s} & \leq C\left(\left|\cM_{\alpha_{j}\be_{1}}^{\fs,\fr2^{-\left(n-j\right)}}u(x)-\cM_{\alpha_{j+1}\be_{1}}^{\fs,\fr2^{-\left(n-j\right)}}u(x)\right|^{s}\right.+\\
 & \qquad\left.+\left|\cM_{\alpha_{j+1}\be_{1}}^{\fs,\fr2^{-\left(n-j\right)}}u(x)-\cM_{\alpha_{j+1}\be_{1}}^{\fs,\fr2^{-\left(n-j-1\right)}}u(x)\right|^{s}\right)\\
 & \leq(\fr2^{-n})^{s-d}2^{j(s-d)}\norm{\nablas u}_{L^{s}\of{\conv\left(\Ball{\fr2^{-\left(n-j\right)}}{\left\{ \alpha_{j}\be_{1},\alpha_{j+1}\be_{1}\right\} }\right)\cup\Ball{\fr2^{-\left(n-j-1\right)}}{\alpha_{j+1}\be_{1}}}}^{s}
\end{align*}
Defining 
\[
\tilde{f}^{s}:=\sum_{j}(\fr2^{-n})^{s-d}2^{j(s-d)}\chi_{\conv\left(\Ball{\fr2^{-\left(n-j\right)}}{\left\{ \alpha_{j}\be_{1},\alpha_{j+1}\be_{1}\right\} }\right)\cup\Ball{\fr2^{-\left(n-j-1\right)}}{\alpha_{j+1}\be_{1}}}
\]
and using local finiteness of the covering as well as 
\[
\left|\conv\left(\Ball{\fr2^{-\left(n-j\right)}}{\left\{ \alpha_{j}\be_{1},\alpha_{j+1}\be_{1}\right\} }\right)\cup\Ball{\fr2^{-\left(n-j-1\right)}}{\alpha_{j+1}\be_{1}}\right|\leq C\left(\fr^{d}2^{-d(n-j)}\right)\,,
\]
we find with $\frac{\left(s-d\right)p}{p-s}+d=\frac{s(p-d)}{p-s}-1=\frac{s(1-d)-p+s}{p-s}=\frac{2s-d-p}{p-s}$
$\delta$
\begin{align*}
\int_{\Ball 1{\left\{ y_{1},y_{2}\right\} }}\left|\tilde{f}\right|^{\frac{sp}{p-s}} & \leq C\sum_{j}(\fr2^{-n})^{\frac{\left(s-d\right)p}{p-s}}2^{j\frac{\left(s-d\right)p}{p-s}}\fr^{d}2^{-d(n-j)}\\
 & \leq C\delta^{\frac{s(p-d)}{p-s}}\sum_{j=1}^{\ln_{2}\frac{\fr}{\delta}}2^{j\frac{s(1-d)}{p-s}}\leq C\delta^{\frac{s(p-d)}{p-s}}\delta^{-1}
\end{align*}
From here we conclude the first part. Inequality (\ref{eq:lem:thm:poisson-final-help-2})
follows from the fact that $\tilde{\rho}_{i}\propto\fr_{i}$ and the
disjointness of the balls $\Ball{\frac{\fr}{16}}{x_{a}}$ with $\Ball{\fr_{i}}{p_{i}}$
and Lemma \ref{lem:general-Korn-Poincar=0000E9} with $\fr=const$.
\end{proof}

\begin{proof}[Proof of Theorem \ref{thm:boolean-final}]
We work with the enumeration $\left(p_{i}\right)_{i\in\N}$ and $\X_{\fr}=\left(x_{a}\right)_{a\in\N}$
and make use of the underlying point process $\X_{\pois}$: For every
$a\in\N$ there exists $y_{x_{a}}\in\X_{\pois}$ such that $x_{a}\in\Ball 1{y_{x_{a}}}$
for every $p_{i}$ there almost surely exists a unique $y_{p_{i}}\in\X_{\pois}$
such that $p_{i}\in\Ball 1{y_{p_{i}}}$. Due to the minimal mutual
distance of points in $\X_{\fr}$, we can conclude the following:
Since $p_{i}\in\Ball{\fr}{G_{a}}$, $\Ball{\fr}{x_{a}}\subset\tilde{P}\cap G_{a}$
there exists a constant $C$ depending only on $\fr$ and $d$ such
that always 
\begin{equation}
\left|y_{p_{i}}-y_{x_{a}}\right|\leq C\,d_{a}\,.\label{eq:thm:boolean-final-help-1}
\end{equation}
Since 
\[
\left|\tau_{i}^{\fs}u-\cM_{a}^{\fs}u\right|^{s}\leq3\left(\left|\tau_{i}^{\fs}u-\cM_{y_{p_{i}}}^{\fs}u\right|^{s}+\left|\cM_{y_{x_{a}}}^{\fs}u-\cM_{a}^{\fs}u\right|^{s}+\left|\cM_{y_{x_{a}}}^{\fs}u-\cM_{y_{p_{i}}}^{\fs}u\right|^{s}\right)
\]
we find
\[
\frac{1}{\left|\bQ\right|}\int_{\bP\cap\bQ_{\fr}}\sum_{\substack{i\neq0}
}\sum_{x_{a}\in\X(\bQ)}f_{a}\frac{\left|\partial_{l}\phi_{i}\right|}{D_{l+}}\left|\tau_{i}^{\fs}u-\cM_{a}^{\fs}u\right|^{s}\leq I_{1}+I_{2}+I_{3}
\]
 where we provide and estimate $I_{1}$, $I_{2}$ and $I_{3}$ as
follows: First, we observe there exists $n_{0}$ such that $n_{0}\fr>1$.
Then with help of (\ref{eq:lem:thm:poisson-final-help-2})
\begin{align*}
I_{1} & :=\frac{1}{\left|\bQ\right|}\int_{\bP\cap\bQ_{\fr}}\sum_{\substack{i\neq0}
}\sum_{x_{a}\in\X(\bQ)}f_{a}\frac{\left|\partial_{l}\phi_{i}\right|}{D_{l+}}\left|\tau_{i}^{\fs}u-\cM_{y_{p_{i}}}^{\fs}u\right|^{s}\\
 & \leq\frac{1}{\left|\bQ\right|}\int_{\bP\cap\bQ_{\fr}}\sum_{x_{a}\in\X(\bQ)}f_{a}\sum_{y_{b}\in\X_{\pois,\infty}}\sum_{p_{i}\in\partial\Ball 1{y_{b}}}\frac{\left|\partial_{l}\phi_{i}\right|}{D_{l+}}\left|\tau_{i}^{\fs}u-\cM_{y_{b}}^{\fs}u\right|^{s}\\
 & \leq\frac{1}{\left|\bQ\right|}\int_{\bP\cap\bQ_{\fr}}\sum_{x_{a}\in\X(\bQ)}\chi_{\Ball{2d_{a}}{x_{a}}}\sum_{y_{b}\in\X_{\pois,\infty}}\sum_{p_{i}\in\partial\Ball 1{y_{b}}}\chi_{\Ball{\tilde{\rho}_{i}}{p_{i}}}\left|\tau_{i}^{\fs}u-\cM_{y_{b}}^{\fs}u\right|^{s}\\
 & \leq C\left(\frac{1}{\left|\bQ\right|}\int_{\bP\cap\bQ_{\fr}}\left(\sum_{x_{a}\in\X(\bQ)}\chi_{\Ball{2d_{a}}{x_{a}}}\sum_{y_{b}\in\X_{\pois,\infty}}\chi_{\Ball 1{y_{b}}}\right)^{\frac{p}{p-s}}\right)^{\frac{p-s}{p}}\left(\frac{1}{\left|\bQ\right|}\int_{\bP\cap\bQ_{\fr}}\left|\nabla u\right|^{p}\right)^{\frac{s}{p}}\,.
\end{align*}
Because of Lemmas \ref{lem:estim-E-fa-fb} and \ref{lem:X-pois-covery-finite}
and the exponential decay of probabilities of $d_{a}$ the first integral
on the right hand side is always bounded. Note that (\ref{eq:lem:thm:poisson-final-help-2})
also implies 
\begin{align*}
I_{2} & :=\frac{1}{\left|\bQ\right|}\int_{\bP\cap\bQ_{\fr}}\sum_{\substack{i\neq0}
}\sum_{x_{a}\in\X(\bQ)}f_{a}\frac{\left|\partial_{l}\phi_{i}\right|}{D_{l+}}\left|\cM_{y_{x_{a}}}^{\fs}u-\cM_{a}^{\fs}u\right|^{s}\\
 & \leq\frac{1}{\left|\bQ\right|}\int_{\bP\cap\bQ_{\fr}}\sum_{x_{a}\in\X(\bQ)}f_{a}\,d_{a}^{d}\,\left|\cM_{y_{x_{a}}}^{\fs}u-\cM_{a}^{\fs}u\right|^{s}\\
 & \leq\left(\frac{1}{\left|\bQ\right|}\int_{\bP\cap\bQ_{\fr}}\left(\sum_{y_{b}\in\X_{\pois,\infty}}\sum_{x_{a}\in\X(\bQ)\cap\Ball 1{y_{b}}}d_{a}^{2d}\right)^{\frac{p}{p-s}}\right)^{\frac{p-s}{p}}\left(\frac{1}{\left|\bQ\right|}\int_{\bP\cap\bQ_{1}}\left|\nabla u\right|^{p}\right)^{\frac{s}{p}}\,.
\end{align*}
Again, the first integral on the right hand side is bounded. 

Last, the term 
\[
I_{3}:=\int_{\bP\cap\bQ_{\fr}}\sum_{\substack{i\neq0}
}\sum_{x_{a}\in\X(\bQ)}f_{a}\frac{\left|\partial_{l}\phi_{i}\right|}{D_{l+}}\left|\cM_{y_{x_{a}}}^{\fs}u-\cM_{y_{p_{i}}}^{\fs}u\right|^{s}
\]
is the most tricky part. 

We find a path $Y(y_{x_{a}},y_{p_{i}})=(y_{1},\dots,y_{n(x_{a},p_{i})})$
such that $y_{1}=y_{x_{a}}$, $y_{n(x_{a},p_{i})}=y_{p_{i}}$ such
that $y_{j}$, $y_{j+1}$ are neighbors. By our assumptions, for every
two points $y,\tilde{y}\in\X_{\pois,\infty}$ with $y-\tilde{y}<2\fr$,
the convex hull of $\Ball{\fr}{\left\{ y,\tilde{y}\right\} }$ lies
in $\bP_{\infty}$. Hence we iteratively replace sequences $(\dots y_{i},y_{i+1},y_{i+2},\dots)$
in the path $Y$ by $(\dots y_{i},y_{i+2},\dots)$ if $\left|y_{i+2}-y_{i}\right|<2\fr$.
Hence, w.l.o.g we obtain from (\ref{eq:thm:boolean-final-help-1})
and the definition of the statistical strech factor 
\[
n(x_{a},p_{i})\leq2\frac{\Length Y}{\fr}\leq2\fr^{-1}Cd_{a}S(y_{x_{a}})\,.
\]
Therefore, for $y\in\X_{\pois,\infty}$ with $\chi_{\Ball 1y}\chi_{G_{a}}\neq0$
we observe and the shortest path $Y(x_{a},y_{p_{i}})$ and with Lemma
\ref{lem:thm:poisson-final-help}
\begin{align*}
\left|\cM_{y_{x_{a}}}^{\fs}u-\cM_{y_{p_{i}}}^{\fs}u\right|^{s} & \leq\left(2\fr^{-1}Cd_{a}S(y_{x_{a}})\right)^{s}\sum_{k=1}^{n(x_{a},y_{p_{i}})-1}\left|\cM_{y_{k}}^{\fs}u-\cM_{y_{k+1}}^{\fs}u\right|^{s}\\
 & \leq\left(2\fr^{-1}Cd_{a}S(y_{x_{a}})\right)^{s}\sum_{k=1}^{n(x_{a},y_{p_{i}})-1}\norm{f\nablas u}_{L_{\Ball 1{\left\{ y_{k},y_{k+1}\right\} }}^{s}}^{s}\,.
\end{align*}
Now all points $y_{i}\in Y(x_{a},y_{p_{i}})$ lie within a radius
of $2Cd_{a}S(y_{x_{a}})$ around $x_{a}$, wich implies 
\begin{align*}
I_{3} & \leq\int_{\bP\cap\bQ_{\fr}}\sum_{x_{a}\in\X(\bQ)}\chi_{\Ball{2Cd_{a}S(y_{x_{a}})}{x_{a}}}d_{a}^{d}f^{s}\left|\nablas u\right|^{s}\\
 & \leq\left(\frac{1}{\left|\bQ\right|}\int_{\bP\cap\bQ_{\fr}}\left(\sum_{x_{a}\in\X(\bQ)}\chi_{\Ball{2Cd_{a}S(y_{x_{a}})}{x_{a}}}d_{a}^{d}f^{s}\right)^{\frac{p}{p-s}}\right)^{\frac{p-s}{p}}\left(\frac{1}{\left|\bQ\right|}\int_{\bP\cap\bQ_{1}}\left|\nabla u\right|^{p}\right)^{\frac{s}{p}}\,.
\end{align*}
Now, by independence of the respective variables, the constant in
front converges to 
\[
\left(\E\left(\sum_{x_{a}\in\X(\bQ)}\chi_{\Ball{2Cd_{a}S(y_{x_{a}})}{x_{a}}}d_{a}^{d}\right)^{\frac{p}{p-s}}\,\E f^{\frac{ps}{p-s}}\right)^{\frac{p-s}{p}}\,.
\]
The first term in the product can be estimated with help of Lemma
\ref{lem:estim-E-fa-fb} and is bounded for every $p$ and $s$ by
the exponential distribution of $d_{a}$ and $S$. The second term
can be estimated similarly.
\end{proof}
A further important property which we will not use in this work,
but which is central for part III of this series is the following
result.
\begin{thm}
\label{thm:overlay-count-2} Let $\X_{\pois,\infty,\fr}:=\left\{ x\in\X_{\pois,\infty}:\,\,\forall y\in\X_{\pois,\infty}\backslash\{x\}\,\left|x-y\right|>\frac{\fr}{8}\right\} $
be a Matern reduction of the infinite component. For fixed $y_{0}\in\X_{\pois,\infty,\fr}$
and every $\tilde{y}\in\X_{\pois,\infty,\fr}$ let $P(y_{0},\tilde{y})=\left(y_{0},y_{1}(\tilde{y}),\dots,y_{N}(\tilde{y})\right)_{N\in\N}$
with $y_{N}(\tilde{y})=\tilde{y}$ be the shortest path of points
in $\X_{\pois,\infty,\fr}$ connecting $y_{0}$ and $\tilde{y}$ in
$\bP$ and having length $L(y_{0},\tilde{y})$. Then there exists
\begin{align*}
\gamma_{y_{0,}\tilde{y}}:\,[0,L(y_{0},\tilde{y})]\times\Ball{\frac{\fr}{16}}0 & \to\bP\\
(t,z) & \mapsto\gamma_{y_{0,}\tilde{y}}(t,z)
\end{align*}
such that $\gamma_{y_{0,}\tilde{y}}(t,\cdot)$ is invertible for every
$t$ and $\norm{\partial_{t}\gamma_{y_{0,}\tilde{y}}}_{\infty}\leq2$.
For $R>1$ let 
\[
N_{y_{0},R}(x):=\#\left\{ \tilde{y}\in\Ball R{y_{0}}\cap\X_{\mat}:\,\exists t:\,x\in\gamma_{y_{0},\tilde{y}}\of{t,\Ball{\frac{\fr}{16}}0}\right\} \,.
\]
Then for every $y_{0}$ there exists almost surely $C>0$, $S>0$
such that it holds
\[
N_{y_{0},R}(x)\leq C\left(R^{d}-\left(\frac{x}{2}\right)^{d}\right)\quad\text{for }\left|x-y_{0}\right|<SR\,,\qquad N_{y_{0},R}(x)=0\quad\text{else.}
\]
\end{thm}

\begin{proof}
The function $\gamma_{y_{0,}\tilde{y}}$ consists basically of pipes
connecting $y_{i}(\tilde{y})$ with $y_{i+1}(\tilde{y})$ that conically
become smaller within the ball $\Ball{\frac{1}{2}}{y_{i}(\tilde{y})}$
to fit through the connection between two neighboring balls. Defining
\[
N_{y_{0},r,R}(x):=\#\left\{ \tilde{y}\in\of{\Ball R{y_{0}}\backslash\Ball r{y_{0}}}\cap\X_{\pois,\infty,\fr}:\,\exists t:\,x\in\gamma_{y_{0},\tilde{y}}\of{t,\Ball{\frac{\fr}{16}}0}\right\} 
\]
we apply Lemma \ref{lem:strech-factor-boolean} instead of \cite{xia2013stretch}
implies $N_{y_{0},r,R}(x)=0$ for all $\left|x-y_{0}\right|>SR$ but
also due to the minimal mutual distance $N_{y_{0},r,R}(x)\leq CR^{d-1}(SR-r)$,
where $C$ depends only on $\fr$ and $d$. From here we follow the
proof of Theorem \ref{thm:overlay-count-1}.
\end{proof}
We close this section and this work by proving Theorem \ref{thm:main-Boolean}.
\begin{proof}[\textbf{Proof of Theorem \ref{thm:main-Boolean}}]
The statement on the support is provided by Theorem \ref{thm:support}
and the fact that we restrict to functions with support in $m\bQ$.
Hence in the following we can apply all cited results to $\Ball{m^{1-\beta}}{m\bQ}$
instead of $m\bQ$. According to Lemmas \ref{lem:local-delta-M-extension-estimate}
and \ref{lem:6-4}--\ref{lem:6-5} and to Theorem \ref{thm:boolean-final}
we need only need to ensure $p>d$ as well as 
\[
\E\left(1+M_{[\frac{1}{8}\delta],\Rd}\right)^{kr}+\E\left|\tilde{\rho}_{\Rd}\right|^{-\frac{sr}{s-r}}<\infty\,,
\]
where $k=1$ for the simple extension case and $k=2$ for the symmetric
extension case. Since $d_{a}$ is distributed exponentially and the
corresponding terms are bounded as long as $r\neq s\neq p$, we observe
that we do not have to care about the involved polynomial terms. 

According to Theorem \ref{thm:boolean-delta-M-distrib} it is sufficient
that $\frac{sr}{s-r}<2$ (i.e. $\frac{pr}{p-r}<2$) and $kr<d+2$. 
\end{proof}

\addcontentsline{toc}{section}{Nomenclature}\settowidth{\nomlabelwidth}{$\closedsets_{V}$, $\closedsets^{K}$, $(\closedsets(\Rd),\ttopology F)$}
\printnomenclature{}

\end{document}